\documentclass[mn,a4paper,fleqn]{wm-art}
\usepackage{times,cite,w-thm}
\usepackage{amsmath,amssymb,calc,graphicx}

\numberwithin{equation}{section}

\newcommand{\itemseparate}{\itemsep0pt}
\newcommand{\refpart}[1]{{\it (#1)}}

\newcommand{\heunde}[1]{{\textstyle\left(#1\right)}}
\newcommand{\hpgo}[2]{{}_{#1}\mbox{\rm F}_{\!#2}}
\newcommand{\hpg}[5]{{}_{#1}\mbox{\rm F}_{\!#2}\! \left(\left.{#3 \atop #4}\right|\, #5 \right) }
\newcommand{\heun}[5]{\mbox{\rm Hn}\!\left( {#1 \atop #2} \left|  {#3   \atop #4} \right|\, #5 \right) }

\newcommand{\AJ}[1]{\mbox{\rm #1}}
\newcommand{\CPR}[2]{\,[#2]}

\newcommand{\swq}[2]{#2 &\em #1}

\newcommand{\led}[1]{#1}

\newcommand{\jsp}{\!\!\!}
\newcommand{\fr}[2]{\stackrel{#1}{}\hspace{-3.5pt}/#2}
\newcommand{\frc}[2]{\hspace{-0.5pt}\stackrel{#1}{}\hspace{-4pt}/\mbox{\scriptsize $\!#2$}}

\newcommand{\cdt}{\!\cdot\!}
\newcommand{\xsix}[1]{\hspace{-3pt}\footnotesize$\frac{#1:}{}\!$}

\newcommand{\ZZ}{{\Bbb Z}}
\newcommand{\NN}{{\Bbb N}}
\newcommand{\QQ}{{\Bbb Q}}
\newcommand{\CC}{{\Bbb C}}
\newcommand{\PP}{{\Bbb P}}
\newcommand{\RR}{{\Bbb R}}
\newcommand{\FF}{{\Bbb F}}
\newcommand{\OO}{{\cal O}}

\begin{document}

\DOIsuffix{theDOIsuffix}
\Volume{248}
\Month{01}
\Year{2007}
\pagespan{1}{}

\title[Belyi functions for hyperbolic Heun functions]
{Belyi functions for hyperbolic hypergeometric-to-Heun\\ transformations}

\author[M. van Hoeij]{Mark van Hoeij\inst{1,}
\footnote{E-mail: {\sf hoeij@math.fsu.edu}. Supported by NSF grant 1017880.}}
\address[\inst{1}]{Department of Mathematics, 
Florida State University, Tallahassee, Florida 32306, USA.}
        
\author[R. Vid\=unas]{Raimundas Vidunas\inst{2,}
\footnote{E-mail: {\sf vidunas@math.kobe-u.ac.jp}. Supported by JSPS grant 20740075.}}
\address[\inst{2}]{Lab of Geometric \& Algebraic Algorithms,
Department of Informatics \& Telecommunications,
National Kapodistrian University of Athens,
Panepistimiopolis 15784, Greece . }

\begin{abstract}
A complete classification of Belyi functions for
transforming certain hypergeometric equations to Heun equations is given. 
The considered hypergeometric equations have the local exponent differences $1/k,1/\ell,1/m$ 
that satisfy $k,\ell,m\in\NN$ and the hyperbolic condition $1/k+1/\ell+1/m<1$.
There are 366 Galois orbits of Belyi functions
giving the considered (non-parametric) hypergeometric-to-Heun pull-back transformations.
Their maximal degree is 60, which is well beyond reach of
standard computational methods. To obtain these Belyi functions, we developed two
efficient algorithms that exploit the implied pull-back transformations. 
\end{abstract}

\maketitle

\section{Introduction}
\label{sec:intro}`

Belyi functions and {\em dessins d'enfants} \cite{Wikipedia} is a captivating field of research
in algebraic geometry, complex analysis, Galois theory and related fields.
However, computation of Belyi functions of degree over 20 is still considered hard 
\cite[Example 2.4.10]{LandoZvonkin} 
even for genus 0 Belyi coverings $\PP^1\to\PP^1$.
This computational difficulty promises a long lasting appeal,
both for constructivists and theoreticians.
Grothendieck \cite[pg.~248]{GrothESQ} doubted that
{\em ``there is a uniform method for solving the problem by computer''}.
The subject of this paper is effective computation of certain Belyi functions $\PP^1\to\PP^1$,
of degree up to 60, utilizing the fact that those functions transform Fuchsian differential
equations with a small number of singularities.

This paper considers  {\em genus $0$ Belyi functions}, 
that is, rational functions $\varphi:\PP_x^1\to\PP_z^1$ that branches only in the 3 fibers $z=\varphi(x)\in\{0,1,\infty\}$.  
We distinguish the two projective lines by the indices $x,z$ just as in \cite{HeunClass}.
To describe the Belyi functions we classify, we introduce the following definitions.


\begin{definition} \label{df:klmregular} 
Given positive integers $k,\ell,m$, a Belyi function $\varphi:\PP_x^1\to\PP_z^1$ 
is called {\em $(k,\ell,m)$-regular} if all points above $z=1$ have branching order $k$,
all points above $z=0$ have branching order $\ell$, and all points above $z=\infty$
have branching order $m$. 
\end{definition}

Examples of $(2,3,m)$-regular Belyi functions with $m\in\{3,4,5\}$
are the well-known Galois coverings $\PP^1\to\PP^1$ of degree 12, 24, 60
with the tetrahedral $A_4$, octahedral $S_4$ or icosahedral $A_5$ monodromy groups,
respectively. The Platonic solids give their dessins d'enfant \cite{platonic}.

\begin{definition} \label{df:klmnregular}
Given yet another positive integer $n$, a
 Belyi function $\varphi:\PP_x^1\to\PP_z^1$ 
is called {\em $(k,\ell,m)$-minus-$n$-regular} 
if, with exactly $n$ exceptions, all points above $z=1$ have branching order $k$,
all points above $z=0$ have branching order $\ell$, and all points above $z=\infty$
have branching order $m$. We will also use the shorter term  {\em $(k,\ell,m)$-minus-$n$}.
\end{definition}

Examples of $(k,\ell,m)$-minus-$2$ functions 
are quotients of the just mentioned Galois coverings by a cyclic monodromy group. 
If $1/k+1/\ell+1/m>1$ and $n \geq 3$, there are $(k,\ell,m)$-minus-$n$
Belyi functions 
of arbitrary high degree. They give Kleinian pull-back transformations \cite{klein78,kleinvhw} 
to second order Fuchsian equations with finite monodromy (i.e., a basis of algebraic solutions)
from a few standard hypergeometric equations. 
An example of a $(2,3,5)$-minus-3 Belyi function of degree 1001 is given online
at \cite[\sf NamingConvention]{HeunURL}. 
As Remark \ref{rm:noexist} here shows, $(k,\ell,m)$-minus-1-regular Belyi functions 
exist only if $1\in\{k,\ell,m\}$.

\begin{definition} \label{df:mnhyperbolic}
A Belyi function $\varphi$ 
is called {\em minus-$n$-hyperbolic} if:
\begin{enumerate} \itemseparate
\item \label{df:ConditionH1} there are positive integers $k,\ell,m$ satisfying $1/k+1/\ell+1/m<1$ (the {\em hyperbolic condition})
such that $\varphi$ is $(k,\ell,m)$-minus-$n$-regular;
\item \label{df:ConditionH2} there is at least one point of branching order $k$ above $z=1$, 
a point of order $\ell$ above $z=0$, and a point of order $m$ above $z=\infty$.
\end{enumerate}
\end{definition}

Minus-3-hyperbolic Belyi functions are listed in \cite[\S9]{VidunasFE}.
Table~3 in \cite{VidunasFE} lists nine\footnote{Minus-3-hyperbolic Belyi functions 
give rise to the hypergeometric transformations described in \cite[\S9]{VidunasFE}.
There are 10 different such Belyi functions up to M\"obius transformations, in 9 Galois orbits.
The degree 18 Belyi function there is defined over $\QQ(\sqrt{-7})$. }
Galois orbits of such Belyi functions, of degree up to 24. 

This paper gives all {\em minus-$4$-hyperbolic Belyi functions} $\PP^1\to\PP^1$.
The motivation is that they give
transformations of Gauss hypergeometric differential equations without 
Liouvillian \cite{Wikipedia}
solutions to Heun equations (i.e., Fuchsian equations with 4 singularities). This allows to express 
non-Liouvillian Heun functions in terms of better understood Gauss hypergeometric functions.
The application to these transformations of Fuchsian equations is discussed in \S \ref{sec:heun}. 
This paper, combined with the list of {\em parametric} hypergeometric-to-Heun transformations
in \cite{HeunForm}, covers all non-Liouvillian cases of hypergeometric-to-Heun  transformations.

We used two algorithms to compute the minus-$4$-hyperbolic Belyi functions.
They both utilize the fact that these Belyi functions give hypergeometric-to-Heun transformations.
One algorithm is probabilistic and uses modular lifting. It exploits the fact that Heun's equation is 
represented by few parameters. The other algorithm is deterministic, and uses existence of a
hypergeometric-to-Heun transformation to get more algebraic equations for the (a priori)
undetermined coefficients of a Belyi function.

The branching patterns are enumerated in \S \ref{sec:classify}, 
following the approach from \cite{HeunClass}.
We notice that some of our Belyi functions are related to 
notable Shimura curves  \cite{elkies237}, \cite{voightg2}.
The application to hypergeometric-to-Heun transformations is explained  in \S \ref{sec:heun}.
Our algorithms are presented in \S \ref{sec:algorithms}.
Section \ref{sec:conic} discusses special {\em obstructed} cases of encountered Belyi functions.
The Appendix sections give ordered lists A--J of computed Belyi functions,
discusses composite Belyi functions, and compares our results
with Felixon's list \cite{Felixon} of {\em Coxeter decompositions} in the hyperbolic plane.
All dessins d'enfant of computed Belyi coverings are depicted in this paper, 
most of them next to the A--J tables of \S \ref {sec:ajtables}.
Our list of dessins is larger than \cite{dcatalog,BeukersMontanus,Luttmer} combined.

%
%

\section{Organizing definitions, examples}
\label{sec:deeper}

We start with a few definitions that will help us to organize the list of Belyi functions.
Then we take a relaxed look at a few examples, including those of the largest degree 60.
At the same time, dessins d'enfant are introduced more fully, 
setting a geometric tone of our presentation.

\begin{definition}
Let $\varphi$ be a $(k,\ell,m)$-minus-$n$-regular 
Belyi function for some $n$. The {\em regular branchings} of $\varphi$
are the points above $z=1$ of order $k$, the points above $z=0$ of order $\ell$,
and the points above $z=\infty$ of order $m$. The other $n$ points in the three fibers 
are called {\em exceptional points} of $\varphi$. A {\em branching fraction} of $\varphi$
is a rational number $A/B$, where $A$ is a branching order at an exceptional point $Q$,
and $B\in\{k,\ell,m\}$ is the prescribed branching order for the fiber of $Q$.
\end{definition}

\begin{definition} \label{df:jtfield}
Let $\varphi:\PP_x^1\to\PP_z^1$ be a $(k,\ell,m)$-minus-$4$
Belyi function. Let $q_1,q_2,q_3,q_4$ denote its exceptional points. 
The {\em $j$-invariant} of $\varphi$ is the $j$-invariant
of the elliptic curve  \mbox{$Y^2=\prod_{q_i\neq \infty}(X-q_i)$}, 
given by formula~$(\ref{eq:jinv})$ below.
It is invariant under M\"obius transformations of $\PP_x^1$.

A {\em canonical form} of $\varphi$ is a composition of $\varphi$ with a M\"obius transformation
that has three exceptional points at $x=0, 1, \infty$. The fourth exceptional point 
then becomes $x=t$, where $t$ is a cross-ratio of $q_1,q_2,q_3,q_4$.
The cross-ratio depends on the order of $q_1,q_2,q_3,q_4$, and there is an $S_3$-orbit  ($S_3 \cong S_4/V_4$)
\begin{equation} \label{S3action} 
\left\{t,1-t,\frac{t}{t-1},\frac1t,\frac1{1-t},1-\frac1t \right\}
\end{equation}
of related cross-ratios. Any of these values is a {\em $t$-value} of $\varphi$.
The $j$-invariant is
\begin{equation} \label{eq:jinv}
j(t)=\frac{256\,(t^2-t+1)^3}{t^2 (t-1)^2}.
\end{equation}
\end{definition}
As an example, $t\in\{-1,2,\frac12\}$ gives $j=1728$. If $j \not\in \{0, 1728\}$ then the
six $t$-values in above $S_3$-orbit are distinct. The $S_3$-action gives a homomorphism
$S_3 \rightarrow {\rm Gal}(\QQ(t)/\QQ(j))$, hence $[\QQ(t):\QQ(j)] \in \{1,2,6\}$.

\begin{definition} \label{df:jtfield2}
The {\em $t$-field} resp. {\em $j$-field} of $\varphi$ is the number field
generated by a $t$-value resp. the $j$-invariant.
The $r$-field ({\em canonical realization field}) of $\varphi$ is the
smallest field over which a canonical form of $\varphi$ is defined.
These fields do not depend on the ordering of the $4$ exceptional points.
\end{definition}

\begin{example}   \label{ExampleB7b}
The degree 12 rational function 
\[ \varphi(x) = \frac{64x^2 (x-3)^9 (x-9)}{27 (x-1) (8x^3-72x^2-27x+27)^3} \]
is a $(2,3,9)$-minus-$4$ Belyi map. Indeed, the numerator of $\varphi(x)-1$
is a full square. 
It is already in a canonical form, as $x=0$, $x=1$ and $x=\infty$ (of branching order 2)
are exceptional points. The fourth exceptional point $x=9$ is a $t$-value. The $j$-invariant is
equal to $2^273^3/3^4$ by formula $(\ref{eq:jinv})$. 

The {\em branching pattern} of $\varphi$ is given by three partitions of the degree $d=12$
into branching orders above $0,1,\infty$. Using the notation in \cite{HeunClass}, we express 
the {\em branching pattern} of $\varphi$ shortly as follows:
\[
 6 \, [2] = 3 \, [3] +2+1 = [9]+2+1.
\]
The prescribed branching orders are indicated with square brackets, with their multiplicity in front.
The 4 branching orders that are not enclosed in square brackets
represent the 4 exceptional points. Dividing them by their prescribed branching order(s)
produces the 4 branching fractions: $1/3,2/3,1/9,2/9$.
\end{example}

In the application 
setting of hypergeometric-to-Heun transformations in \S \ref{sec:heun},
the regular branchings will become {\em regular points} (after a proper projective normalization)
of the pulled-back Heun equation $H$; the exceptional points will be the {\em singularities} of $H$;
and the branching fractions will be the {\em exponent differences} of $H$. 
The exponent differences of the hypergeometric equation under transformation
will be $1/k,1/\ell,1/m$. Example \ref{ExampleB7b} will be continued in \S \ref{sec:heun}.


%

Definitions \ref{df:jtfield}, \ref{df:jtfield2} will be used to  group the obtained Belyi functions 
into manageable classes.
%
%
The Belyi functions will be listed twice in this paper. The first list is Tables 2.3.7--3.4.4 
of \S \ref{sec:classify}. Its ordering by the $(k,\ell,m)$-triples and branching patterns
reflects the classification scheme. In Appendix \S \ref{sec:ajtables},  
the list of Galois orbits is grouped and ordered by the $j$-fields, $t$-fields, branching fractions.
%
%
This order allows quick recognition whether a given Heun function is reducible to 
a hypergeometric function with a rational argument $\varphi$.

Belyi functions nicely correspond to certain graphs 
called {\em dessins d'enfant}\footnote{Generally, a {\em dessin d'enfant} \cite{Wikipedia} is
a bi-colored graph  (possibly with multiple edges), with a cyclic order of edges around each vertex given.
This defines a unique (up to homotopy) embedding of the bi-colored graph into a Riemann surface.
Customarily, the vertex colors are black and white.
The dessins d'enfant of genus 0 Belyi coverings can be drawn on a plane,
as the Riemann sphere minus a point is homeomorphic to a plane.
Given a Belyi covering $\varphi$, its dessin d'enfant is 
realized as the pre-image of the interval segment $[0,1]\subset\RR\subset\CC$ 
onto its Rieman surface, with the vertices above $z=0$ colored black
and  the vertices $z=1$ colored white.  
The branching pattern of $\varphi$ determines the degrees (i.e., valencies) of vertices of both colors
of its dessin d'enfant, and the degrees of cells on the Riemann surface.
The cell degree is determined by counting vertices of one color while tracing its boundary.
The degree of a dessin d'enfant is the degree of the corresponding Belyi function.}.
Mimicking \cite[Section~2]{BeukersMontanus}, we spell out standard correspondences 
for genus 0 Belyi functions 
as follows. There are 1-1 correspondences between these objects:
\begin{enumerate} \itemseparate
\item[(I)] Belyi functions $\PP_z^1\to\PP_x^1$, 
up to M\"obius transformations \mbox{$x\mapsto (ax+b)/(cx+d)$}.
\item[(II)] Plane dessins d'enfant, 
up to homotopy on the Riemann sphere.
\item[(III)] The triples $(g_0,g_1,g_\infty)$ of elements in a symmetric group $S_d$, 
such that:
\begin{itemize} \itemseparate
\item $g_0g_1g_{\infty}=1$;
\item the total number of cycles in $g_0,g_1,g_\infty$ is equal to $d+2$;
\item $g_0,g_1,g_\infty$ generate a transitive action on a set of $d$ elements;
\end{itemize}
up to simultaneous conjugacy of $g_0,g_1,g_\infty$ in $S_d$.
\item[(IV)] Field extensions of $\overline{\QQ}(z)$ of 
genus 0,  unramified outside \mbox{$z=0,1,\infty$}.
\end{enumerate}
Part (III) gives the monodromy presentation of a Belyi covering, 
and $d$ is the degree. The dessins d'enfant is basically a graphical representation
of the combinatorial data in (III). 
This paper presents all obtained dessins pictorially, 
while the accompanying website \cite{HeunURL} 
gives  the Belyi maps (I), the permutations in (III) and other data (such as $j,t,r$-fields). 
The number $d+2$ is illuminated in the proof of Lemma \ref{th:klm}.
For each fiber $z\in\{0,1,\infty\}$, the conjugacy class of $g_z$ 
is determined by the partition of $d$ that reflects the branching pattern in the fiber.
Part (IV) is convenient for considering the composition structure of Belyi maps;
see Appendix \ref{sec:compose}.

The considered Belyi functions have rather regular dessins d'enfant. 
Definitions \ref{df:klmregular}--\ref{df:mnhyperbolic} are easy to reformulate for 
dessins d'enfant;
\begin{definition} \label{df:ddregular}
A dessin d'enfant is called  {\em $(k,\ell,m)$-minus-$n$-regular} if, 
with exactly $n$ exceptions,  all white vertices have degree $k$, 
all black vertices have degree $\ell$, and all cells have degree $m$.
\end{definition}
\begin{definition}
A dessin d'enfant $\Gamma$ is called {\em minus-$n$-hyperbolic} if:
\begin{enumerate}  \itemseparate
\item there are positive integers $k,\ell,m$ satisfying $1/k+1/\ell+1/m<1$ 
such that $\Gamma$ is $(k,\ell,m)$-minus-$n$-regular;
\item there is at least one white vertex of degree $k$, a black vertex of degree $\ell$,
and a cell of degree $m$.
\end{enumerate}
\end{definition}

All minus-4-hyperbolic  dessins d'enfants could be found by a combinatorial
search on a computer. But with our {\sf Maple} implementations it was faster 
to compute first the minus-4-hyperbolic Belyi functions,
and then compute their monodromy permutations in (III). 
This paper presents all minus-4-hyperbolic dessins
(up to complex conjugation), most of them next to the tables of Appendix \S \ref{sec:ajtables}.

In total, there are 872 Belyi functions of the minus-4-hyperbolic type,
up to M\"obius transformations in both $x$ and $z$.
They come in  366 
Galois orbits\footnote{Belyi functions are explicitly defined
over algebraic number fields, and the absolute Galois group Gal$(\overline{\QQ}/\QQ)$
permutes Belyi coverings with the same branching pattern. 
The size of a Galois orbit of dessins d'enfant is the degree of the moduli field; see \S \ref{sec:conic}.
Given a branching pattern, the set of Belyi coverings
with that branching pattern is finite (up to M\"obius transformations), possibly empty. 
The Galois action does not need to be transitive on this set, 
and several Galois orbits with the same branching pattern may appear.}.
In leap years we could decorate a calendar with the minus-4-hyperbolic dessins d'enfant,
one Galois orbit per day. We categorize and label the Galois orbits 
of the objects in (I)--(IV) as A1--J28; see \S \ref{sec:ajnumbers} and Appendix \S \ref{sec:ajtables}. 
The largest Galois orbit J28 has 15 dessins, for a $(2,3,7)$-minus-4 branching pattern of degree 37.
Completeness is checked with two independent algorithms
and other checks, see \S \ref{sec:algorithms} and Appendix \S \ref{sec:coxeter}.


\begin{vchfigure}
\vspace{-20pt}
\begin{picture}(375,232)(1,3)
\put(25,112){\includegraphics[height=108pt]{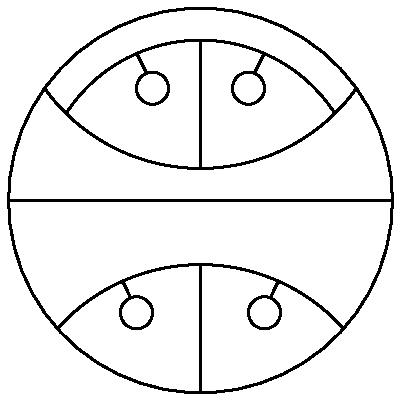}}
\put(140,112){\includegraphics[height=108pt]{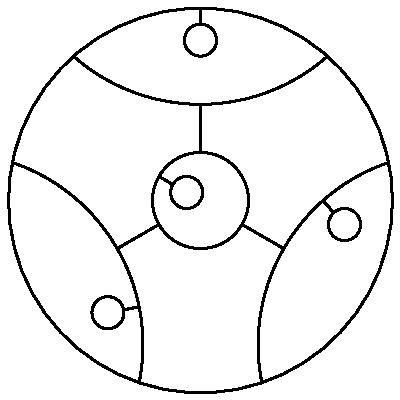}}
\put(255,112){\includegraphics[height=108pt]{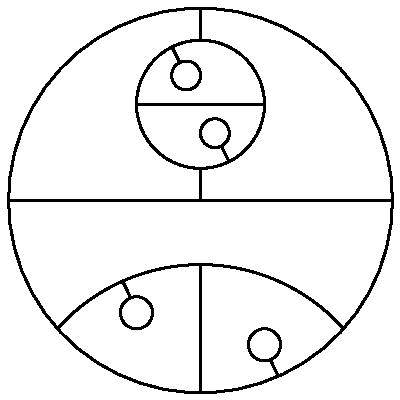}}
\put(25,0){\includegraphics[height=108pt]{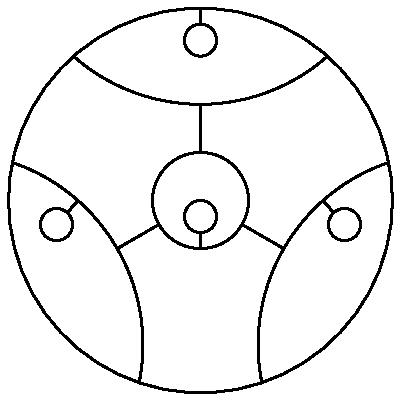}}
\put(140,0){\includegraphics[height=108pt]{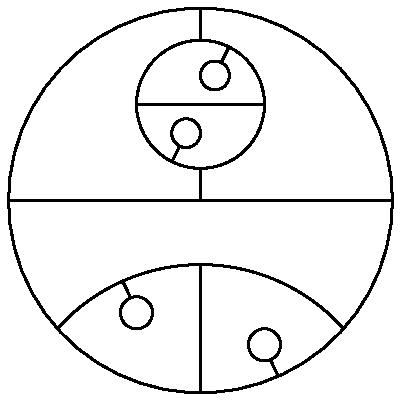}}
\put(255,0){\includegraphics[height=108pt]{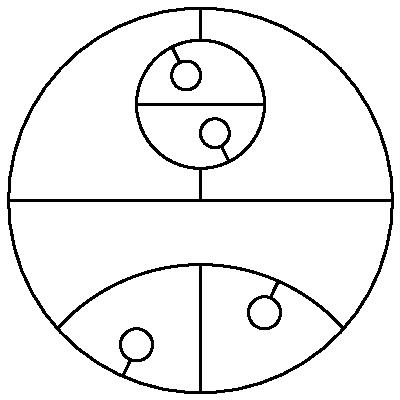}}
\put(0,166){H14} \put(0,51){H46} 
\end{picture}
\vchcaption{The degree 60 dessins d'enfant} \label{fg:deg60s}
\end{vchfigure}

The highest degree of a minus-4-hyperbolic Belyi function is 60.
Its branching pattern is $30\,[2]=20\,[3]=8\,[7]+1+1+1+1$.
There are two Galois orbits 
for this branching pattern, with three dessins each. 
We identify the two Galois orbits as H14 and H46.
The dessins d'enfant for these Belyi functions are depicted\footnote{The dessins in Figure \ref{fg:deg60s}
have all white vertices of order 2, hence they are examples of {\em clean} dessins d'enfant.
It is customary to depict clean dessins without white vertices, so that edges connect black vertices
directly, and loops are possible. A white vertex is then implied in the middle of each edge.}
in Figure \ref{fg:deg60s}. The 4 exceptional points in each dessin are represented by circular loops;
they could be assumed to lie in the center of each cell of degree 1.
The other cells (including the outer ones) have degree 7.
The left-most dessins of H14 and H46 clearly have a reflection symmetry,
hence they are defined over $\RR$. The other two dessins of H46 are mirror images
of each other, and are related by the complex conjugation.

The Belyi functions of degree 60 are composite.  
Their 15 and 30 components are labeled  H10 (for H14) and H46, J19 (for H45).
The Belyi functions H10, H14 are examples that have an {\em obstruction},
as  described in \S\ref{sec:conic}.
This has interesting geometric consequences for the dessins d'enfant. 
Although both have a totally real moduli field $\QQ(\cos\frac{2\pi}7)$, 
not all dessins of H10 and H14 have a reflection symmetry. Rather, 
the complex conjugation may give a homeomorphic dessin, 
identifiable with the original only after an automorphism of the Riemann sphere. For example, 
consider the middle and the right-most dessins of H14 in Figure \ref{fg:deg60s}. 
The dessins d'enfant for H10 are depicted in Figure \ref{fg:mdessins}, together with most of
other examples with an obstruction. 

\section{The branching patterns}
\label{sec:classify}

We enumerate the possible branching patterns in the same way as was done
for {\em parametric} hypergeometric-to-Heun transformations in \cite{HeunClass}.
To end up with a finite number of cases, we use Hurwitz formula and
the hyperbolic condition $1/k+1/\ell+1/m<1$.
Without loss of generality, we assume  the non-decreasing order $k\le \ell\le m$ 
for the regular branching orders from now on. 
\begin{lemma} \label{th:klm}
Let $\varphi$ be a minus-$4$-hyperbolic Belyi covering of degree $d$, 
with the regular branching orders $k\le\ell\le m\in\ZZ_{>0}$. Then
\begin{enumerate}  \itemseparate
\item There are exactly $d-2$ regular branchings and $4$ exceptional points.
\item $\displaystyle d-\left\lfloor \frac{d}k\right\rfloor-\left\lfloor \frac{d}\ell
\right\rfloor-\left\lfloor\frac{d}m\right\rfloor\le 2.$
\item 
Let $S$ denote the sum of $4$ branching fractions. Then
$\displaystyle d=\frac{2-S}{1-\frac1k-\frac1\ell-\frac1m}.$
\item $\displaystyle \left(1-\frac1k-\frac1{\ell}\right) m^2-3m+4\le 0.$
\item $\displaystyle \frac12\le\frac1k+\frac1\ell<1.$
\end{enumerate}
\end{lemma}
\begin{proof}
By Hurwitz formula (or \cite[Lemma 2.5]{VidunasFE}), 
there are $3d-(2d-2)=d+2$ distinct points above $\{0,1,\infty\}$ when $\varphi$ 
is a Belyi map $\PP^1\to\PP^1$. The first claim follows.
The number of regular branchings is at most 
$\left\lfloor d/k\right\rfloor+\left\lfloor d/\ell\right\rfloor+\left\lfloor d/m\right\rfloor$.
This gives the inequality in \refpart{ii}.
The number $d-2$ of regular branchings is also equal to 
$d/k+d/\ell+d/m-S$, giving the degree formula in \refpart{iii}.

We have $d\ge m$, otherwise condition \refpart{ii} of Definition \ref{df:mnhyperbolic} is not satisfied.
Combining this with the degree formula gives the inequality in \refpart{iv}. 
Together with $m\ge 4$, the inequality  in \refpart{iv} gives $1-1/k-1/\ell\le 1/2$. Part \refpart{v} follows.
\end{proof}

The inequalities in \refpart{iv}, \refpart{v} gives a finite list
of triples $(k,\ell,m)$.
Setting $S=4/m$ in part \refpart{iii} gives an upper bound for $d$, leaving the following
candidates for $(k,\ell,m,d)$:
\[
\begin{array}{cccc}
(2,3,7,\le60),  & (2,3,8,\le 36), & (2,3,9,\le 28),  & (2,3,10,\le 24),\\
(2,3,11,\le 21), & (2,3,12,\le 20), & (2,3,13,\le 18), & (2,3,14,\le 18),\\
(2,3,15,\le 17), & (2,3,16,\le 16), & (2,4,5,\le 24), & (2,4,6,\le 16), \\
(2,4,7,\le 13), & (2,4,8,\le 12), & (2,4,9,\le 11), & (2,4,10,\le 10), \\
(2,5,5,\le 12), & (2,5,6,\le 10), & (2,5,7,\le 9), & (2,5,8,\le 8), \\
(2,6,6,\le 8), & (2,6,7,\le 7), & (3,3,4,\le 12), & (3,3,5,\le 9),\\
(3,3,6,\le 8), & (3,3,7,\le 7), & (3,4,4,\le 6), & (3,4,5,\le 5),\\
(4,4,4,\le 4).
\end{array}
\]
The last two candidates give less than 4 exceptional points. 

Given a candidate tuple $(k, \ell, m, d)$, it is straightforward to enumerate the  correspondent
branching patterns. Let $h_1,h_0,h_{\infty}$ denote the eventual number of regular branchings
in the fibers $z=0,1,\infty$, respectively. Then we have
 \mbox{$h_1+h_0+h_{\infty}=d-2$}
by Lemma \ref{th:klm}\refpart{i}, and $0<h_1\le\lfloor d/k \rfloor$, etc. 
With a possible integer solution $(h_1,h_0,h_{\infty})$ at hand, we have to partition
the numbers $d-kh_1$, $d-\ell h_0$, $d-mh_\infty$ into total 4 positive parts,
not equal to the respective regular orders $k,\ell,m$. 
For example, if $(k,\ell,m,d)=(2,3,7,28)$ then we are lead to partition either 
7 into 3 parts (in the $m=7$ fiber) or 4 into 4 parts (in the $\ell=3$ fiber).
There are four such partitions of 7 and one of 4, hence five branching patterns.

In total, there are 378 branching patterns\footnote{One branching pattern $6\,[3]=9\,[2]=8+7+1+1+1$ is counted twice. It appears in Tables 2.3.7 and 2.3.8 because it is $(2,3,m)$-minus-4-regular
for two values of $m$ (=7 or 8).
It turns out, however, there are no Belyi functions with this branching pattern.} 
for minus-4-hyperbolic Belyi functions.
We list them in the first two columns of Tables 2.3.7--3.4.4,
by giving their branching fractions and the degree. 
The table titles refer to the tuple $(k,\ell,m)$.
The branching fractions are left unsimplified (e.g. $4/8$ instead of $1/2$)
to keep the fibers and branching orders of exceptional points plainly visible.  
The branching patterns are uniquely determined by the unsimplified branching fractions.

\subsection{Summary of computed results}
\label{sec:ajnumbers}

The last two columns of Tables 2.3.7--3.4.4 give information about found coverings with 
the considered branching patterns. 
Computation of Belyi coverings for each possible branching pattern is the most demanding step. 
The algorithms used to generate and verify the list of Belyi functions 
are presented in Section  \S \ref{sec:algorithms}.

The third column of Tables 2.3.7--3.4.4 gives a label for every Galois orbit with 
the branching pattern defined by a sequence of 4 branching fractions in the first column.
With the application to Heun equations in mind,  we group 
the Belyi functions by the $\QQ$-extension of the $j$-invariant. 
The cases with $j\in\QQ$ are further grouped by the $t$-field.
We group the computed minus-4-hyperbolic Belyi 
functions into the following 10 classes, labeled alphabetically  
from A to J:
\vspace{10pt} \\
\begin{tabular}{rp{10cm}}
A1--A24: & the Belyi functions with $j=1728$, that is $t\in\{-1,2,1/2\}$; \\
B1--B34: & the other Belyi functions with $t\in\QQ$; \\
C1--C42: & the Belyi functions with $j\in\QQ$ and a real quadratic $t$-field; \\
D1--D50: & the Belyi functions with $j\in\QQ$ and an imaginary quadratic $t$-field; \\ 
E1--E25: &  the Belyi functions with $j\in\QQ$, when the $t$-field  
is the splitting field (of degree 6 always)  of a cubic polynomial in $\QQ[z]$; \\
F1--F25: & the Belyi functions with a real quadratic $j$-field; \\ 
G1--G52: & the Belyi functions with an imaginary quadratic $j$-field; \\
H1--H53: & the Belyi functions with a cubic $j$-field; \\ 
I1--I33: & the Belyi functions with a $j$-field of degree 4 or 5;  \\
J1--J28: & the Belyi functions with a $j$-field of degree at least 6 (and  $\le15$).
\vspace{10pt} \\
\end{tabular}
In each class, the Galois orbits of Belyi functions 
are ordered by the criteria described in Appendix \S \ref{sc:criteria}.
A numbered label refers to a whole Galois orbit of
Belyi functions (or dessins d'enfant), 
as already mentioned in \S \ref{sec:deeper}. 
If there is more than one Galois orbit with the same branching pattern,
a line is devoted to each Galois orbit in Tables 2.3.7--3.4.4.


 
 \begin{table}
 \begin{minipage}{72mm}
\noindent \underline{\small \bf Table 2.3.7.  \hspace{156pt} \hfill \mbox{}  } \\[1pt] 
\begin{tabular}{@{}lrll@{}} 
$\fr{1}{7}$,$\fr{1}{7}$,$\fr{1}{7}$,$\fr{1}{7}$& 60 &H14&$j^3(-7,7)$, $t^{6}$  \\
   &                                           &H46&$j^3(-7,14)$, $t^{18}$  \\
$\fr{1}{7}$,$\fr{1}{7}$,$\fr{1}{7}$,$\fr{2}{7}$& 54 &D28& $\quad t(\sqrt{-3})$  \\
$\fr{1}{3}$,$\fr{1}{7}$,$\fr{1}{7}$,$\fr{1}{7}$& 52 &D35& $\quad t(\sqrt{-5})$  \\
$\fr{1}{7}$,$\fr{1}{7}$,$\fr{1}{7}$,$\fr{3}{7}$& 48&\hspace{2pt}--- &no covering  \\
$\fr{1}{7}$,$\fr{1}{7}$,$\fr{2}{7}$,$\fr{2}{7}$&    &A21& $\quad t = -1$  \\
   &                                           &I5& $j^4(-7,0,14)$, $t^{8}$  \\
$\fr{1}{3}$,$\fr{1}{7}$,$\fr{1}{7}$,$\fr{2}{7}$& 46 &J7&$j^{6}(2,4,8,4,4)$, $t^{12}$  \\
\end{tabular} 
\end{minipage} \hfil
\begin{minipage}{72mm}
\begin{tabular}{@{}lrll@{}} \\
$\fr{1}{3}$,$\fr{1}{3}$,$\fr{1}{3}$,$\fr{1}{3}$& 28 &G36&$j^2(\sqrt{-7})$, $t^{4}$  \\
   &                                           &H13&$j^3(-7,7)$, $t^{6}$  \\
$\fr{1}{3}$,$\fr{1}{7}$,$\fr{1}{7}$,$\fr{5}{7}$&    &I8&$j^4(0,3,3)$, $t^{8}$  \\
$\fr{1}{3}$,$\fr{1}{7}$,$\fr{2}{7}$,$\fr{4}{7}$&    &B16& $\quad t = 3^{3} / 2  $  \\
$\fr{1}{3}$,$\fr{1}{7}$,$\fr{3}{7}$,$\fr{3}{7}$&   &\hspace{2pt}--- &no covering  \\
$\fr{1}{3}$,$\fr{2}{7}$,$\fr{2}{7}$,$\fr{3}{7}$&   &\hspace{2pt}--- &no covering  \\
$\fr{1}{2}$,$\fr{1}{7}$,$\fr{1}{7}$,$\fr{4}{7}$& 27 &D46& $\quad t(\sqrt{-15})$  \\
\end{tabular} 
\end{minipage}
\end{table}
\begin{table}
\begin{minipage}{72mm}
\begin{tabular}{@{}lrll@{}} 
$\fr{1}{2}$,$\fr{1}{2}$,$\fr{1}{7}$,$\fr{3}{7}$& 18 &I27&$j^{5}(\mbox{\small $7,14,0,\!-49$}), t^{10}$  \\
$\fr{1}{2}$,$\fr{1}{2}$,$\fr{2}{7}$,$\fr{2}{7}$&    &A3& $\quad t = -1$  \\
   &                                           &H37&$j^3(9,2)$, $t^{6}$  \\
$\fr{1}{3}$,$\fr{1}{3}$,$\fr{1}{3}$,$\fr{4}{7}$&    &E6& $\quad t^{\rm spl}(-3,10)$  \\
$\fr{1}{3}$,$\fr{2}{3}$,$\fr{1}{7}$,$\fr{3}{7}$&    &I26&$j^{5}(-2,3,9,3)$  \\
$\fr{1}{3}$,$\fr{2}{3}$,$\fr{2}{7}$,$\fr{2}{7}$&   &\hspace{2pt}--- &no covering  \\
$\fr{1}{7}$,$\fr{1}{7}$,$\fr{1}{7}$,$\fr{8}{7}$&   &\hspace{2pt}--- &no covering  \\
$\fr{1}{7}$,$\fr{1}{7}$,$\fr{3}{7}$,$\fr{6}{7}$&    &C4& $\quad t(\sqrt{2})$  \\
\end{tabular} 
\end{minipage} \hfil
\begin{minipage}{72mm}
\begin{tabular}{@{}lrll@{}} 
$\fr{1}{3}$,$\fr{1}{3}$,$\fr{1}{3}$,$\fr{5}{7}$&  12  &E14& $\quad t^{\rm spl}(3,1)$  \\
$\fr{1}{3}$,$\fr{2}{3}$,$\fr{1}{7}$,$\fr{4}{7}$&    &B25& $\quad t = 3^{5}$  \\
$\fr{1}{3}$,$\fr{2}{3}$,$\fr{2}{7}$,$\fr{3}{7}$&    &B20& $\quad t = 2^{7} / 3  $  \\
$\fr{1}{2}$,$\fr{1}{3}$,$\fr{1}{3}$,$\fr{4}{7}$& 11 &H27&$j^3(-4,4)$, $t^{6}$  \\
$\fr{1}{2}$,$\fr{2}{3}$,$\fr{1}{7}$,$\fr{3}{7}$&    &H53&$j^3(-42,140)$  \\
$\fr{1}{2}$,$\fr{2}{3}$,$\fr{2}{7}$,$\fr{2}{7}$&    &C42& $\quad t(\sqrt{385})$  \\
$\fr{1}{2}$,$\fr{1}{2}$,$\fr{1}{3}$,$\fr{3}{7}$& 10 &H34&$j^3(-5,5)$, $t^{6}$  \\
$\fr{1}{3}$,$\fr{1}{3}$,$\fr{2}{3}$,$\fr{3}{7}$&    &D30& $\quad t(\sqrt{-5})$  \\
$\fr{2}{3}$,$\fr{2}{3}$,$\fr{1}{7}$,$\fr{2}{7}$&    &C40& $\quad t(\sqrt{105})$  \\
$\fr{4}{3}$,$\fr{1}{7}$,$\fr{1}{7}$,$\fr{1}{7}$&    &E20& $\quad t^{\rm spl}(21,14)$  \\
$\fr{1}{2}$,$\fr{1}{2}$,$\fr{1}{2}$,$\fr{2}{7}$&  9 &E21& $\quad t^{\rm spl}(9,2)$  \\
$\fr{1}{2}$,$\fr{1}{3}$,$\fr{2}{3}$,$\fr{2}{7}$&    &I15&  $j^4(-24,62,\!-48)$  \\
$\fr{1}{2}$,$\fr{1}{2}$,$\fr{2}{3}$,$\fr{1}{7}$&  8 &H47&$j^3(7,42)$, $t^{6}$  \\
$\fr{1}{3}$,$\fr{2}{3}$,$\fr{2}{3}$,$\fr{1}{7}$&    &C29& $\quad t(\sqrt{7})$  \\
$\fr{1}{2}$,$\fr{1}{2}$,$\fr{1}{2}$,$\fr{1}{3}$&  7 &G35& $j^2(\sqrt{-7})$, $t^{4}$  \\
$\fr{1}{2}$,$\fr{1}{3}$,$\fr{1}{3}$,$\fr{2}{3}$&    &H35& $j^3(0,28)$, $t^{6}$  \\
\end{tabular} \\[2pt]

\noindent \underline{\small \bf Table 2.3.8.  \hspace{158pt} \hfill \mbox{}  } \\[1pt] 
\begin{tabular}{@{}lrll@{}} 
$\fr{1}{8}$,$\fr{1}{8}$,$\fr{1}{8}$,$\fr{1}{8}$& 36 &F6&$j^2(\sqrt{2})$, $t^{4}$  \\
   &                                           &G18& $j^2(\sqrt{-2})$, $t^{12}$  \\
$\fr{2}{8}$,$\fr{2}{8}$,$\fr{1}{8}$,$\fr{1}{8}$& 30 &B10& $\quad t = 3^{2}$  \\
   &                                           &G5& $j^2(\sqrt{-1})$, $t^{4}$  \\
$\fr{1}{8}$,$\fr{1}{8}$,$\fr{1}{8}$,$\fr{3}{8}$&    &D24& $m^2(\sqrt{-2}), j\!=\!0$  \\
$\fr{1}{3}$,$\fr{2}{8}$,$\fr{1}{8}$,$\fr{1}{8}$& 28 &G8&$j^2(\sqrt{-1})$, $t^{4}$  \\
$\fr{1}{2}$,$\fr{1}{8}$,$\fr{1}{8}$,$\fr{1}{8}$& 27 &G6&$j^2(\sqrt{-1})$, $t^{4}$  \\
$\fr{1}{3}$,$\fr{1}{3}$,$\fr{1}{8}$,$\fr{1}{8}$& 26 &C22& $\quad t(\sqrt{6})$  \\
\end{tabular} 
\end{minipage}
\end{table}
\begin{table}
\begin{minipage}{72mm}
\begin{tabular}{@{}lrll@{}} 
$\fr{2}{8}$,$\fr{2}{8}$,$\fr{3}{8}$,$\fr{3}{8}$& 18  &C5& $\quad t(\sqrt{3})$  \\
$\fr{2}{8}$,$\fr{6}{8}$,$\fr{1}{8}$,$\fr{1}{8}$&   &\hspace{2pt}--- &no covering  \\
$\fr{1}{8}$,$\fr{1}{8}$,$\fr{1}{8}$,$\fr{7}{8}$&   &\hspace{2pt}--- &no covering  \\
$\fr{1}{8}$,$\fr{1}{8}$,$\fr{3}{8}$,$\fr{5}{8}$&   &\hspace{2pt}--- &no covering  \\
$\fr{1}{8}$,$\fr{3}{8}$,$\fr{3}{8}$,$\fr{3}{8}$&   &\hspace{2pt}--- &no covering  \\
$\fr{1}{2}$,$\fr{1}{3}$,$\fr{1}{3}$,$\fr{1}{8}$& 17 &J20& $j^{9}$, $t^{18}$  \\
$\fr{4}{8}$,$\fr{1}{3}$,$\fr{2}{8}$,$\fr{2}{8}$& 16&\hspace{2pt}--- &no covering  \\
$\fr{4}{8}$,$\fr{1}{3}$,$\fr{1}{8}$,$\fr{3}{8}$&   &\hspace{2pt}--- &no covering  \\
\end{tabular} \\[2pt]

\noindent \underline{\small \bf Table 2.3.9.  \hspace{148pt} \hfill \mbox{}  } \\[1pt] 
\begin{tabular}{@{}lrll@{}} 
$\fr{3}{9}$,$\fr{1}{9}$,$\fr{1}{9}$,$\fr{1}{9}$& 24&\hspace{2pt}--- &no covering  \\
$\fr{1}{9}$,$\fr{1}{9}$,$\fr{2}{9}$,$\fr{2}{9}$&    &A24& $\quad t = -1$  \\
   &                                           &H8&$j^3(-3,4)$, $t^{6}$  \\
$\fr{1}{3}$,$\fr{1}{9}$,$\fr{1}{9}$,$\fr{2}{9}$& 22 &H9&$j^3(-3,4)$, $t^{6}$  \\
$\fr{1}{2}$,$\fr{1}{9}$,$\fr{1}{9}$,$\fr{1}{9}$& 21 &G52&$j^2(\sqrt{-15})$, $t^{12}$  \\
$\fr{1}{3}$,$\fr{1}{3}$,$\fr{1}{9}$,$\fr{1}{9}$& 20 &H1&$j^3(-3,1)$  \\
   &                                           &J3&$j^{6}(3,3,0,0,5), t^{12}$  \\
$\fr{3}{9}$,$\fr{3}{9}$,$\fr{1}{9}$,$\fr{2}{9}$& 18&\hspace{2pt}--- &no covering  \\
$\fr{3}{9}$,$\fr{1}{9}$,$\fr{1}{9}$,$\fr{4}{9}$&   &\hspace{2pt}--- &no covering  \\
\end{tabular}
\end{minipage} \hfil
\begin{minipage}{72mm}
\begin{tabular}{@{}lrll@{}} 
$\fr{3}{9}$,$\fr{2}{9}$,$\fr{2}{9}$,$\fr{2}{9}$&  18 &\hspace{2pt}--- &no covering  \\
$\fr{6}{9}$,$\fr{1}{9}$,$\fr{1}{9}$,$\fr{1}{9}$&    &E3& $\quad t^{\rm spl}(3,2)$  \\
$\fr{1}{9}$,$\fr{1}{9}$,$\fr{2}{9}$,$\fr{5}{9}$&    &C23& $\quad t(\sqrt{6})$  \\
$\fr{1}{9}$,$\fr{2}{9}$,$\fr{2}{9}$,$\fr{4}{9}$&   &\hspace{2pt}--- &no covering  \\
$\fr{1}{3}$,$\fr{3}{9}$,$\fr{3}{9}$,$\fr{1}{9}$& 16 &A6& $\quad t = -1$  \\
$\fr{1}{3}$,$\fr{3}{9}$,$\fr{2}{9}$,$\fr{2}{9}$&    &C7& $\quad t(\sqrt{3})$  \\
$\fr{1}{3}$,$\fr{1}{9}$,$\fr{1}{9}$,$\fr{5}{9}$&    &H25&$j^3(3,1)$, $t^{6}$  \\
$\fr{1}{3}$,$\fr{1}{9}$,$\fr{2}{9}$,$\fr{4}{9}$&    &H6&$j^3(3,2)$  \\
\end{tabular} \\[2pt] 

\noindent \underline{\small \bf Table 2.3.10.  \hspace{152pt} \hfill \mbox{}  } \\[1pt] 
\begin{tabular}{@{}lrll@{}} 
$\frc{1}{10}$,$\frc{1}{10}$,$\frc{1}{10}$,$\frc{1}{10}$& 24 &D6&$\quad t(\sqrt{-1})$  \\
   &                                                   &F17&$j^2(\sqrt{5})$, $t^{4}$  \\
$\frc{5}{10}$,$\frc{1}{10}$,$\frc{1}{10}$,$\frc{1}{10}$ & 18 &E7& $\quad t^{\rm spl}(5,10)$  \\
$\frc{2}{10}$,$\frc{2}{10}$,$\frc{2}{10}$,$\frc{2}{10}$    &   &\hspace{2pt}--- &no covering  \\
$\frc{2}{10}$,$\frc{2}{10}$,$\frc{1}{10}$,$\frc{3}{10}$  &    &B24& $\quad t = 2^{7} / 3  $  \\
$\frc{2}{10}$,$\frc{4}{10}$,$\frc{1}{10}$,$\frc{1}{10}$  &    &D4& $\quad t(\sqrt{-1})$  \\
$\frc{1}{10}$,$\frc{1}{10}$,$\frc{3}{10}$,$\frc{3}{10}$&    &C25& $\quad t(\sqrt{6})$  \\
$\fr{1}{3}$,$\frc{2}{10}$,$\frc{2}{10}$,$\frc{2}{10}$    & 16&\hspace{2pt}--- &no covering  \\
\end{tabular}
\end{minipage}
\end{table}
\begin{table}
\begin{minipage}{75mm}   
\noindent \underline{\small \bf Table 2.3.11.  \hspace{154pt} \hfill \mbox{}  } \\[1pt] 
\begin{tabular}{@{}lrll@{}} 
$\frc{1}{11}$,$\frc{1}{11}$,$\frc{1}{11}$,$\frc{4}{11}$& 18&\hspace{2pt}--- &no covering  \\
$\frc{1}{11}$,$\frc{1}{11}$,$\frc{2}{11}$,$\frc{3}{11}$&    &D15&$\quad t(\sqrt{-2})$  \\
$\frc{1}{11}$,$\frc{2}{11}$,$\frc{2}{11}$,$\frc{2}{11}$&   &\hspace{2pt}--- &no covering  \\
$\fr{1}{3}$,$\frc{1}{11}$,$\frc{1}{11}$,$\frc{3}{11}$ & 16 &D41&$\quad t(\sqrt{-7})$  \\
$\fr{1}{3}$,$\frc{1}{11}$,$\frc{2}{11}$,$\frc{2}{11}$ &    &C13&$\quad t(\sqrt{3})$  \\
$\fr{1}{2}$,$\frc{1}{11}$,$\frc{1}{11}$,$\frc{2}{11}$ & 15 &H26&$j^3(-4,4)$, $t^{6}$  \\
$\fr{1}{3}$,$\fr{1}{3}$,$\frc{1}{11}$,$\frc{2}{11}$  & 14 &H28&$j^3(-4,4)$, $t^{6}$  \\
$\fr{2}{3}$,$\frc{1}{11}$,$\frc{1}{11}$,$\frc{1}{11}$ &    &E15&$\quad t^{\rm spl}(-11,22)$  \\
$\fr{1}{2}$,$\fr{1}{3}$,$\frc{1}{11}$,$\frc{1}{11}$  & 13 &J9 
  & \small$\!j^{6}\mbox{\scriptsize$\!(30,\!20,\!216,\!372,\!172), t^{12}$}$  \\
$\fr{1}{3}$,$\fr{1}{3}$,$\fr{1}{3}$,$\frc{1}{11}$   & 12 &G45&$j^2(\sqrt{-11}), t^{12}$  \\
\end{tabular} \\

\noindent \underline{\small \bf Table 2.3.12.  \hspace{153pt} \hfill \mbox{}  } \\[1pt] 
\begin{tabular}{@{}lrll@{}} 
$\frc{3}{12}$,$\frc{1}{12}$,$\frc{1}{12}$,$\frc{1}{12}$& 18 &D21&$\quad j = 0$  \\
$\frc{2}{12}$,$\frc{2}{12}$,$\frc{1}{12}$,$\frc{1}{12}$ &    &B11&$\quad t = 3^{2}$  \\
$\fr{1}{3}$,$\frc{2}{12}$,$\frc{1}{12}$,$\frc{1}{12}$ & 16 &D14&$\quad t(\sqrt{-2})$  \\
   &                                                  &G24&$j^2(\sqrt{-3}), t^{4}$  \\
$\fr{1}{2}$,$\frc{1}{12}$,$\frc{1}{12}$,$\frc{1}{12}$& 15 &G27&$j^2(\sqrt{-3}), t^{12}$  \\
$\fr{1}{3}$,$\fr{1}{3}$,$\frc{1}{12}$,$\frc{1}{12}$ & 14 &F9&$j^2(\sqrt{3})$  \\
   &                                                &G23&$j^2(\sqrt{-3})$, $t^{4}$  \\
\end{tabular} \\[-1pt]

\noindent \underline{\small \bf Table 2.3.13.  \hspace{153pt} \hfill \mbox{}  } \\[1pt] 
\begin{tabular}{@{}lrll@{}} 
$\frc{1}{13}$,$\frc{1}{13}$,$\frc{1}{13}$,$\frc{2}{13}$& 18 &E16&$\quad t^{\rm spl}(-1,2)$  \\
$\fr{1}{3}$,$\frc{1}{13}$,$\frc{1}{13}$,$\frc{1}{13}$ & 16 &G28&$j^2(\sqrt{-3}), t^{12}$  \\
\end{tabular} \\

\noindent \underline{\small \bf Table 2.3.14.  \hspace{153pt} \hfill \mbox{}  } \\[1pt] 
\begin{tabular}{@{}lrll@{}} 
$\frc{1}{14}$,$\frc{1}{14}$,$\frc{1}{14}$,$\frc{1}{14}$& 18 &C2&$\quad\; t(\sqrt{2})$  \\
\end{tabular} \\[-1pt]

\noindent \underline{\small \bf Table 2.4.5.  \hspace{157pt} \hfill \mbox{}  } \\[1pt] 
\begin{tabular}{@{}lrll@{}} 
$\fr{1}{5}$,$\fr{1}{5}$,$\fr{1}{5}$,$\fr{1}{5}$& 24 &A20&$\quad t = -1$  \\
   &                                           &F12&$j^2(\sqrt{5})$  \\
   &                                           &G48&$j^2(\sqrt{-15})$, $t^{4}$  \\
$\fr{1}{4}$,$\fr{1}{4}$,$\fr{1}{5}$,$\fr{1}{5}$& 22 &I30&$j^{5}(\mbox{\small$10,30,30,8$}), t^{10}$  \\
$\fr{1}{4}$,$\fr{1}{4}$,$\fr{1}{4}$,$\fr{1}{4}$& 20 &C17&$\quad t(\sqrt{5})$  \\
   &                                           &D45&$\quad t(\sqrt{-15})$  \\
   &                                           &E11&$\quad t^{\rm spl}(5,10)$  \\
$\fr{1}{5}$,$\fr{1}{5}$,$\fr{1}{5}$,$\fr{2}{5}$&    &E9&$\quad t^{\rm spl}(5,10)$  \\
\end{tabular} 
\end{minipage} \hfil
\begin{minipage}{70mm}  \vspace{-4pt}
\begin{tabular}{@{}lrll@{}}
$\fr{1}{5}$,$\fr{1}{5}$,$\fr{1}{5}$,$\fr{4}{5}$&  12    &D3&$\quad t(\sqrt{-1})$  \\
$\fr{1}{5}$,$\fr{1}{5}$,$\fr{2}{5}$,$\fr{3}{5}$&    &B23&$\quad t = 2^{7} / 3  $  \\
$\fr{1}{5}$,$\fr{2}{5}$,$\fr{2}{5}$,$\fr{2}{5}$&   &\hspace{2pt}--- &no covering  \\
$\fr{1}{2}$,$\fr{2}{4}$,$\fr{1}{4}$,$\fr{1}{5}$& 11 &I29&$j^{5}(10,30,30,8)$  \\
$\fr{1}{2}$,$\fr{1}{2}$,$\fr{1}{4}$,$\fr{1}{4}$& 10 &C14& $\quad t(\sqrt{5})$  \\
   &                                           &C30&$\quad t(\sqrt{10})$  \\
   &                                           &H17&$j^3(5,10)$, $t^{6}$  \\
$\fr{2}{4}$,$\fr{2}{4}$,$\fr{1}{4}$,$\fr{1}{4}$&    &F13&$j^2(\sqrt{5})$, $t^{4}$  \\
$\fr{2}{4}$,$\fr{1}{5}$,$\fr{1}{5}$,$\fr{3}{5}$&    &D43& $\quad t(\sqrt{-15})$  \\
$\fr{2}{4}$,$\fr{1}{5}$,$\fr{2}{5}$,$\fr{2}{5}$&    &C16& $\quad t(\sqrt{5})$  \\
$\fr{1}{4}$,$\fr{1}{4}$,$\fr{1}{4}$,$\fr{3}{4}$&   &\hspace{2pt}--- &no covering  \\
$\fr{1}{4}$,$\fr{1}{4}$,$\fr{1}{5}$,$\fr{4}{5}$&    &B8& $\quad t = 3^{2}$  \\
$\fr{1}{4}$,$\fr{1}{4}$,$\fr{2}{5}$,$\fr{3}{5}$&   &\hspace{2pt}--- &no covering  \\
$\fr{1}{2}$,$\fr{1}{4}$,$\fr{1}{5}$,$\fr{3}{5}$&  9 &I24&$j^{5}(-2,4,-6,4)$  \\
$\fr{1}{2}$,$\fr{1}{4}$,$\fr{2}{5}$,$\fr{2}{5}$&    &D27&$\quad t(\sqrt{-3})$  \\
$\fr{1}{2}$,$\fr{1}{2}$,$\fr{1}{5}$,$\fr{2}{5}$&  8 &H19&$j^3(5,10)$, $t^{6}$  \\
$\fr{2}{4}$,$\fr{2}{4}$,$\fr{1}{5}$,$\fr{2}{5}$&   &\hspace{2pt}--- &no covering  \\
$\fr{2}{4}$,$\fr{1}{4}$,$\fr{1}{4}$,$\fr{3}{5}$&    &C19&$\quad t(\sqrt{6})$  \\
$\fr{1}{4}$,$\fr{3}{4}$,$\fr{1}{5}$,$\fr{2}{5}$&    &B26&$\quad t = 3^{4}/2^{5}$  \\
$\fr{1}{2}$,$\fr{2}{4}$,$\fr{1}{4}$,$\fr{2}{5}$&  7 &H39&$j^3(2,2)$  \\
$\fr{1}{2}$,$\fr{3}{4}$,$\fr{1}{5}$,$\fr{1}{5}$&    &G29&$j^2(\sqrt{-5})$, $t^{4}$  \\
$\fr{1}{2}$,$\fr{1}{2}$,$\fr{2}{4}$,$\fr{1}{5}$&  6 &G46&$j^2(\sqrt{-15})$, $t^{4}$  \\
$\fr{1}{2}$,$\fr{1}{2}$,$\fr{1}{2}$,$\fr{1}{4}$&  5 &E10&$\quad t^{\rm spl}(5,10)$  \\
\end{tabular} \\[2pt] 

\noindent \underline{\small \bf Table 2.4.6.  \hspace{148pt} \hfill \mbox{}  } \\[1pt] 
\begin{tabular}{@{}lrll@{}} 
$\fr{1}{6}$,$\fr{1}{6}$,$\fr{1}{6}$,$\fr{1}{6}$& 16 &B28& $\quad t = 3^{4}/2^{5}$  \\
   &                                           &D11& $\quad t(\sqrt{-2})$  \\
$\fr{1}{4}$,$\fr{1}{4}$,$\fr{1}{6}$,$\fr{1}{6}$& 14 &C11& $\quad t(\sqrt{3})$  \\
   &                                           &I13& $j^4(0,4,12), t^{8}$  \\
$\fr{3}{6}$,$\fr{1}{6}$,$\fr{1}{6}$,$\fr{1}{6}$& 12 &D18& $\quad j = 0$  \\
$\fr{2}{6}$,$\fr{2}{6}$,$\fr{1}{6}$,$\fr{1}{6}$&    &B5& $\quad t = 3^{2}$  \\
$\fr{1}{4}$,$\fr{1}{4}$,$\fr{1}{4}$,$\fr{1}{4}$&    &B4& $\quad t = 2^{2}$  \\
   &                                           &B22& $\quad t = 2^{7} / 3  $  \\
$\fr{2}{4}$,$\fr{2}{6}$,$\fr{1}{6}$,$\fr{1}{6}$& 10&\hspace{2pt}--- &no covering  \\
$\fr{3}{6}$,$\fr{1}{4}$,$\fr{1}{4}$,$\fr{1}{6}$&    &H5& $j^3(0,2)$, $t^{6}$  \\
$\fr{2}{6}$,$\fr{2}{6}$,$\fr{1}{4}$,$\fr{1}{4}$&    &D32& $\quad t(\sqrt{-5})$  \\
$\fr{1}{2}$,$\fr{2}{6}$,$\fr{1}{4}$,$\fr{1}{6}$&  9 &I1&$j^4(0,8,12)$  \\
$\fr{1}{2}$,$\fr{1}{2}$,$\fr{1}{6}$,$\fr{1}{6}$&  8 &C3&$\quad t(\sqrt{2})$  \\
   &                                           &G11&$j^2(\sqrt{-2})$, $t^{4}$  \\
$\fr{2}{4}$,$\fr{2}{4}$,$\fr{1}{6}$,$\fr{1}{6}$&   &\hspace{2pt}--- &no covering  \\
$\fr{2}{4}$,$\fr{2}{6}$,$\fr{1}{4}$,$\fr{1}{4}$&    &D12&$\quad t(\sqrt{-2})$  \\
$\fr{1}{4}$,$\fr{3}{4}$,$\fr{1}{6}$,$\fr{1}{6}$&    &G20&$j^2(\sqrt{-3})$, $t^{4}$  \\
$\fr{1}{2}$,$\fr{2}{4}$,$\fr{1}{4}$,$\fr{1}{6}$&  7 &I12&$j^4(0,4,12)$  \\
$\fr{1}{2}$,$\fr{1}{2}$,$\fr{1}{4}$,$\fr{1}{4}$&  6 &C24&$\quad t(\sqrt{6})$  \\
   &                                           &D25&$\quad t(\sqrt{-3})$  \\
\end{tabular} \\[2pt]

\noindent \underline{\small \bf Table 2.4.7.  \hspace{148pt} \hfill \mbox{}  } \\[1pt] 
\begin{tabular}{@{}lrll@{}} 
$\fr{1}{7}$,$\fr{1}{7}$,$\fr{1}{7}$,$\fr{2}{7}$& 12&\hspace{2pt}--- &no covering  \\
$\fr{2}{4}$,$\fr{1}{7}$,$\fr{1}{7}$,$\fr{1}{7}$& 10&\hspace{2pt}--- &no covering  \\
$\fr{1}{4}$,$\fr{1}{4}$,$\fr{1}{7}$,$\fr{2}{7}$&    &D34& $\quad t(\sqrt{-5})$  \\
$\fr{1}{2}$,$\fr{1}{4}$,$\fr{1}{7}$,$\fr{1}{7}$&  9 &I7&$j^4(2,8,9)$, $t^{8}$  \\
$\fr{2}{4}$,$\fr{1}{4}$,$\fr{1}{4}$,$\fr{1}{7}$&  8 &G37&$j^2(\sqrt{-7})$, $t^{4}$  \\
$\fr{1}{2}$,$\fr{1}{4}$,$\fr{1}{4}$,$\fr{1}{4}$&  7 &E24& $\quad t^{\rm spl}(-7,14)$  \\
\end{tabular} \\[2pt] 
\end{minipage}
\end{table}
\begin{table}
\begin{minipage}{72mm}
\noindent \underline{\small \bf Table 2.4.8.  \hspace{138pt} \hfill \mbox{}  } \\[1pt] 
\begin{tabular}{@{}lrll@{}} 
$\fr{1}{8}$,$\fr{1}{8}$,$\fr{1}{8}$,$\fr{1}{8}$& 12&\hspace{2pt}--- &no covering  \\
$\fr{1}{4}$,$\fr{1}{4}$,$\fr{1}{8}$,$\fr{1}{8}$& 10 &B9&$\quad t = 3^{2}$  \\
   &                                           &G4&$j^2(\sqrt{-1})$, $t^{4}$  \\
$\fr{1}{4}$,$\fr{1}{4}$,$\fr{1}{4}$,$\fr{1}{4}$&  8 &A16& $\quad t = -1$  \\
\end{tabular} \\[2pt]

\noindent \underline{\small \bf Table 2.5.5.  \hspace{138pt} \hfill \mbox{}  } \\[1pt] 
\begin{tabular}{@{}lrll@{}} 
$\fr{1}{5}$,$\fr{1}{5}$,$\fr{1}{5}$,$\fr{1}{5}$& 12 &A19& $\quad t = -1$  \\
   &                                           &F11&$j^2(\sqrt{5})$  \\
   &                                           &G47&$j^2(\sqrt{-15})$, $t^{4}$  \\
$\fr{1}{5}$,$\fr{1}{5}$,$\fr{1}{5}$,$\fr{2}{5}$& 10 &E8& $\quad t^{\rm spl}(5,10)$  \\
$\fr{1}{5}$,$\fr{1}{5}$,$\fr{1}{5}$,$\fr{3}{5}$&  8 &E17& $\quad t^{\rm spl}(0,10)$  \\
$\fr{1}{5}$,$\fr{1}{5}$,$\fr{2}{5}$,$\fr{2}{5}$&    &H15&$j^3(5,10)$  \\
$\fr{1}{2}$,$\fr{1}{5}$,$\fr{1}{5}$,$\fr{2}{5}$&  7 &D37& $\quad t(\sqrt{-7})$  \\
$\fr{1}{2}$,$\fr{1}{2}$,$\fr{1}{5}$,$\fr{1}{5}$&  6 &A1& $\quad t = -1$  \\
   &                                           &G49&$j^2(\sqrt{-15})$, $t^{4}$  \\
\end{tabular} \\[2pt]

\noindent \underline{\small \bf Table 2.5.6.  \hspace{138pt} \hfill \mbox{}  } \\[1pt] 
\begin{tabular}{@{}lrll@{}} 
$\fr{1}{6}$,$\fr{1}{6}$,$\fr{1}{6}$,$\fr{1}{6}$& 10 &C31& $\quad t(\sqrt{10})$  \\
$\fr{2}{6}$,$\fr{1}{5}$,$\fr{1}{5}$,$\fr{1}{5}$&  8&\hspace{2pt}--- &no covering  \\
$\fr{1}{5}$,$\fr{2}{5}$,$\fr{1}{6}$,$\fr{1}{6}$&    &D36& $\quad t(\sqrt{-6})$  \\
$\fr{1}{2}$,$\fr{1}{5}$,$\fr{1}{5}$,$\fr{1}{6}$&  7 &H24&$j^3(3,1)$, $t^{6}$  \\
\end{tabular} \\[2pt] 

\noindent \underline{\small \bf  Table 2.5.7.  \hspace{138pt} \hfill \mbox{}  } \\[1pt] 
\begin{tabular}{@{}lrll@{}} 
$\fr{1}{5}$,$\fr{1}{5}$,$\fr{1}{5}$,$\fr{1}{7}$&  $\;$ 8 &E23&$\quad t^{\rm spl}(2,2)$  \\
\end{tabular} 
\end{minipage} \hfil
\begin{minipage}{72mm} \vspace{-2pt}
\noindent \underline{\small \bf Table 2.6.6.  \hspace{138pt} \hfill \mbox{}  } \\[1pt] 
\begin{tabular}{@{}lrll@{}} 
$\fr{1}{6}$,$\fr{1}{6}$,$\fr{1}{6}$,$\fr{1}{6}$&  $\;$ 8 &B27&$\quad t = 3^{4}/2^{5}$  \\
   &                                           &D10&$\quad t(\sqrt{-2})$  \\
\end{tabular} \\[-2pt] 

\noindent \underline{\small \bf Table 3.3.4.  \hspace{138pt} \hfill \mbox{}  } \\[1pt] 
\begin{tabular}{@{}lrll@{}} 
$\fr{1}{4}$,$\fr{1}{4}$,$\fr{1}{4}$,$\fr{1}{4}$& 12 &A17& $\quad t = -1$  \\
$\fr{1}{3}$,$\fr{1}{3}$,$\fr{1}{4}$,$\fr{1}{4}$& 10 &A8& $\quad t = -1$  \\
   &                                           &H21&$j^3(5,10)$, $t^{6}$  \\
$\fr{2}{4}$,$\fr{1}{4}$,$\fr{1}{4}$,$\fr{1}{4}$&  9&\hspace{2pt}--- &no covering  \\
$\fr{1}{3}$,$\fr{1}{3}$,$\fr{1}{3}$,$\fr{1}{4}$&    &E4& $\quad t^{\rm spl}(-3,10)$  \\
$\fr{1}{3}$,$\fr{1}{3}$,$\fr{1}{3}$,$\fr{1}{3}$&  8 &D7&$\quad t(\sqrt{-2})$  \\
   &                                           &F1&$j^2(\sqrt{2})$  \\
$\fr{2}{4}$,$\fr{1}{3}$,$\fr{1}{3}$,$\fr{1}{4}$&  7 &F7& $m^4$, $j(\sqrt{2})$, $t^{4}$  \\
$\fr{2}{4}$,$\fr{1}{3}$,$\fr{1}{3}$,$\fr{1}{3}$&  6 &E1& $\quad t^{\rm spl}(3,2)$  \\
$\fr{1}{3}$,$\fr{2}{3}$,$\fr{1}{4}$,$\fr{1}{4}$&    &C20& $\quad t(\sqrt{6})$  \\
$\fr{1}{3}$,$\fr{1}{3}$,$\fr{2}{3}$,$\fr{1}{4}$&  5 &F18&$j^2(\sqrt{6})$, $t^{4}$  \\
\end{tabular} \\

\noindent \underline{\small \bf  Table 3.3.5.  \hspace{138pt} \hfill \mbox{}  } \\[1pt] 
\begin{tabular}{@{}lrll@{}} 
$\fr{1}{5}$,$\fr{1}{5}$,$\fr{1}{5}$,$\fr{1}{5}$&  $\;$ 9 &\hspace{2pt}--- &no covering  \\
$\fr{1}{3}$,$\fr{1}{3}$,$\fr{1}{5}$,$\fr{1}{5}$&  7 &C37& $\quad t(\sqrt{21})$  \\
$\fr{1}{3}$,$\fr{1}{3}$,$\fr{1}{3}$,$\fr{1}{5}$&  6 &E12& $\quad t^{\rm spl}(3,1)$  \\
$\fr{1}{3}$,$\fr{1}{3}$,$\fr{1}{3}$,$\fr{1}{3}$&  5 &B18& $\quad t = 2^{5} / 5  $  \\
\end{tabular} \\

\noindent \underline{\small \bf Table 3.4.4.  \hspace{138pt} \hfill \mbox{}  } \\[1pt] 
\begin{tabular}{@{}lrll@{}} 
$\fr{1}{4}$,$\fr{1}{4}$,$\fr{1}{4}$,$\fr{1}{4}$&  $\;$ 6 &B3& $\quad t = 2^{2}$  \\
   &                                           &B21& $\quad t = 2^{7} / 3  $  \\
$\fr{1}{3}$,$\fr{1}{3}$,$\fr{1}{4}$,$\fr{1}{4}$&  5 &D31& $\quad t(\sqrt{-5})$  \\
\end{tabular}  \\[2pt] 
\end{minipage}
\end{table}

\begin{table}
\begin{tabular}{@{}l|c|rc|rr|rrrrrrr@{}}
\hline $k,\ell,m$ & Max. & \multicolumn{2}{|l}{Br. patterns} & \multicolumn{2}{|l}{Coverings} & 
\multicolumn{7}{|l}{Moduli field degree} \\
\cline{3-4} \cline{5-6} \cline{7-13}
 & \!\!degree\!\! & Total\!\!\! & No cov.\!\! & Orb.\! & Total & 1 & 2 & 3 & 4 & 5 & 6 & $\!\!\!\ge 7$\qquad \\ \hline
2, 3, 7 & 60 & 152 & 30 & 140 & 427 & 51 & 27 & 25 & 11 & 8 & 5 & 13 \\
2, 3, 8 & 36 & 65 & 16 & 58 & 130 & 23 & 24 & 4 & 4 & & 2 & 2 \\
2, 3, 9 & 24 & 32 & 6 & 29 & 67 & 12 & 3 & 11 & 1 & & 2 & \\
2, 3, 10 & 24 & 20 & 2 & 21 & 38 & 13 & 5 & 1 & & & 2 & \\
2, 3, 11..14 & 18 & 18 & 2 & 18 & 33 & 9 & 6 & 2 & & & 1 & \\
2, 4, 5 & 24 & 40 & 5 & 42 & 91 & 21 & 10 & 6 & 1 & 4 & & 1 \\ 
2, 4, 6/7/8 & 16 & 24 & 5 & 25 & 43 & 16 & 4 & 1 & 4 & & & \\
2, 5/6/7 & 12 & 12 & 1 & 15 & 22 & 10 & 3 & 2 & & & & \\
3, 3/4/5 & 12 & 15 & 2 & 16 & 21 & 12 & 3 & 1 & & & & \\ \hline
Total & --- & \!378 & 69 & \!366 & \!872 & \!167 & 85 & 53 & 21 &12 & 12 & 16 \\ \hline 
\end{tabular}
\caption{Statistics of Belyi maps}  \label{statistics}
\end{table}


The last column of the same tables gives basic information about the size of Galois orbits, 
$j$-fields, $t$-fields of the computed Belyi functions. 
The $j$-field is indicated as follows:
\begin{itemize} \itemseparate
\item by the field degree $d$, in the power notation $j^d$;
\item if the degree is 3, 4, 5 or 6, a minimal field polynomial 
\mbox{$X^d+a_{d-2}X^{d-2}+\ldots+a_1X+a_0$}
is indicated by $j^d(a_{d-2},\ldots,a_1,a_0)$;
\item if the field is quadratic, $j^2(\sqrt{a})$ means the field $\QQ(\sqrt{a})$;
\item if $j=0$, it is stated so;
\item for $j\in\QQ\setminus\{0\}$, no $j$-notation is given,
but the $t$-field and (possibly) the moduli field are indicated.
\end{itemize}
The $t$-field is specified as follows:
\begin{itemize} \itemseparate
\item if the $j$-field is indicated, 
the $t$-field degree $d$ is given (in the power notation $t^d$)
only if $j\neq 0$ and the $t$-field is an extension of the $j$-field;
\item if $j\in\QQ\setminus\{0\}$ and $t\in\QQ$, a value of $t$ is given in 
the factorized form (as motivated by \S \ref{sec:arithmetic}); 
\item if $j\in\QQ\setminus\{0\}$ and the $t$-field is quadratic, $t(\sqrt{a})$ means the field $\QQ(\sqrt{a})$;
\item if $j\in\QQ\setminus\{0\}$ and the $t$-field degree is greater than 2,
$t^{\rm spl}(a,b)$ means the splitting field (of degree 6 in our examples) of the polynomial
$X^3+aX+b$.
\end{itemize}
The size of the Galois orbit\footnote{Our notation allows to count the total number
of dessins d'enfant in selected Galois orbits rather quickly in Tables 2.3.7--3.4.4.
Each fourth column starts either with the $m^d$ or $j^d$ notation 
(where $d$ is the size of the Galois orbit), or a statement of no covering,
or starts with an indented data about $t$ or $j=0$. In the latter cases, $d=1$.}
is equal to the degree of the moduli field.Ê 
In most cases, the $j$-field and the moduli field coincide\footnote{
The moduli field always contains the $j$-field, 
as the $j$-value is an invariant of M\"obius transformations. 
In our encountered cases,  the moduli field is at most a quadratic extension of the $j$-field.}.
The moduli field is indicated only if it differs from the $j$-field,
either with $m^2(\sqrt{a})$ if its degree is $2$, or with $m^d$ if its degree $d>2$.
Statistics of computed Belyi coverings is presented in Table \ref{statistics}.


More information about each computed Galois orbit can be found in the tables 
of Appendix \S \ref{sec:ajtables} and our website \cite{HeunURL}.
In order to compute and simplify the whole set of minus-4-hyperbolic Belyi functions,
and to obtain interesting additional information about them,
we used the computer algebra package {\sf Maple 15}, 
the {\sf polredabs} command of {\sf GP/PARI},
and had to implement several algorithms.
The main work of computing the Belyi functions is described in \S \ref{sec:algorithms}.
Here is a list of additional handled problems, sorted roughly by the amount of work involved:
\begin{itemize}
\item Given a minus-4-hyperbolic Belyi function, 
compute its branching type, its $t$-value, $j$-invariant,
the canonical realization field, and moduli field.
\item Given two triples $(g_1,g_0,g_{\infty})$ of elements in $S_d$ as in (III) of \S \ref{sec:deeper},
decide whether they represent the same dessin d'enfant up to conjugacy in $S_d$.
\item In the {\em obstructed} cases as described in \S \ref{sec:conic},
compute the obstruction conic and a conic-model (if possible).
\item Find possible decompositions of a Belyi function $\varphi(x)$ into smaller degree 
rational functions. 
\item Given a Belyi function $\varphi\in K(x)$ and an embedding $K \rightarrow \CC$,
compute the dessin d'enfant $\varphi$ under this embedding.
\item Find a M\"obius-equivalent Belyi function $\tilde{\varphi}(x)$ of substantially
smaller bitsize, if possible.
\end{itemize}
Algorithmic solution of less trivial problems of this list is elaborated in \cite[\S 3, \S 4]{BelyiKit}.
The implementations and computed data are available at \cite{HeunURL}.

A portion of computed Belyi functions has been known, inevitably. 
Most notably, the Belyi covering G45 defined over $\QQ(\sqrt{-11})$ has the monodromy group 
isomorphic to the sporadic  Mathieu group $M_{12}$. Its humanoid dessin d'enfant 
is called {\em Monsieur Mathieu}; see the appendix dessins.
The Galois orbits A19, G47, F11 are considered in \cite{Filimonenkov} and \cite[Example 5.7]{zvonkin10}.
The C30 dessin (turned $90^0$ in Figure \ref{fg:mdessins} here) 
appears in \cite{CouveignesGranboulan} as a {\em 
rabbit with a lopped off left ear and a sidelong smirk on the right hand side.}
The degree 24 coverings with $(k,\ell)=(2,3)$ were computed in \cite{BeukersMontanus}. 

An important area where Belyi functions appear is Shimura curves \cite{elkies237}, \cite{voightg2},
\cite{tuyang}. Checking the list of low genus 
Shimura curves ${\cal X}_0(n)$ in \cite{voightg2}, we recognize our H10, H11, H12
as Belyi coverings for the congruence groups 
$\Gamma_0(29), \Gamma_0(43), \Gamma_0(13)\subset PSL(2,{\cal O})$,
where $\cal O$ is the quaternion order over $\QQ(\mbox{Re } \zeta_7)$ considered in \cite{elkies237}.
H1 is a Belyi covering for $\Gamma_0(19)$ of similar quaternions over $\QQ(\mbox{Re } \zeta_9)$. 
The coverings A2, A6, A7, A8, A16, A19, A20, C1 appear in diagrams III, VI, XI in \cite{tuyang}.
Minus-4-hyperbolic Belyi functions appear in coverings of classical modular curves as well.
Checking the Cummins-Pauli online list \cite{CumminsPauli} of genus 0 congruence subgroups
of $PSL(2,\ZZ)$, we recognize A5, A18, B19, C1, D8, D19, G13, G35, G39 as coverings for the
congruence subgroups 
8I$^0=\Gamma_1(8)$, 8G$^0$, 10A$^0$, 8D$^0$, 8E$^0$, 9C$^0$, 8A$^0$, 7A$^0$, 7C$^0$,
respectively. 
Any Belyi covering gives a modular curve with respect to some (not necessarily congruence)
subgroup of $PSL(2,\ZZ)$, since $\Gamma(2)\subset PSL(2,\ZZ)$ is a free group on two generators
\cite{voightg2}. The minus-4-hyperbolic functions tend to give Shimura curves corresponding to 
manageable  non-congruence subgroups. Our computational routine \cite[\sf ComputeBelyi.mpl]{HeunURL} 
can be used to investigate genus 0 Shimura curves more thoroughly. 

\section{Application to Heun functions}
\label{sec:heun}

The minus-4-hyperbolic Belyi functions have application to transformations between
hypergeometric and Heun functions (or their differential equations). 
This allows to express some Heun functions in terms of better understood
hypergeometric functions. In fact, we utilize this application
in our algorithms to compute the Belyi functions.

The Gauss hypergeometric equation
\begin{equation} \label{HGE}
\frac{d^2y(z)}{dz^2}+
\left(\frac{C}{z}+\frac{A+B-C+1}{z-1}\right)\,\frac{dy(z)}{dz}+\frac{A\,B}{z\,(z-1)}\,y(z)=0
\end{equation}
and the Heun differential equation
\begin{equation} \label{Heun}
\frac{d^2y(x)}{dx^2}+\biggl(\frac{c}{x}+\frac{d}{x-1}+\frac{a+b-c-d+1}{x-t}\biggr)\frac{dy(x)}{dx}
+\frac{abx-q}{x(x-1)(x-t)}y(x)=0
\end{equation}
are second order Fuchsian equations \cite{Wikipedia} with 3 or 4 singularities, respectively.
The singular points are $z=0,1,\infty$ and $x=0,1,t,\infty$. 
If $C\not\in \ZZ$, a basis of local solutions of (\ref{HGE}) at $x=0$ is given
by the famous {\em Gauss hypergeometric series}: 
\begin{equation} \label{eq:locsols}
z^0\;\cdot\;\hpg{2}{1}{A,\,B}{C}{\,z}, \qquad z^{1-C}\;\cdot\;\hpg{2}{1}{\!1+A-C,\,1+B-C}{2-C}{\,z}.
\end{equation}
The starting powers $0,1-C$ of the local parameter $z$ are the {\em local exponents} at $z=0$.
The local exponents at $z=1$ are $0,C-A-B$, while the exponents at $z=\infty$ are $A,B$. 
The local exponents for Heun's equation (\ref{Heun}) are
\begin{align*}
\mbox{at } x=0: & \quad 0, 1-c; &  \mbox{at } x=\infty: & \quad a,b;\\
\mbox{at } x=1: & \quad 0, 1-d; & \mbox{at } x=t: & \quad 0, c+d-a-b.
\end{align*}
The local solution at $x=0$ with the exponent $0$ is denoted by
\begin{equation} \label{eq:HF}
\heun{\,t\,}{q}{\,a,\,b\,}{c;\,d}{\,x\,}.
\end{equation}
The parameter $q$ is an {\em accessory parameter}; it does not influence the local exponents.
If $c\not\in\ZZ$, then an independent local solution at $x=0$ is
\begin{equation} \label{eq:HF2}
x^{1-c}\,\heun{t}{q_1}{a-c+1,\,b-c+1}{2-c;\,d}{x}
\end{equation}
with $q_1=q-(c-1)(a+b-c-d+d\,t+1)$. 

A {\em pull-back transformation} of a differential equation for $y(z)$ in $d/dz$ 
has the form
\begin{equation} \label{algtransf}
z \longmapsto\varphi(x), \qquad y(z)\longmapsto Y(x)=\theta(x)\,y(\varphi(x)),
\end{equation}
where $\varphi(x)$ is a rational function, and $\theta(x)$ is
a radical function (an algebraic root of a rational function).
Geometrically, the transformation {\em pulls-back} a
differential equation on $\PP^1_z$ to a
differential equation on $\PP^1_x$, with
respect to the covering $\varphi:\PP^1_x\to\PP^1_z$
determined by the rational function $\varphi(x)$.

Pull-back transformations between hypergeometric and Heun equations
give identities between the classical Gauss hypergeometric and Heun functions.
For example, we have 
\begin{align} \label{HeunG}
\heun9{7/9}{1/3, 1}{7/9; 2/3}{x} = \theta(x)\;
\hpg21{1/36, 13/36}{8/9}{\varphi(x)},
\end{align}
where $\varphi(x)$ is as in Example \ref{ExampleB7b}, 
and $\theta(x)=(1-x)^{-1/36} \left( {\textstyle 1-x-\frac83 x^2 - \frac8{27}x^3 } \right)^{-1/12}$.
The transformation of singularities and local exponents 
for Fuchsian equations is explained in \cite[Lemma 2.1]{HeunClass}. 
The prefactor $\theta(x)$ shifts the local exponents at some points, but does not change 
the exponent difference anywhere. The rational function $\varphi(x)$ multiplies the local exponents
and their differences by the branching order at each point. If $Q$ is a singularity of the starting 
Fuchsian equation in $d/dz$, a point $P$ above $Q$ will be non-singular for the pulled-back equation
only if the branching order at $P$ is $n$ and the exponent difference at $Q$ is equal to $1/n$
(and $Q$ is not a logarithmic point when $n=1$).
For example, the $\hpgo21$ function in (\ref{HeunG}) solves a hypergeometric equation in $z,d/dz$
with the exponent differences $1/9,1/2,1/3$ at $z=0,1,\infty$, respectively,
while the exponent differences  for the pulled-back Heun equation 
are the branching fractions $2/9,1/3,1/9,2/3$ at $x=0,1,t,\infty$, respectively.
The roots of $8x^3-72x^2-27x+27$ became non-singular after the proper choice of $\theta(x)$.
By the way, the rational function $\varphi(x)$ of Example \ref{ExampleB7b} is identified in our classification
by the label B7 in Table 2.3.9 of \S \ref{sec:ajnumbers} and in Appendix \S \ref{sec:ajtables}.

Recently, {\em parametric} 
transformations between Heun and hypergeometric equations 
without Liouvillian  solutions\footnote{Liouvillian solutions \cite{Wikipedia} 
of second order linear differential equations
are the ``elementary" solutions: power, algebraic, exponential, trigonometric  functions, 
their integrals (in particular, logarithmic and inverse trigonometric functions).
They can be written in the form $y = {\rm exp}(\int r)\,\hpgo21(\varphi)$,
where  $r, \varphi$ are rational functions, $\mbox{}_2F_1$ is Gauss' hypergeometric function
with a reducible, dihedral or finite monodromy. There are algorithms to find Liouvillian solutions
in this form \cite{kleinvhw}, hence a table of pull-back transformations is not needed.
The hyperbolic restriction $1/k+1/\ell+1/m<1$ gives a finite list of $(k,\ell,m)$-minus-4-regular 
Belyi functions, while $1/k+1/\ell+1/m\ge 1$ would lead to infinitely many 
Belyi functions because of applications to Liouvillian functions.}
were classified  in \cite{HeunClass}, 
\cite{HeunForm}. 
They apply to hypergeometric equations where at least one exponent difference is not restricted
to a value $1/n$ with $n\in\NN$; hence a parameter.  In total, there are 61 parametric transformations 
up to the well known symmetries of hypergeometric 
and Heun equations \cite{HeunForm}.
But the number of Galois orbits of utilized Belyi coverings (up to M\"obius transformations) is 48. 
These Belyi functions are listed in \cite[Table 4]{HeunClass}. 
They satisfy condition \refpart{i} but not \refpart{ii} of Definition \ref{df:mnhyperbolic}, 
because the parameter(s) could be specialized to satisfy the hyperbolic condition. 
The parametric transformations are labeled P1--P61 in  \cite{HeunForm},
following similar criteria as in Appendix \S \ref{sc:criteria} here.
The Belyi functions of this article complete the list of hypergeometric-to-Heun transformations
when no Liouvillian solutions are involved.

\begin{remark} \label{rm:noexist} 
Non-existence of Belyi functions with some branching patterns can be proved 
by non-existence of implied transformations of Fuchsian equations \cite[\S 5]{HeunClass}. 
For example, there is no $(2,3,10)$-minus-$4$ Belyi function with the branching pattern
$9\,[2]=6\,[3]=[10]+2+2+2+2$, because it would also be a $(2,3,2)$-minus-$1$ function. 
It would pull-back a hypergeometric equation with
the exponent differences $1/2,1/3,1/2$ to a non-existent Fuchsian equation with a single 
singularity where the exponent difference is not $1$ (but $5$). This example illustrates that there are
no $(k,\ell,m)$-minus-$1$ functions, unless $1\in\{k,\ell,m\}$. In the exceptions, the implied
hypergeometric equation must have a logarithmic singularity with the exponent difference $1$.
In particular, the polynomial $(x^d+1)^k$ is a $(k,1,dk)$-minus-$1$ Belyi function,
and $(x^3-3x)^{2k}/(x^2-2)^{3k}$ is a $(2k,1,3k)$-minus-$1$ Belyi function.
\end{remark}

An interesting observation is that the pull-back covering $\varphi(x)$ can be recovered 
from local solutions of the related hypergeometric and Heun equations, if only an oracle would tell
us one constant. Particularly, suppose that the point $x=0$ of Heun's equation 
lies above the singularity $z=0$ of hypergeometric equation.
Let $y_1,y_2$ denote the hypergeometric local solutions in (\ref{eq:locsols}), respectively,
and let $Y_1,Y_2$ denote the Heun local solutions in (\ref{eq:HF}), (\ref{eq:HF2}), respectively.
We have the formula $Y_1(x)=\theta(x)\,y_1(\varphi(x))$ like (\ref{HeunG}), 
and a similar formula \cite[Lemma 3.1]{HeunForm} relating $y_2,Y_2$
but normalized by a constant $K$ that depends on the first power series term of $\varphi(x)$.
The quotient $\psi_1(x)=Y_2/Y_1$ does not depend on the prefactor $\theta(x)$,
and can be identified with the respective quotient $\psi_0(z)=y_2/y_1$ up to the constant multiple $K$.
We have \mbox{$\psi_1(x)=x^{1-c}\,(1+O(x))$} and $\psi_0(z)=z^{1-C}(1+O(z))$.
The identification $\psi_0(z)=K\psi_1(x)$ gives $z=\psi_0^{-1}(K\psi_1(x))$.
Therefore, the Belyi covering $\varphi(x)$ the composition of  the inverse function 
$\psi_0^{-1}$ with $\psi_1(x)$ multiplied by the proper constant $K$.
For instance, the Belyi function of Example \ref{ExampleB7b} can be computed by inverting the function
\[
z^{1/9}\,\hpg21{5/36, 17/36}{10/9}{z} \left/ \hpg21{1/36, 13/36}{8/9}{z} \right.
\]
and composing with
\[
(K x^2)^{1/9}\,\heun9{187/81}{5/9, 11/9}{11/9; 2/3}{x} \left/ \heun9{7/9}{1/3, 1}{7/9; 2/3}{x} \right.
\] 
where $K=-64/3$. The ratio of two independent solutions of the same differential equation 
of order 2 is called a {\em Schwarz map} of the differential equation. 
We consider the Schwarz maps\footnote{In the general context of Fuchsian equations
related by a pull-back transformation, the pull-back covering can be similarly recovered by
a proper identification  (up to a constant multiple) of Schwarz maps as well.
In fact, our implemented algorithms often assume a pull-back of a hypergeometric equation
to a Fuchsian equation with 4 singularities (rather than canonically normalized Heun's equation),
so to avoid unnecessaary extensions of the moduli field. This is done when two or more branching fractions
are equal and represent points in the same fiber, as demonstrated by the polynomial $W$
in Example \ref{ex:deg54}. Instead of the constants $t,q,K$ in \S \ref{ProgramH}, 
the constants $j,q,K$ were generally used. }
again in Appendix \S \ref{sec:coxeter}.

This observation is significant in a few ways. Firstly, a data base of our Belyi functions could be given by
the data of Heun equations to which they apply (the exponent differences, the parameters $q, t$),
the hyperbolic type $(k,\ell,m)$, and the constant $K$. The Belyi coverings would be then recovered by
reconstructing a rational function from a power series. If $d$ is the degree of a Belyi covering,
$2d+8$ power series terms would suffice (and exclude most of false rational reconstructions). 
Secondly, given a branching pattern 
(and thus the exponent differences of presumably related Heun and hypergeometric equations),
the Belyi coverings $\varphi(x)$ could be computed by assuming undetermined constants $t,q,K$ 
and finding algebraic restrictions between them for reconstruction of $\varphi(x)$ from the power series
of $\psi_0^{-1}(K\psi_1(x))$. This approach does not appear practical, but \S \ref{ProgramV} presents a deterministic algorithm that uses an implied Heun-to-hypergeometric transformation in a similarly
general way, and eliminates all undetermined variables except 3 before calling Gr\"obner basis routines.
And thirdly, our probabilistic algorithm \S \ref{ProgramH} searches through all possible $t,q,K$ in finite fields,
reconstructs possible minus-4-hyperbolic Belyi functions over considered finite fields,
and uses a version of Henzel lifting to produce Belyi functions in $\overline{\QQ}(x)$.

\section{Computation of Belyi coverings}
\label{sec:algorithms}

The list of minus-4-hyperbolic Belyi functions was originally generated by a probabilistic algorithm
by a thorough examination of Heun functions and their Schwarz maps over some finite fields,
and lifting, identifying the obtained Belyi functions in $\overline{\QQ}(x)$.
This is explained in \S \ref{ProgramH}. The complete list was generated 
by considering at most 7 finite fields $\FF_p=\ZZ/(p)$ for $p< 960$, 
though eventually we kept the algorithm running for total 100 primes. 
In principle, this does not ensure completeness of the list however.

The deterministic algorithm in \S \ref{ProgramV} takes a branching pattern as an input,
and produces the Belyi coverings with that branching pattern. By using the implied
Heun-to-hypergeometric transformations,  smaller degree algebraic systems for undetermined 
coefficients are obtained than with straightforward methods, 
and with far less {\em parasitic} solutions \cite{kreines}.
The deterministic algorithm produced the same Belyi maps (up to M\"obius transformations)
as the probabilistic one. Completeness of our results is proved assuming correct implementation 
of the deterministic algorithm.

As a practical matter of confidence, the completeness of results is foremost verified by the same output
of the two independent algorithms. In addition, we did a combinatorial search to find 
all minus-4-hyperbolic dessins d'enfant, up to degree $36$.
This gives a verification of a large part ($\approx 95\%$) of relevant branching patterns,
covering $\approx 91\%$ of obtained dessins.
We also compared the list of Belyi functions with the $r$-field in $\RR$ 
with Felikson's list \cite{Felixon} of {\em Coxeter decompositions} in the hyperbolic plane; 
see Appendix \S \ref{sec:coxeter}. This provides enough confidence in completeness of our results.

\subsection{A probabilistic modular method}
\label{ProgramH}

The used probabilistic algorithm is based on the expectation that
a Belyi function with any realization field 
will be properly defined over a $p$-adic field $\QQ_p$ for some prime $p$
among a sequence of considered subsequent or random primes. 
Concretely, suppose that a Belyi function $\varphi(x)$ pulls-back a hypergeometric
equation to Heun's equation with specific parameters $t,q$, and that
respective Schwarz maps of both equations are identified by a constant $K$
as described after Remark \ref{rm:noexist}. 
If $t,q,K$ are elements of a number field $\QQ(\alpha)$, then $\varphi(x)\in\QQ(\alpha)(x)$. 
By Chebotarev's theorem \cite{Wikipedia},
the minimal polynomial for $\alpha$ has a root in $\FF_p$ 
for a positive density of primes $p$. The density is at least $1/D$, where $D$ is the degree
of the number field $\QQ(\alpha)$. For all but finitely many of those primes,
we will have $\alpha\in\QQ_p$ and $t,q,K\in\ZZ_p$ (the $p$-adic integers).
The Belyi function $\varphi(x)$ can be found as follows:
\begin{enumerate}
\item Consider all possible values $\overline{t},\overline{q},\overline{C}\in\FF_p$ 
of the reduction of $t,q,K$ modulo an (eventually) suitable prime $p$;
\item Reconstruct $\varphi(x)$ as a rational function in $\FF_p[x]$ by identifying the Schwarz maps
as described after Remark \ref{rm:noexist}. We need the first $2d+8$ terms in the Schwarz maps
$\psi_0(x)$ and $\psi_1(x)$ to be in $\FF_p$, so $p$ has to be sufficiently large.
For example, if a local exponent difference is $1/3$, 
then we need $p > 3 (2d+8)$ to ensure that local hypergeometric solutions 
have the first $2d+8$ coefficients in $\ZZ_p$. For degree 60 coverings, 
the starting prime was $907>7(2\cdot60+8)$.
\item Use Hensel lifting to obtain an expression of $\varphi(x)$ in $\QQ_p$;
\item Use LLL techniques to compute minimal polynomials of its coefficients, 
thus reconstructing $\varphi(x)$ as an element of $\QQ(\alpha)(x)$.
\end{enumerate}

Our strategy is as follows. For each branching pattern of Tables 2.3.7--3.4.4,
we run through a sequence of primes $p$ and the possible reduced values
$\overline{t},\overline{q},\overline{C}\in\FF_p[x]$.
For each of the $O(p^2)$ pairs of $\overline{t},\overline{q}$  we have to compute series expansions
for the solutions of $H_0$ and $H_1$.  This is done rapidly using linear recurrences for
coefficients of these solutions; {\sf Maple} has the command {\sf gfun[diffeqtorec]} for getting
the recurrences. We expect $\varphi$ to be in $\FF_p[[x]]$ for suitable primes $p$.
If $\psi_0\circ\varphi$ matches $K \psi_1$ in $\QQ_p[[x]]$, then this poses
certain necessary conditions\footnote{If $\varphi = \lambda x^m + \cdots \in \ZZ_p[[x]]$ is substituted into
$\psi_0 = x^{d_1}(1+a_1 x + a_2 x^2 + \ldots)$, with $a_i \in \QQ_p$, and if the first $a_i \not\in \ZZ_p$ is $a_n$, and if $\psi_1=\lambda^{d_1} x^{d_1 m}(1+b_1x + b_2 x^2 + \ldots)$ is the result of the substitution, 
then the first $b_i \not\in \ZZ_p$ must be $b_{mn}$.} on the $p$-adic valuations of the coefficients of $\psi_0$ and $\psi_1$. We compute the series solutions of $H_0$ and $H_1$ to enough precision so
that we can test these necessary conditions.  This way, many pairs $\overline{t},\overline{q}$ can be discarded, and we typically end up with $O(p^1)$ pairs.  Thus, the rational reconstruction step \refpart{ii}
``only'' needs to be called for $O(p^2)$ combinations of $\overline{t},\overline{q},\overline{K}$.

If we find a $\varphi$ mod $p$, we store it in a file. Another program will Hensel lift it,
apply LLL reconstruction to $\QQ(\alpha)(x)$, and compare with the already computed data base.
Each Belyi map $\varphi$ has a density $\delta_\varphi$ of suitable primes. 
The expectation number of times that the same $\varphi$ will be found is then $100 \cdot \delta_\varphi$.  
Unless the density is tiny, the likelihood that $\varphi$ will be found is very high.
The smallest $\delta_\varphi$ encountered was $1/6$, for the H10--H14 
coverings\footnote{The estimate $\delta_\varphi \geq 1/{\rm deg}\;\QQ(\alpha)$ 
is sharp when $\QQ(\alpha)\supset\QQ$ is a Galois extension. 
This is the case for $\QQ( \zeta_7 )$. 
Higher degree encountered number fields (such as for J28) had significantly higher $\delta_\varphi>1/6$. }
with the realization field $\QQ( \zeta_7)$. 
Most of the table was found after just two primes.  
The first 10 primes took about a week on {\sf Maple}, running on 8 Intel X3210 CPU cores.
Among the 100 primes, each Belyi function was found at least  16 times.

The modular method is quite slow, because $O(p^2)$ combinations of 
$\overline{t},\overline{q},\overline{C}$ have to be inspected for each $p$.
But its advantage is low requirement of computer memory. This means that
the computation can continue for weeks on end, 
without a risk that the computation will halt due to memory problems,
and without human intervention (this is important, because if human intervention 
is needed in any of the steps, then, in a table with hundreds of cases,
a gap would become likely).

%

\subsection{A deterministic algorithm}
\label{ProgramV}

A $(k,\ell,m)$-minus-$4$ Belyi function is determined by a polynomial identity
\begin{equation} \label{eq:belyiabc}
P^\ell\,U=Q^{m}\,V+R^k\,W,
\end{equation}
where $P,Q,R$ are monic polynomials in $\CC[x]$ whose roots 
are  the regular branchings, and $U,V,W$ are polynomials whose roots are exceptional points
with correct multiplicities. The Belyi function is then expressed as
\begin{equation} \label{eq:varphis}
\varphi(x)=\frac{P^\ell \, U}{Q^m V}, \qquad 
1-\varphi(x)=\frac{R^k \, W}{Q^m V}.
\end{equation}
The polynomials $P,Q,R$ should not have multiple roots; $V$ may be monic.
The degrees of the polynomials in (\ref{eq:belyiabc}) are determined by the branching pattern
and the assignment of $x=\infty$.
The most straightforward computational method is to assume undetermined coefficients 
of the polynomials in (\ref{eq:belyiabc}), and solve the resulting system of algebraic equations
between the coefficients. This is not practical for Belyi functions of degree $\ge12$,
mainly because of numerous {\em parasitic} \cite{kreines} solutions where some polynomials 
in (\ref{eq:belyiabc}) have common roots.

A more restrictive set of equations for undetermined coefficients 
can be obtained by differentiating $\varphi(x)$, as 
comprehensively described in \cite[\S 2.1]{BelyiKit}.
In particular, the roots of $\varphi'(x)$ include the branching points 
above $\varphi=1$ with the multiplicities reduced by 1.
A factorized shape of the logarithmic derivative of $\varphi(x)$ and $\varphi(x)-1$ 
must be the following: 
\begin{eqnarray} \label{logd1}
\frac{\varphi'(x)}{\varphi(x)}=h_1 \frac{R^{k-1}\,W}{P\,Q\,F}, \qquad\qquad
\frac{\varphi'(x)}{\varphi(x)-1}=h_2 \frac{P^{\ell-1}\,W}{Q\,R\,F}.
\end{eqnarray}
Here $h_1, h_2$ are constants, and  $F$ is the product of irreducible factors 
of $U\,V\,W$, each to the power 1. On the other hand,
\begin{eqnarray} \label{logd2}
\frac{\varphi'(x)}{\varphi(x)}=\ell\,\frac{P'}{P}+\frac{U'}{U}-m\,\frac{Q'}{Q}-\frac{V'}{V}, \qquad 
\frac{\varphi'(x)}{\varphi(x)-1}=k\,\frac{R'}{R}+\frac{W'}{W}-m\,\frac{Q'}{Q}-\frac{V'}{V}.
\end{eqnarray}
We have thus two expressions for both logarithmic derivatives, 
of $\varphi(x)$ and $\varphi(x)-1$. 
As shown in \cite[\S 2.1]{BelyiKit}, this gives a generally stronger over-determined set 
of algebraic equations, of smaller degree and with less parasitic solutions.
If $k=2$, the polynomial $R$ can be even eliminated symbolically.

To get an even more restrictive system of algebraic equations, we utilize the fact that
our Belyi functions transform hypergeometric equations to Heun equations. 
The method bluntly uses the following lemma.
\begin{lemma} \label{lm:gpback}
Let $\varphi(x)$ be a Belyi map as in $(\ref{eq:varphis})$.
Hypergeometric equation $(\ref{HGE})$ with
\[
A=\frac12\left(1-\frac1k-\frac1{\ell}-\frac1m\right), \quad
B=\frac12\left(1-\frac1k-\frac1{\ell}+\frac1m\right), \quad
C=1-\frac{1}{\ell}.
\]
is transformed to the following differential equation 
under the pull-back transformation
$z\mapsto \varphi(x)$, $y(z)\mapsto \left(Q^{m} V\right)^{\!A} \, Y(\varphi(x))$:
\begin{align*}
& \frac{d^2Y(x)}{dx^2}+
\left(\frac{F'}{F}-\frac{U'}{\ell\,U}-\frac{V'}{m\,V}-\frac{W'}{k\,W}\right) \frac{Y(x)}{dx} + \\
& \! + A \left[ B \left( \frac{ h_1h_2\,P^{\ell-2}R^{k-2}\,U\,W}{Q^2 F^2}
- \frac{m^2Q'{}^2}{Q^2} -\frac{V'{}^2}{V^2} \right)+\frac{mQ''}{Q}+\frac{V''}{V}
+\qquad \right. \nonumber \\
& \qquad \left. +\left(\frac1k+\frac1\ell\right)\!\frac{mQ'V'}{Q\,V}
+\left(\frac{m Q'}{Q}+\frac{V'}{V}\right) \!
\left(\frac{F'}{F}-\frac{U'}{\ell\,U}-\frac{V'}{V}-\frac{W'}{k\,W}\right)
\right]Y(x)=0.
\end{align*}
\end{lemma}
\begin{proof} A lengthy symbolic computation, using (\ref{eq:varphis}) and (\ref{logd2}).
\end{proof}

The transformed equation is to be identified with the target Heun equation,
or (if the roots of $U,V,W$ are not normalized to $x=0,1,t,\infty$) with a Fuchsian equation
with 4 singularities at the roots of $UVW$. The accessory parameter $q$ is a new undetermined variable.
The terms to $dY(x)/dx$ are always identical, but comparison of the terms to $Y(x)$
gives new algebraic equations between the undetermined variables.
If $k=2$, $\ell=3$, not only $R$ but also $P$ can be eliminated symbolically.
The two expressions in (\ref{logd2}) and Lemma \ref{lm:gpback} then give 
a non-linear differential equation for $Q$, with $q$ and the coefficients of $U,V,W$ 
as parametric variables. After substitution of general polynomial expression 
for $Q$, we collect to the powers of $x$ and get a system of algebraic equation
for undetermined coefficients. This is explained more thoroughly in \cite[\S 2.2]{BelyiKit}.
The logarithmic derivative ansatz and Lemma \ref{lm:gpback} do not use the location $\varphi=1$
of the third fiber, hence the polynomials $U,V,W$ can be assumed to be monic as well.
Then the Belyi function $\varphi(x)$ has to be adjusted by a constant multiple at the latest stage.
In most cases, all but 3 variables\footnote{Or all except 4 weighted homogeneous variables, 
if the scaling transformations $x\to\alpha x$ are left to act. The Schwarz maps of \S \ref{ProgramH}
are determined by 3 values as well: the location parameter, the accessory parameter,
and the constant multiple.} 
are eliminated linearly, leaving only so many variables for hard Gr\"obner basis computations.
Our implementation \cite[{\sf ComputeBelyi.mpl}]{HeunURL} for {\sf Maple 15}
computes the degree 60 Belyi maps in 110s, 
the Galois orbit J28 in 274s, and the orbit pair H11, J26 in 830s.

\begin{example} \label{ex:deg54}
Consider computation of degree $54$ Belyi functions with the branching fractions $1/7,1/7,1/7,2/7$.
We assign 
the branching fraction $2/7$ to $x=\infty$, so that $U=W=1$.
The polynomials $P,Q,R,V$ are assumed to be monic, without multiple roots, 
of degree $18$, $7$, $27$, $3$ respectively.
If we would assume $V=x(x-1)(x-t)$, the Heun equation would have \mbox{$a=9/14$},
$b=13/14$ and $c=d=6/7$. To avoid increase of the moduli field, we rather assume $V=x^3+v_2x+v_3$.
Here the $x^2$ term is zero-ed by a translation $x\to x+\beta$, 
so that only scaling M\"obius transformations $x\to \alpha\,x$ are left to act.
The transformed Fuchsian equation must have the following term to $Y(x)$: $ab(x-q)/V$.
The logarithmic derivative ansatz gives
\[
2R=3P'QV-7PQ'V-PQV',\qquad 2P^2=2QR'V-7Q'RV-QRV',
\]
while Lemma \ref{lm:gpback} gives
\[
\frac{13}{84}\left(\frac{4P}{Q^2V^2}-\frac{49Q'{}^2}{Q^2}-\frac{V'{}^2}{V^2}\right)
+\frac{7Q''}{Q}+\frac{V''}{V}+\frac{35Q'V'}{6QV}=\frac{351\,(x-q)}{7V}.
\]
Symbolic elimination of $R,P$ on {\sf Maple} gives the following differential expression:
\begin{align*}
&\frac{7Q''''}{15Q}+\frac{7Q'''}{3Q}\left(\frac{V'}{V}-\frac{Q'}{Q}\right)+\frac{(7Q'')^2}{26Q^2}
+\frac{Q''V'}{QV}\left(\frac{13V'}{7V}-\frac{35Q'}{13Q}\right) \\
& \! +\frac{3Q''}{7QV}\left(115q-\frac{1033}{13}x\right)
+\frac{Q'{}^2}{7Q^2V}\left(\frac32(163x-247q)+\frac{16V'{}^2}{13V}\right)-\frac{13V'}{2V^2}\\
& \! +\frac{3Q'}{2QV}\left(\left(\frac{183}7q-\frac{241}{13}x\right)\frac{V'}{V}+\frac{67}{21}\right)
+\frac{18}{V^2}\left(2x-\frac{13}5q\right)\left(\frac{46}{13}x-\frac{27}7q\right)=0.
\end{align*}
Here the values $V''=6x$, $V'''=6$, $V''''=0$ are simplified.  
Substituting the explicit $V$, $Q=x^7+c_1x^6+\ldots+c_6x+c_7$ 
and clearing the denominator, we obtain a polynomial expression of degree $15$ in $x$. 
The leading term gives \mbox{$q=-5c_1/52$}. 
The next term gives nothing new (as follows from \cite[Lemma\ 2.1]{BelyiKit}). 
But the next $5$ equations allow subsequent elimination of $c_3,c_4,c_5,c_6,c_7$ in terms of
$c_1,c_2,v_2,v_3$. The $4$ remaining variables are weighted-homogeneous, with the weights $1, 2, 2, 3$. 
Elimination of $v_2,v_3$ using the other $10$ equations is done with the Gr\"obner basis routine of
{\sf Maple 15} in about 35s (on a PC with 2.66GHz Intel Core Duo). 
The algebraic system has $4$ Galois orbits of solutions, $3$ of them 
parasitic\footnote{The parasitic solutions are: the degree 18 coverings mentioned in footnote 3;
a degree 10 covering with the branching pattern $5\,[2]=3\,[2]+4\,[1]=7+2+1$; 
and the non-cyclic cubic Belyi covering. In all cases, the simplification of the numerator 
and the denominator of $\varphi(x)$ is by a linear polynomial to the maximal power (36, 44 or 51).}. 
The proper solution has the label D28. We can take
\begin{align*}
V=x^3-4899x-370078,\quad 
Q=x^7+28x^6+\frac{29063265}{512}x^5+\ldots. 
\end{align*}
The expression for $\varphi(x)$ is long. We looked for an optimizing M\"obius transformation.
The bit size of $\varphi(x)$ is reduced by the factor $\approx 2.26$
after the M\"obius substitution \mbox{$x\mapsto (241x-212)/(x+4)$}.
Then
\begin{equation*}
\varphi(x)=\frac{P^3}{864(x-4)(3x^2+1)(x+4)^2Q^7},
\end{equation*}
where $Q=3x^7-7x^6-14x^5-98x^4+147x^3-7x^2+56x+16$ and
\begin{align*}
P=&47x^{18}-2028x^{17}+5502x^{16}+54540x^{15}-263535x^{14}-32592x^{13}+2249268x^{12}\\
&-3436872x^{11}+14145x^{10}-1425900x^9-8774370x^8-1715652x^7-10594017x^6\\
&+2223144x^5-5284080x^4+1638144x^3-1306368x^2+239616x-135168.
\end{align*}
\end{example}

\section{Moduli fields and obstruction conics}
\label{sec:conic}

Particularly interesting are Belyi functions with moduli field issues. 
Here we present these instances among the minus-4-hyperbolic functions.
At the same time, we briefly recall cohomological and conic obstructions
on realization fields of Belyi functions, give a straightforward characterization
of the obstruction conic (in Lemma \ref{lm:conic}) that applies to our cases,
and express a few Belyi coverings as functions on the obstruction conics.
Further computational and geometrical details are considered in \cite[\S 4]{BelyiKit}.

Let $\OO$ denote the group of M\"obius transformations:
\[ \OO = \left\{ \frac{ax+b}{cx+d} \, | \, a,b,c,d \in \overline{\QQ} {\rm \ with \ }ad-bc \neq 0 \right\} 
\cong {\rm Aut}(\overline{\QQ}(x)/\overline{\QQ}). \]
Two rational functions $\varphi_1, \varphi_2 \in \overline{\QQ}(x)$ are called 
{\em M\"obius-equivalent}, denoted $\varphi_1 \sim \varphi_2$, 
if there exists $\mu \in \OO$ with $\varphi_1 \circ \mu = \varphi_2$.
A {\em realization field} of a Belyi covering $\varphi$ is a number field over which
a M\"obius equivalent function $\varphi\circ\mu$ is defined. 
The $r$-field from Definition~\ref{df:jtfield2} is such a field,
but often not of minimal degree. 
\begin{definition} \label{def:moduli}
Let $\varphi \in \overline{\QQ}(x)$ be a Belyi function. 
The {\em moduli field} $M_{\varphi}$ is the fixed field of 
$\{ \sigma\in{\rm Gal}(\overline{\QQ}/\QQ) \,|\, \varphi \sim \sigma(\varphi) \}$.
\end{definition}
Clearly $M_{\varphi} \subseteq K_{\varphi}$ for any 
realization field $K_\varphi$.  
The moduli field is known to be equal the intersection of the realization fields of $\varphi$.
Two Belyi functions are M\"obius-equivalent if and only 
if they have the same dessin d'enfant up to homotopy.
Thus, the moduli field of a dessin d'enfant is well defined.
The number of different dessins (up to homotopy) 
in a Galois orbit is equal to the degree of the moduli field.

For each Belyi function $\varphi$ in our list, 
we determined its moduli field and realization fields.
Among the minus-4-hyperbolic Belyi functions,
there are 14 Galois orbits for which the moduli field not a realization field. 
They are given in Table \ref{conic}. The realization fields are then determined
by an {\em obstruction conic}, 
as explained in \S \ref{sec:obstructf}.
The last two columns characterize the conics.

\begin{vchtable} 
\begin{tabular}{@{}llclcll@{}} 
\hline Id & branching\!\! & $d$ & $[k \ell m]$ & Moduli &
Obstruction & Bad  \\
& fractions &    &    & field & conic & primes \\  \hline
B12 & $\fr17,\fr17,\fr37,\fr37$ & 36 & [237] & $\QQ$ & $u^2+v^2+7$
& $7,\infty$ \\ 
C6 & $\fr13,\fr13,\fr27,\fr27$ & 32 & [237] & $\QQ$ & $u^2+v^2+1$ &
$2,\infty$ \\ 
C30 & $\fr12,\fr12,\fr14,\fr14$ & 10 & [245] & $\QQ$ & $u^2+2v^2+5$
& $5,\infty$ \\
D45 & $\fr14,\fr14,\fr14,\fr14$ & 20 & [245] & $\QQ$ & $u^2+2v^2+5$
& $5,\infty$ \\
F1 & $\fr13,\fr13,\fr13,\fr13$ & 8 & [334] & $\QQ(\sqrt{2})$ &
$u^2 + 3v^2 + \sqrt{2} - 1$ & $3,\infty$ \\
F4 & $\fr12,\fr12,\fr18,\fr18$ & 18 & [238] & $\QQ(\sqrt{2})$ &
$u^2+v^2+1$ & $\infty,\infty$ \\
F6 &  $\fr18,\fr18,\fr18,\fr18$ & 36 & [238]  & $\QQ(\sqrt{2})$ &
$u^2+v^2+1$ & $\infty,\infty$ \\
F11 &  $\fr15,\fr15,\fr15,\fr15$ & 12 & [255] & $\QQ(\sqrt{5})$ &
$u^2+2v^2+\sqrt5$ & $5,\infty$ \\
H1 &  $\fr13,\fr13,\fr19,\fr19$ & 20 & [239]  & $\QQ(\mbox{Re}\,\zeta_9)$
 & $u^2+v^2+ \mbox{Re}\,\zeta_9$ & $\infty,\infty$ \\
H10 & $\fr12,\fr12,\fr17,\fr17$ & 30 & [237]  & $\QQ(\mbox{Re}\,\zeta_7)$
 & $u^2+v^2-\mbox{Re}\,\zeta_7$ & $\infty,\infty$ \\
H11 & $\fr13,\fr13,\fr17,\fr17$ & 44 & [237]  & $\QQ(\mbox{Re}\,\zeta_7)$
 & $u^2+v^2-\mbox{Re}\,\zeta_7$ & $\infty,\infty$ \\
H12 & $\fr12,\fr12,\fr13,\fr13$ & 14 & [237]  & $\QQ(\mbox{Re}\,\zeta_7)$
 & $u^2+v^2-\mbox{Re}\,\zeta_7$ & $\infty,\infty$ \\
H13 & $\fr13,\fr13,\fr13,\fr13$ & 28 & [237]  & $\QQ(\mbox{Re}\,\zeta_7)$
 & $u^2+v^2-\mbox{Re}\,\zeta_7$ & $\infty,\infty$ \\
H14 & $\fr17,\fr17,\fr17,\fr17$ & 60 & [237]  & $\QQ(\mbox{Re}\,\zeta_7)$
 & $u^2+v^2-\mbox{Re}\,\zeta_7$ & $\infty,\infty$ \\
\hline
\end{tabular} 
\vchcaption{Belyi functions with an obstruction.}
\label{conic}
\end{vchtable}

The moduli fields are computed directly from Definition \ref{def:moduli}
by checking which Galois conjugates of $\varphi$ are M\"obius-equivalent to $\varphi$. 
The computed Belyi functions $\varphi$ 
always had $[K_{\varphi}:\QQ(j)] \leq 2$, where $\QQ(j)$ is the $j$-field. But 
$\QQ(j) \subseteq M_{\varphi} \subseteq K_{\varphi}.$ 
Therefore $M_{\varphi} = K_{\varphi}$ if $K_{\varphi} = \QQ(j)$.
For those $K_{\varphi} \neq \QQ(j)$, let $\sigma$ be the non-trivial element 
of ${\rm Gal}(K_{\varphi} / \QQ(j))$.
Our moduli fields are then determined by just checking whether $\varphi \sim \sigma(\varphi)$.

The dessins d'enfant of most of the Belyi maps of Table \ref{conic}
are depicted in Figure \ref{fg:mdessins}. The other Galois orbits with
obstructed dessins are found in Figures \ref{fg:deg60s} and \ref{fg:vdessins}.
The interesting questions whether a dessin has a moduli field 
$\subset \RR$, and if so, does it have a realization over $\RR$, are considered in 
\cite{CouveignesGranboulan}. 
Although all moduli fields in the obstructed cases are real,
not all their dessins have a reflection symmetry (i.e., have a realization over $\RR$). 
Rather, their complex conjugates are equivalent to the original up to homotopy that
permutes the cells, reflecting a non-trivial M\"obius equivalence.
The number of these skew-symmetric dessins depends on the number of bad $\infty$-primes
shown in the last column of Table \ref{conic}.
The moduli field for  H1, H10--H14 has three infinite primes, but only two of them are bad.
Therefore one dessin in those orbits has a reflection symmetry,
and the other two are skew symmetric. Likewise, F1 and F11 each have one dessin 
with $\RR$-realization and one dessin without.

\input{Mdessins.tex}

\subsection{Obstructions on realization fields}
\label{sec:obstructf}

If the moduli field $M_\varphi$ is not a realization field, 
the realization fields are determined by a {\em conic obstruction}.
For each of the cases of Table \ref{conic}, the realization fields are those extensions of $M_\varphi$
that have a rational point on the conic curves given in the sixth column.

For $\varphi \in \overline{\QQ}(x)$, 
let us denote $\Gamma_\varphi=\mbox{Gal}(\overline{\QQ}/M_\varphi)$.
Let
\[ \OO_{\varphi} = \{ \mu \in \OO | \varphi \circ \mu 
= \varphi\} \cong {\rm Aut}(\overline{\QQ}(x)/\overline{\QQ}(\varphi)), 
\]
be the group of M\"obius automorphisms of $\varphi$. 
For any $\sigma\in\Gamma_\varphi$
we have $|\OO_\varphi|$ choices for $\mu\in\OO$ in $\sigma(\varphi)=\varphi\circ\mu$.
If for each $\sigma\in\Gamma_\varphi$  we can choose such $\mu_\sigma\in\OO$  
so that $\mu_\sigma\circ\sigma(\mu_\rho)=\mu_{\sigma\rho}$ for any $\sigma,\rho\in\Gamma_\varphi$, 
then we have a cocycle of Galois cohomology \cite{SerreLF} 
representing an element of $H^1(\Gamma_\varphi, \OO)$.
This choice is certainly possible if $|\OO_\varphi|=1$. The realization fields $L$ are then those
which are mapped to the identity in $H^1({\rm Gal}( \overline{\QQ} / L), \OO)$.
As recalled in \cite{Filimonenkov}, the elements of $H^1( \Gamma_\varphi, \OO)$
are in one-to-one correspondence with isomorphism classes of conic curves over $M_\varphi$.
This is a special case of the construction in \cite[Ch. XIV]{SerreLF}. 

In turn, a conic  is determined 
up to birational equivalence over $M_\varphi$ by the primes $\mathfrak{p}$ of bad reduction. 
The number of bad primes is always even. 
The bad primes are precisely those for which $\varphi$ 
has no realization over the completion of $M_{\varphi}$ at $\mathfrak{p}$.
The completion at a real prime 
is isomorphic to $\RR$. 
Notice that the conics for C6 and F4 look the same $u^2+v^2+1=0$ but over different moduli fields.
In particular, their sets of bad primes differ.

In \cite[\S 7]{Filimonenkov} it is proved that if $\varphi(x)\in\overline{\QQ}(x)$ has a Galois cocycle,
then there is a realization over a quadratic extension of the moduli field $M_\varphi$.
These realizations are straightforward to obtain for Belyi functions
with exactly two points of some branching order in the same fiber $f\in\{0,1,\infty\}$.
Designating those two points as $x=\infty$, $x=0$ extends the moduli field at most quadratically.
This applies to all our examples except D45, F6, H13, H14.

Suppose now $\varphi(x)\in M_\varphi(\sqrt{A})$ for $A\in M_\varphi$, and let $\mu(x)\in\OO$ 
be the cocycle representative of those Galois elements that conjugate $\sqrt{A}\to-\sqrt{A}$. 
With $x=\infty$, $x=0$ set as just above, the possible M\"obius transformations
are $x\mapsto -x$ or $x\mapsto B/x$. In the former case, the quadratic extension disappears after 
the scaling $x\mapsto \sqrt{A}\,x$. 
\begin{lemma} \label{lm:conic}
Suppose that we have a Belyi function $\varphi(x)\in M_\varphi(\sqrt{A})$ where $M_\varphi$ 
is the moduli field. Suppose that there is a Galois cocycle that sends the Galois elements that
conjugate \mbox{$\sqrt{A}\to-\sqrt{A}$} to $x\mapsto B/x$ for $B\in M_\varphi$. 
Then the obstruction conic is isomorphic to \mbox{$u^2=Av^2+B$.}
\end{lemma}
\begin{proof}
The functions
\begin{equation} \label{eq:conicp}
u=\frac12\left(x+\frac{B}x\right), \qquad 
v=\frac1{2\sqrt{A}}\left(x-\frac{B}x\right).
\end{equation}
are invariant under the Galois action, hence they are $M_\varphi$-rational
functions on the obstruction conic. They are related by $u^2=Av^2+B$.
\end{proof}

\begin{example}
The branching pattern for  \AJ{C30} is $2\,[5]=2\,[4]+1+1=4\,[2]+1+1$. 
The moduli field is $\QQ$. Here is a realization over $\QQ(\sqrt{-3})$,
with the points of branching order $5$ assigned as $x=\infty$, $x=0$:
\begin{equation} \label{eq:c30}
\varphi(x) = \frac{2(x^2+5x-5)^4\big((x^2+5)\sqrt{-3} - 3x^2 - 60x + 15\big)}{(12 x)^5}. 
\end{equation}
We have $|\OO_\varphi|=1$, since the numerator of $\varphi(y)-\varphi(x)$ has only one
linear factor $y-x$. Let $\sigma: \sqrt{-3} \mapsto -\sqrt{-3}$ denote the non-trivial element 
of ${\rm Gal}(\QQ(\sqrt{-3})/\QQ)$. The numerator of $\varphi(y)-\sigma(\varphi(x))$ has
a linear factor $xy+5$, giving the M\"obius transformation $\mu(x)=-5/x$ for 
$\sigma(\varphi) = \varphi \circ \mu$. By Lemma $\ref{lm:conic}$, the obstruction conic
is isomorphic to $C:u^2+3v^2+5=0$. 
We can express $\varphi$ as a function on this conic 
by writing $\varphi(x)$ as a product of Laurent polynomials and substituting
\[
x=u+v\sqrt{-3}, \qquad \frac1x=\frac{-u+v\sqrt{-3}}{5}.
\]
The expression is
\begin{equation} \label{eq:conic30}
\varphi=\left(\frac{u}6+\frac{5}{12}\right)^{\!4} (v-u-10) \in \QQ(u,v)/( u^2 + 3 v^2 + 5 ).
\end{equation}
A point $(u_0,v_0)\in C$ defined over some number field $L\supset\QQ$ 
gives a parametrization $\lambda:\PP^1\to C$ by the lines passing  through $(u_0,v_0)$. 
The composition $\varphi\circ\lambda$ gives then a realization of $\varphi$ over $L$.
Formula $(\ref{eq:conicp})$ gives one such parametrization.
The conic $C$ is isomorphic to the conic given by $u^2+2v^2+5=0$ as they have the same set of bad primes. A projective isomorphism is $(u:v:1)\mapsto \big(\frac12(u-5):v:\frac12(u+1)\big)$. 
Its computation is explained in \cite[\S 3.5]{BelyiKit}.
\end{example}

The obstructed cases without a cocycle are the following: D45, F6, H13, H14.
These are exactly the cases of Table \ref{conic} with $|\OO_{\varphi}| > 1$.
In fact, $|\OO_{\varphi}|  = 2$ for these Galois orbits\footnote{
Non-existence of a cocycle defined over $\RR$ can be shown geometrically by using the criterium 
in \cite[Theorem 2]{CouveignesGranboulan}. Each of the dessins for D45, F6, H13, H14 without 
a reflection symmetry has a tetrahedral carcass (obtained by taking out some faces 
around 4 exceptional points or cells). A pair of opposite tetrahedron edges $e,f$ relate 
to the exceptional faces differently than the other tetrahedron edges. 
If we assume the tetrahedron to have equal straight edges, the dessin symmetry is rotation by $\pi$ 
around the axis connecting the midpoints of $e,f$. The complex conjugation is realized
by permutations $w,w^{-1}$ of half-edges (connecting black vertices and white midpoints)
that swap $e,f$ and cyclically permute the other $4$ tetrahedron edges.
The order of $w$ is thus $4$, but we must have $w^2=$id for a cocycle.}.
To get explicit realizations for these Belyi functions, 
we suggest to take their quotients by $\OO_\varphi$.
The quotients are C30, F4, H12, H10, respectively.
The smaller coverings do have a cocycle and parametrizations by obstruction conics.
As shown in \cite[\S 3.2]{BelyiKit}, the realization fields of C30, D45 
(and of F4, F6; or H12, H13; or H10, H14) are the same. 
In particular, each realization $\varphi\circ\lambda$ of C30
is composed with a quadratic covering (to D45) defined over the same field. 
The quadratic covering composes with the conic paramaterization $\lambda$,
not with the conic\footnote{In fact \cite{Macrae}, a conic defined 
over a field $K$ without a $K$-rational point cannot have quadratic coverings defined over $K$.}
realization $\varphi$. 

In the obstructed cases with a cocycle, realization of a Belyi covering as a function 
on the obstruction conic is a specially compact expression of the Belyi covering,
as demonstrated in (\ref{eq:conic30}). Here are two more examples. For C6, we have
\begin{equation}  
\varphi=\frac{(u-5)(u-1)^3(v^4+18v^2+8v(v-58)(u-4)-3403)^3}{3456(v(u-4)-13)^7} 
\end{equation}
on the conic $u^2+v^2+2=0$. The conic is isomorphic to $u^2+v^2+1=0$ by $(u,v)\mapsto(u+v,u-v)$.
For F11, we use the expression in \cite{Filimonenkov} and obtain
\begin{equation} 
\varphi= -\frac{(u+\sqrt5+2)^5\big(u-2(3-\sqrt5)v-5\sqrt5\big)}
{(u-\sqrt5-2)^5\big(u-2(3-\sqrt5)v+5\sqrt5\big)}.
\end{equation}
on the conic $u^2+2v^2+\sqrt5=0$. After the substitution 
\[
(u,v)\mapsto \left( (\sqrt5+2)\,\frac{u+1}{u-1}, \frac{(3\sqrt5+1)u+5\sqrt5v-2\sqrt5+1}{(3-\sqrt5)(u-1)}\right)
\]
the expression for F11 becomes $u^5(1-u-v)/v$, though the conic equation then becomes complicated.

\begin{vchfigure}\begin{picture}(314,94)(0,-1) \thicklines{\thinlines
\put(0,73){F1} \put(23,91){\circle*3} \put(53,59){\circle*3} \put(53,59){\line(-1,0){29}} \put(23,91){\line(1,0){29}} \put(53,91){\circle3} \put(23,59){\circle3} \put(23,60){\line(0,1){31}} \put(53,90){\line(0,-1){31}} \put(53,59){\line(1,0){18}} \put(72,59){\circle3} \put(54,91){\line(1,0){18}} \put(72,91){\circle*3} \put(23,91){\line(6,-5){12}} \put(36,80){\circle3} \put(24,60){\line(6,5){12}} \put(36,70){\circle*3} \put(142,74){\line(1,-1){12}} \put(142,76){\line(1,1){12}} \put(154,88){\circle*3} \put(154,62){\circle*3} \put(141,75){\circle3} \put(140,75){\line(-1,0){15}} \put(125,75){\circle*3} \qbezier(124,76)(112,109)(88,76) \qbezier(124,74)(112,41)(88,74) \put(87,75){\circle3} \put(99,75){\circle*3} \put(99,75){\line(-1,0){11}} \put(99,75){\line(5,-3){10}} \put(99,75){\line(5,3){10}} \put(110,81){\circle3} \put(110,69){\circle3} 
\put(178,73){F7} \put(205,91){\circle*3} \put(235,59){\circle*3} \put(235,59){\line(-1,0){29}} \put(205,91){\line(1,0){29}} \put(235,91){\circle3} \put(205,59){\circle3} \put(235,90){\line(0,-1){31}} \put(235,59){\line(1,0){18}} \put(254,59){\circle3} \put(236,91){\line(1,0){18}} \put(254,91){\circle*3} \qbezier(206,60)(217,75)(206,90) \qbezier(204,60)(193,75)(204,90) \put(349,75){\circle3} \put(348,75){\line(-1,0){14}} \put(334,75){\circle*3} \qbezier(333,76)(323,91)(313,76) \qbezier(333,74)(323,59)(313,74) \put(312,75){\circle3} \put(310,75){\line(-1,0){14}} \put(296,75){\circle*3} \qbezier(295,76)(281,98)(265,76) \qbezier(295,74)(281,52)(265,74) \put(280,75){\circle*3} \put(280,75){\line(-1,0){14}} \put(265,75){\circle3} \put(209,36){\circle*3} \put(209,36){\line(1,0){39}} \put(249,36){\circle3} \put(249,36){\line(0,-1){31}} \put(209,36){\line(0,-1){31}} \put(209,4){\circle3} \qbezier(210,35)(232,20)(210,5) \qbezier(248,35)(226,20)(248,5) \qbezier(208,35)(186,20)(208,5) \qbezier(250,35)(272,20)(250,5) \put(248,35){\line(-2,-5){4}} \put(210,35){\line(2,-5){4}} \put(214,24){\circle3} \put(249,4){\circle*3} \put(249,4){\line(-1,0){39}} \put(208,5){\line(-2,5){4}} \put(250,5){\line(2,5){4}} \put(254,16){\circle3} \put(204,15){\circle*3} \put(244,25){\circle*3} \put(282,20){\circle3} \qbezier(283,21)(292,41)(305,21) \qbezier(283,19)(292,-1)(305,19) \put(306,20){\circle*3} \put(306,20){\line(-5,-2){10}} \put(306,20){\line(-5,2){10}} \put(295,24){\circle3} \put(295,16){\circle3} \put(283,20){\line(1,0){21}} \qbezier(282,21)(293,60)(318,21) \qbezier(282,19)(293,-20)(318,19) \put(319,20){\circle*3} \put(319,20){\line(1,0){26}} \put(346,20){\circle3} \qbezier(320,21)(334,48)(346,21) \qbezier(320,19)(334,-8)(346,19) \put(336,26){\circle*3} \put(336,14){\circle*3} \put(336,14){\line(2,1){10}} \put(336,26){\line(2,-1){10}} }
\put(0,18){F11} \put(40,20){\line(-4,-3){12}} \put(28,11){\circle*3} \qbezier(40,20)(21,13)(25,25) \qbezier(40,20)(29,37)(25,25) \qbezier(40,20)(62,54)(76,20) \qbezier(40,20)(62,-14)(76,20) \qbezier(76,20)(57,27)(61,15) \qbezier(76,20)(65,3)(61,15) \put(76,20){\line(-4,3){12}} \put(64,29){\circle*3} \qbezier(112,20)(90,22)(100,32) \qbezier(112,20)(110,42)(100,32) \qbezier(112,20)(90,18)(100,8) \qbezier(112,20)(110,-2)(100,8) \put(112,20){\line(1,0){12}} \qbezier(124,20)(150,51)(150,20) \qbezier(124,20)(150,-11)(150,20) \put(124,20){\line(5,-2){15}} \put(124,20){\line(5,2){15}} \put(139,26){\circle*3} \put(139,14){\circle*3} 
\put(172,18){or} 
\end{picture}\vchcaption{Ambiguous cases of moduli fields}\label{fg:vdessins}\end{vchfigure}

\subsection{Ambiguous moduli fields}
\label{sec:f7}

The moduli field for the Galois orbit F7 is $M=\QQ(\sqrt{3+6\sqrt2})$ by the standard definitions.
However, the branching pattern is $[4]+2+1=2\,[3]+1=2\,[3]+1$ has two symmetric fibers $2\,[3]+1$.
The conjugation of $M\supset\QQ(\sqrt2)$ permutes the two fibers, so the derivative of a Belyi function 
for F7 has a compact expression:
\begin{eqnarray}  
\varphi'(x)= 
\frac{\big(2x^2+x+3-2\sqrt2(x+1)\big)^2\big(8x^2-12x+4+\sqrt2(2x-3)\big)^2}
{\sqrt{3+6\sqrt2}\,(753-531\sqrt2)\,x^3 (x-1)^2}.
\end{eqnarray}
The function $\sqrt{3+6\sqrt2}\,\big(2\varphi(x)-1\big)$ is defined over $\QQ(\sqrt2)$ and branches only
above $z=\infty$ and $z=\pm\sqrt{3+6\sqrt2}$. 
Defining a Belyi function by requiring branching in any (at most) 3 fibers,  not specifically $\{0,1,\infty\}$,
would make no geometrical difference because of M\"obius transformations on $\PP_z^1$.
But evidently, there are arithmetic consequences for moduli and realization fields.
The number of dessins for F7 is 2 or 4 depending of whether the dessins are counted up to
M\"obius equivalence on $\PP^1_z$ or not. 
Figure \ref{fg:vdessins} depicts two of the dessins for F7. The other two are obtained by swapping 
the color labeling of black and white vertices. 
If the symmetric fibers are put at $z=0$, $z=1$, the transformation $z\mapsto 1-z$ 
swaps the two symmetric fibers and changes the sign of $\sqrt{3+6\sqrt2}$. 
One conjugation of $\sqrt2$ gives $\sqrt{3+6\sqrt2}\in\RR$,
hence one of the dessins is real. 

Most remarkably, the Galois orbits F1 and F11 demonstrate a 
mix of  a conic obstruction and ambiguous moduli field. Their realization fields are obstructed 
by  the conics in Table \ref{conic} if we insist in having the branching fibers at $\{0,1,\infty\}$.
But M\"obius transformations on $\PP_z^1$ of their Belyi functions can be expressed
over the moduli fields. 
Reflecting this, the first dessin of F1 in Figure \ref{fg:vdessins} is symmetric if vertex coloring is ignored,
but the black and white vertices are interchanged by the complex conjugation.
A Belyi function for F1 is
\begin{equation}
\int 
\frac{8(5+3\sqrt2)\big(x^4+4x^2+6+\sqrt2(14x^2+4)\big)^2}
{3\sqrt{-6\sqrt2}\left(x^2-2\sqrt2x-2-\sqrt2\right)^5} \, dx,
\end{equation}
with a proper integration constant setting the branching fibers $z=0$, \mbox{$z=1$}.
But an expression in $\QQ(\sqrt{2})(x)$ is obtained after multiplication by $\sqrt{-6\sqrt2}$
and loosening the integration constant.
The dessins for F11 are drown in Figure \ref{fg:vdessins} in two variations: 
first compactly, by hiding white vertices of order 2;
then assigning the black and white vertices to represent points of order 5 
to show the fiber interchanging symmetry. A Belyi function for F11 is 
\begin{equation}
\int 
\frac{\sqrt{-2\sqrt5}\,(10+8\sqrt5)\,\big(x^4+(72\sqrt5-156)x^2+4\big)^4}
{25\left(x^6-22x^5+306x^4-840x^3-612x^2-88x-8+2\sqrt{5}xP 
\right)^3} \, dx,
\end{equation}
where $P=5x^4-68x^3+188x^2+136x+20$. The Galois orbits F11, G47, A19 with 
the same branching pattern are considered in \cite{Filimonenkov}, \cite[Example 5.7]{zvonkin10},
though the consequence of auto-duality for F11 is not noticed.

Examples of coverings with this dual interpretation of the moduli field 
are given 
in \cite{Pharamond}. One example of Pharamond is the branching pattern
$4+2+1=4+2+1=4+2+1$ with two Galois orbits.
One moduli field  is $\QQ(\sqrt{-1-2\sqrt2})$, though rational functions
can be expressed over $\QQ(\sqrt2)$ if the fiber location is not fixed. 
A function for the other orbit can be similarly written over $\QQ(\sqrt{-6})$, 
while the moduli field is of degree 12, 
obtained by adjoining the roots of the polynomial $z^3-z^2+(3+\sqrt{-6}) z-3$.

\appendix

\section{Appendix: Sorting criteria}
\label{sc:criteria}

In \S \ref{sec:ajnumbers}, the minus-4-hyperbolic Belyi functions
were grouped into 10 classes A--J.
We order the Belyi functions inside those classes by the following 
criteria:
\begin{itemize} \itemseparate
\item[\refpart{a}] the first criterium is the $j$-invariant;
\item[\refpart{b}] the second criterium is the branching
fractions\footnote{The first two criteria establish that our list is basically sorted
by Heun equations. To identify the Heun equations, invariants describing accessory 
parameters should be added \cite[\S D]{HeunForm}.};
\item[\refpart{c}] the last criterium is the degree of the covering. 
\end{itemize}
The sort of $j$-invariants lexicographically adheres to the following criteria:
\begin{itemize} \itemseparate
\item[\refpart{a1}] the $j$-field;
\item[\refpart{a2}] the $t$-field;
\item[\refpart{a3}] the leading coefficient of the minimal polynomial in $\ZZ[x]$ 
for the $j$-invariant. 
\end{itemize}
The order of $j$-fields and $t$-fields is settled by the following criteria:
\begin{itemize} \itemseparate
\item[\refpart{f1}] the field degree;
\item[\refpart{f2}] if the field is a quadratic extension of $\QQ$ then:
\begin{itemize} \itemseparate
\item[\refpart{f1a}] real quadratic fields have precedence over $\QQ(\sqrt{a})$ with $a<0$;
\item[\refpart{f1b}] the fields  $\QQ(\sqrt{a})$ with the same sign of $a$ are ordered 
by the increasing order of $|a|$.
\end{itemize}
\item[\refpart{f3}] if the field is of higher degree, 
then the criterium is the field discriminant.
\end{itemize}
The integers in \refpart{a3} and \refpart{f3} are ordered as follows:
\begin{itemize} \itemseparate
\item[\refpart{i1}] the product of the primes dividing the integer;
\item[\refpart{i2}] the absolute value.
\end{itemize}
The numbers in \refpart{i1}, \refpart{i2}, \refpart{f1b} and \refpart{c} are ordered
in the increasing order.
The tuples of branching fractions are ordered as follows
\begin{itemize} \itemseparate
\item[\refpart{b1}] in each tuple, the four branching fractions are ordered
in the increasing order of their denominators, then secondarily the numerators.
\item[\refpart{b2}] the tuples are compared lexicographically, from their first elements,
and the elements are matched first by their denominators then numerators.
\end{itemize}
These criteria break all ties in our list of Belyi functions. Due to \refpart{i1},
the fields or $t$-values that ramify or degenerate modulo
the same set of primes are placed next to each other. 
The leading coefficient in \refpart{a3} gives information about the primes 
where the covering is ramified. In particular, for $j\in\QQ$ the leading coefficient 
is the denominator of $j$. 

\section{Appendix: The A-J tables}
\label{sec:ajtables}

The following pages display tables of Galois orbits of minus-4-hyperbolic Belyi functions,
grouped as specified in \S \ref{sec:ajnumbers} and ordered by the criteria in \S \ref{sc:criteria}.
All tables contain the following columns:
\begin{itemize} \itemseparate
\item Id: the label from A1 to J28;
\item Branching frac.: the branching fractions of a minus-4-hyperbolic function;
\item $d$: the degree of a Belyi function;
\item $[k\ell m]$: the values of $k,\ell,m$ written compactly. For $k=2,\ell=3,m\ge10$,
only the value of $m\in[10,14]$ is given.
\item Monodromy/comp.~or Mndr/cmp.: The monodromy group $G=\ldots$ is given for indecomposable coverings, and compositions are indicated otherwise. The composition notation is explained 
in \S \ref{sec:compose}.
\end{itemize}
Other occasional columns:
\begin{itemize} \itemseparate
\item $j$-invariant: given if it is in $\QQ\setminus\{1728\}$, in a factorized form;
\item $d_j$: the degree of the $j$-field (in tables I, J);
\item disc $\QQ(j)$, disc $\QQ(t)$: the field discriminants. 
If the extension $\QQ(t)\supset \QQ(j)$ is of degree 6, the degree of the the $t$-field
is indicated in the disc $\QQ(t)$ column in a small underlined font.
\item $\sqrt{\;}\,$: indicates the quadratic extension of either the $t$-field (in Tables \mbox{C, D})
or of the $j$-field (in Tables F, G);
\item $m$-$\sqrt{\;}\,$: the quadratic extension for the moduli field (only in table A);
\item $r$-$\sqrt{\;}\,$: the quadratic extension for the $r$-field (only in table A).
\end{itemize}
The tables are supplemented by pictures of respective minus-4-hyperbolic dessins d'enfant.
Some thinly drawn\footnote{Our policy 
of drawing dessins is the following. White vertices of order 2 
are not shown, but the edges going through them are drawn thick. 
Other white vertices are shown,  but the incident edges are drawn thin. 
A black vertex of degree $\ge 2$ is not drawn
(as it is a clear branching point), unless it is incident to a thin edge. 
The dessins were drawn from the combinatorial representations $(g_0,g_1,g_\infty)$ first by hand,
then by using a developed script language that was translated to {\sf LaTeX} using {\sf Maple}.}
dessins represent also the composition $4\varphi(1-\varphi)$ giving a clean dessin.
The composition label is then marked \mbox{by $^\bullet$.}
The J-pictures marked by the symbol \# represent 
4 dessins\footnote{Apart from the \#-labeling and Figure \ref{fg:deg60s},
all other pictures represent either one dessin (if there is a reflection symmetry)
or two dessins related by a complex comnjugation (otherwise). In the cases like B13, F12,
a reflection symmetry should be imagined on the Riemann sphere, along a ``circular" equator.}
each, obtainable by reflecting  (with respect to a horizontal axis) their left and right parts 
independently. The dessins of B12, C6, C30, D45, F1, F4, F6, F7, F11, H1, H10--H14, H46 
are displayed in Figure \ref{fg:deg60s}, \ref{fg:mdessins}, \ref{fg:vdessins} earlier.
%

\begin{table}
\begin{minipage}{336pt}
\begin{tabular}{@{}llcclrr@{}} 
\hline Id & Branching frac. & $d$ & $\![k\ell m]\!$ & Monodromy/comp. & \!$m$-$\sqrt{\;}\!$ 
& $\!r$-$\sqrt{\;}\!$ \\ 
\hline
A1 & $\led{\fr12,\fr12,\fr15,\fr15}$ & 6 & [255] & $G = A_5$ & --- & $-1$ \\ 
A2 & 
 & 12 & [245]  & $2\CPR{P_1}{445}\circ6$, $\AJ{A1}\circ2$ & --- & $-1$\\ 
A3 & $\led{\fr12,\fr12,\fr27,\fr27}$ & 18 & [237] & $2\CPR{P_1}{277}\circ9$ & --- & $-7$ \\ 
A4 & $\led{\fr12,\fr14,\fr14,\fr25}$ & 12 & [245] & $2\CPR{P_1}{445}\circ6$ & $-1$ & $-1$ \\ 
A5 & $\led{\fr12,\fr14,\fr18,\fr18}$ & 24 & [238] &
 $2\CPR{P_1}{488}\circ2 \CPR{P_3}{288}\circ 2 \CPR{P_{13}}{248}\circ3\!\!\!$ & --- & --- \\ 
A6 & $\led{\fr13,\fr13,\fr13,\fr19}$ & 16 & [239] & $4\CPR{P_5}{339}\circ 4$ & --- & --- \\ 
A7 & $\led{\fr13,\fr13,\fr23,\fr27}$ & 16 & [237] & $2\CPR{P_1}{337}\circ8$ & $-3$ & $-3$ \\
\hline
\end{tabular}
\end{minipage} \hfil
\begin{minipage}{80pt}
\input{Adessins.tex}

\begin{table}
\begin{minipage}{356pt}
\begin{tabular}{@{}lllccl@{}} 
\hline Id & $j$-invariant & Branching frac. & $d$ & $\![k\ell m]\!$ & Monodromy/comp. \\ 
\hline
B1& $2^{4}13^{3}/3^{2}$ & $\led{\fr12,\fr14,\fr18,\fr38}$ & 18 & [238]
 & $3\CPR{P_{15}}{288}\circ2 \CPR{P_{16}}{248}\circ 3\!$ \\ 
B2 & & $\led{\fr12,\fr17,\fr27,\fr37}$ & 27 & [237] & $3\CPR{P_{15}}{277}\circ9$ \\ 
B3 & & $\led{\fr14,\fr14,\fr14,\fr14}$ & 6 & [344] & $2_H\circ3$ \\ 
B4 & & 
 & 12 & [246] & \small $4_H\!\{\AJ{D25},\AJ{P22}\{\AJ{B3}\},\AJ{P39}\}\!\circ\!3\!$ \\ 
B5 & $2^{2}73^{3}/3^{4}$ & $\led{\fr13,\fr13,\fr16,\fr16}$ & 12 & [246]
 & $6\,\{\AJ{P16},[366]\}\!\CPR{P_{22}}{266}\circ2$ \\
B6 & & $\led{\fr13,\fr23,\fr17,\fr27}$ & 24 & [237] & $3\CPR{P_{19}}{337}\circ8$ \\ 
B7 & & $\led{\fr13,\fr23,\fr19,\fr29}$ & 12 & [239] 
 & $3\CPR{P_{19}}{339}\circ4$ \\
\hline
\end{tabular} 
\end{minipage} \hfil
\begin{minipage}{80pt}
\input{Bdessins.tex}

\begin{table}
\begin{minipage}{366pt}
\begin{tabular}{@{}lcllccl@{}} 
\hline Id & $\sqrt{\;}$ & $j$-invariant & Branching frac. & $d$ & $\![k\ell m]\!$ & Monodromy/comp. \\ 
\hline
C1 & 2 & $2^{3}3^{3}11^{3}$ & $\led{\fr12,\fr12,\fr14,\fr14}$
 & 12 & [238] & $2_H\circ2\CPR{P_{37}}{248}\circ3$ \\ 
C2 & & $5^{3}11^{3}31^{3}/2^{3}7^{6}$ & $\led{\frc1{14},\frc1{14},\frc1{14},\frc1{14}}\!$
 & 18 & [14] & $2_H\circ9$ \\ 
C3 & & $73^{3}601^{3}/2\cdt 3^{4}7^{8}$ & $\led{\fr12,\fr12,\fr16,\fr16}$
 & 8 & [246] & $G=$PSL$(3,2)\!:\!C_2$ \\ 
C4 & & & $\led{\fr17,\fr17,\fr37,\fr67}$ & 18 & [237] & $G = S_{18}$ \\ 
C5 & 3 & $2^{2}193^{3}/3$ & $\led{\fr14,\fr14,\fr38,\fr38}$
 & 18 & [238] &  $2_H\circ3\CPR{P_{39}}{248}\circ3$ \\ 
C6 & & $2^{7}53^{3}/3^{3}$ & $\led{\fr13,\fr13,\fr27,\fr27}$
 & 32 & [237] & $4\CPR{P_{40}}{337}\circ8$ \\ 
C7 & & & $\led{\fr13,\fr13,\fr29,\fr29}$ & 16 & [239] & $4\CPR{P_{40}}{339}\circ4$  \\
\hline
\end{tabular}  

\end{minipage} \hfil
\begin{minipage}{80pt}
\begin{picture}(80,254)(-24,-144) \thicklines
\put(0,236){C25} \qbezier(30,230)(43,242)(30,242) \qbezier(30,230)(17,242)(30,242) \put(30,230){\line(0,-1){9}} \put(30,221){\line(-1,-1){14}} \put(30,221){\line(1,-1){14}} \put(44,207){\line(-1,0){28}} \put(44,207){\line(-1,-1){14}} \put(16,207){\line(1,-1){14}} \put(30,193){\line(0,-1){9}} \qbezier(30,184)(43,172)(30,172) \qbezier(30,184)(17,172)(30,172) 
\put(0,161){C26} \put(4,149){\circle*3} \qbezier(4,149)(12,164)(20,149) \qbezier(4,149)(12,134)(20,149) \put(20,149){\line(1,0){12}} \put(32,149){\line(1,-1){14}} \put(32,149){\line(1,1){14}} \put(46,163){\line(0,-1){28}} \qbezier(46,135)(64,149)(46,163) 
\put(0,122){C27} \qbezier(30,77)(18,65)(30,65) \qbezier(30,77)(42,65)(30,65) \put(30,77){\line(0,1){20}} \qbezier(30,97)(4,111)(30,127) \qbezier(30,97)(56,111)(30,127) \put(30,127){\line(0,-1){9}} \qbezier(30,118)(20,107)(30,107) \qbezier(30,118)(40,107)(30,107) \put(30,87){\circle*3} 
\put(0,50){C28} \put(3,42){\circle*3} \put(3,42){\line(1,0){12}} \put(15,42){\line(1,-1){14}} \put(15,42){\line(1,1){14}} \put(29,56){\line(0,-1){28}} \put(29,56){\line(1,-1){14}} \put(29,28){\line(1,1){14}} \put(43,42){\line(1,0){12}} \put(55,42){\circle*3} 
\put(0,14){C29} \qbezier(41,8)(54,22)(54,8) \qbezier(41,8)(54,-6)(54,8) \put(41,8){\line(-1,0){36}} \put(5,8){\circle*3} \put(29,8){\circle*3} \put(17,8){\circle*3} 
\put(0,-12){C31} \qbezier(18,-30)(-3,-27)(5,-19) \qbezier(18,-30)(13,-11)(5,-19) \qbezier(18,-30)(-3,-33)(5,-41) \qbezier(18,-30)(13,-49)(5,-41) \put(18,-30){\line(1,0){16}} \qbezier(34,-30)(55,-27)(47,-19) \qbezier(34,-30)(39,-11)(47,-19) \qbezier(34,-30)(55,-33)(47,-41) \qbezier(34,-30)(39,-49)(47,-41) 
\put(0,-58){C32} \qbezier(13,-70)(0,-56)(0,-70) \qbezier(13,-70)(0,-84)(0,-70) \qbezier(43,-70)(56,-56)(56,-70) \qbezier(43,-70)(56,-84)(56,-70) \put(13,-70){\line(1,0){30}} \put(23,-70){\circle*3} \put(33,-70){\circle*3} 
\put(1,-95){C33} \put(55,-103){\circle*3} \put(55,-103){\line(-1,0){12}} \put(43,-103){\line(-1,-1){14}} \put(43,-103){\line(-1,1){14}} \put(29,-89){\line(0,-1){28}} \put(29,-89){\line(-1,-1){14}} \put(29,-117){\line(-1,1){14}} {\thinlines\put(15,-103){\circle*3} \put(15,-103){\line(-1,0){11}} \put(2,-103){\circle3} }
\put(0,-131){C34} \qbezier(15,-143)(3,-130)(3,-143) \qbezier(15,-143)(3,-156)(3,-143) \put(15,-143){\line(1,0){10}} \put(25,-143){\line(1,-1){14}} \put(25,-143){\line(1,1){14}} \put(39,-129){\line(0,-1){28}} \put(39,-129){\line(1,-1){14}} \put(39,-157){\line(1,1){14}} \put(53,-143){\circle*3} 
\put(0,-170){C35} \put(2,-185){\circle*3} \qbezier(2,-185)(13,-163)(24,-185) \qbezier(2,-185)(13,-207)(24,-185) \put(24,-185){\line(1,0){12}} \put(36,-185){\line(1,-1){12}} \put(36,-185){\line(1,1){12}} \put(48,-173){\line(0,-1){24}} \qbezier(48,-197)(64,-185)(48,-173) {\thinlines\put(2,-185){\line(1,0){11}} \put(14,-185){\circle3} }
\put(0,-214){C36} \put(12,-268){\line(1,0){36}} \qbezier(48,-268)(30,-284)(12,-268) \qbezier(12,-268)(16,-244)(30,-239) \qbezier(48,-268)(44,-244)(30,-239) \put(30,-239){\line(0,1){10}} \qbezier(30,-229)(8,-218)(30,-207) \qbezier(30,-229)(52,-218)(30,-207) \put(30,-207){\circle*3} {\thinlines\put(30,-207){\line(0,-1){11}} \put(30,-218){\circle3} }\put(30,-268){\line(0,1){8}} \qbezier(30,-260)(41,-249)(30,-249) \qbezier(30,-260)(19,-249)(30,-249) 
\end{picture}\end{minipage}\end{table}
\begin{picture}(340,180)(12,2) \thicklines{\thinlines
\put(0,176){C37/C38$^\bullet$} \put(2,161){\circle3} \put(3,161){\line(1,0){12}} \put(15,161){\circle*3} \qbezier(16,162)(24,177)(32,162) \qbezier(16,160)(24,145)(32,160) \put(33,161){\circle3} \put(34,161){\line(1,0){12}} \put(46,161){\circle*3} \qbezier(47,162)(55,177)(63,162) \qbezier(47,160)(55,145)(63,160) \put(64,161){\circle3} \put(65,161){\line(1,0){12}} \put(77,161){\circle*3} }
\put(89,176){C39} \qbezier(143,161)(149,173)(155,161) \qbezier(143,161)(149,149)(155,161) \put(155,161){\line(1,0){9}} \qbezier(164,161)(148,193)(124,161) \qbezier(164,161)(148,129)(124,161) \put(124,161){\line(-1,0){12}} \put(112,161){\line(-1,-1){12}} \put(112,161){\line(-1,1){12}} \put(100,173){\line(0,-1){24}} \qbezier(100,149)(84,161)(100,173) {\thinlines\put(143,161){\circle*3} \put(143,161){\line(-1,0){8}} \put(134,161){\circle3} }
\put(174,175){C40} \qbezier(216,161)(229,175)(229,161) \qbezier(216,161)(229,147)(229,161) \put(216,161){\line(-1,0){22}} \qbezier(194,161)(186,176)(178,161) \qbezier(194,161)(186,146)(178,161) \put(178,161){\circle*3} \put(205,161){\circle*3} 
\put(239,175){C41} \put(312,161){\circle*3} \put(312,161){\line(-1,0){12}} \qbezier(300,161)(292,176)(284,161) \qbezier(300,161)(292,146)(284,161) \put(284,161){\line(-1,0){12}} \qbezier(272,161)(264,177)(256,161) \qbezier(272,161)(264,145)(256,161) {\thinlines\put(256,161){\circle*3} \put(256,161){\line(-1,0){11}} \put(244,161){\circle3} }
\put(320,175){C42} \put(381,161){\circle*3} \qbezier(381,161)(373,176)(365,161) \qbezier(381,161)(373,146)(365,161) \put(365,161){\line(-1,0){12}} \qbezier(353,161)(345,176)(337,161) \qbezier(353,161)(345,146)(337,161) {\thinlines\put(337,161){\circle*3} \put(337,161){\line(-1,0){11}} \put(325,161){\circle3} }
\put(394,175){D1} \qbezier(408,161)(396,174)(396,161) \qbezier(408,161)(396,148)(396,161) \put(408,161){\line(1,0){10}} \put(418,161){\line(4,-5){12}} \put(418,161){\line(4,5){12}} \qbezier(430,176)(436,164)(448,170) \qbezier(430,176)(442,182)(448,170) \qbezier(430,146)(436,158)(448,152) \qbezier(430,146)(442,140)(448,152) \put(448,170){\line(0,-1){18}} 
\put(0,126){D2} \put(6,112){\circle*3} \put(6,112){\line(1,-1){16}} \put(6,112){\line(1,1){16}} \put(22,128){\line(0,-1){32}} \put(22,128){\line(1,0){12}} \put(22,96){\line(1,0){12}} \qbezier(34,96)(46,108)(46,96) \qbezier(34,96)(46,84)(46,96) \qbezier(34,128)(46,116)(46,128) \qbezier(34,128)(46,140)(46,128) 
\put(59,126){D3} \qbezier(63,112)(79,125)(79,112) \qbezier(63,112)(79,99)(79,112) \qbezier(63,112)(71,126)(91,128) \qbezier(63,112)(71,98)(91,96) \put(91,128){\line(0,-1){32}} \qbezier(91,96)(105,108)(105,96) \qbezier(91,96)(105,84)(105,96) \qbezier(91,128)(105,116)(105,128) \qbezier(91,128)(105,140)(105,128) 
\put(116,126){D4} \qbezier(172,128)(184,140)(184,128) \qbezier(172,128)(184,116)(184,128) \qbezier(172,96)(184,84)(184,96) \qbezier(172,96)(184,108)(184,96) \put(172,96){\line(-1,0){12}} \put(172,128){\line(-1,0){12}} \put(160,128){\line(0,-1){32}} \put(160,96){\line(-1,0){24}} \put(160,128){\line(-1,0){24}} \put(136,128){\line(0,-1){32}} \qbezier(136,96)(117,112)(136,128) 
\put(198,126){D5} \qbezier(211,105)(223,112)(211,119) \qbezier(211,105)(199,112)(211,119) \put(211,105){\line(1,-2){7}} \put(211,119){\line(1,2){7}} \put(218,91){\line(1,3){7}} \put(218,133){\line(1,-3){7}} \put(225,112){\line(1,0){32}} \put(257,112){\line(-1,-3){7}} \put(257,112){\line(-1,3){7}} \put(218,91){\line(1,0){32}} \put(218,133){\line(1,0){32}} \put(250,91){\line(-5,3){10}} \put(250,133){\line(-5,-3){10}} \qbezier(240,127)(225,128)(232,120) \qbezier(240,127)(239,112)(232,120) \qbezier(240,97)(225,96)(232,104) \qbezier(240,97)(239,112)(232,104) 
\put(269,126){D6} \qbezier(283,112)(271,125)(271,112) \qbezier(283,112)(271,99)(271,112) \put(283,112){\line(1,0){10}} \put(293,112){\line(6,-5){24}} \put(293,112){\line(6,5){24}} \put(317,132){\line(2,-5){4}} \qbezier(321,122)(335,116)(325,112) \qbezier(321,122)(315,108)(325,112) \put(317,92){\line(6,5){24}} \put(317,132){\line(6,-5){24}} \put(341,112){\line(1,0){10}} \qbezier(351,112)(364,125)(364,112) \qbezier(351,112)(364,99)(364,112) \put(317,94){\line(-2,5){4}} \qbezier(313,104)(299,110)(309,114) \qbezier(313,104)(319,118)(309,114) {\thinlines
\put(370,126){D7/D8$^\bullet$} \put(377,112){\circle*3} \put(377,112){\line(1,0){12}} \put(390,112){\circle3} \put(391,111){\line(5,-4){20}} \put(391,113){\line(5,4){20}} \put(412,130){\circle*3} \put(412,94){\circle*3} \put(412,94){\line(0,1){11}} \put(412,130){\line(0,-1){11}} \put(412,118){\circle3} \put(412,106){\circle3} \put(413,95){\line(5,4){20}} \put(413,129){\line(5,-4){20}} \put(434,112){\circle3} \put(435,112){\line(1,0){12}} \put(447,112){\circle*3} }
\end{picture}

\begin{table}
\begin{minipage}{366pt}
\begin{tabular}{@{}lcllccl@{}} 
\hline Id & $\sqrt{\;}$ & $j$-invariant & Branching frac. & $d$ & $\![k\ell m]\!$ & Monodromy/comp. \\ 
\hline
D1 & $-1$ & $2^{2}3^{3}13^{3}/5^{4}$ & $\led{\fr14,\fr14,\fr18,\fr58}$ & 18 & [238] & 
$6 \CPR{P_{44}}{248}\circ 3$ \\ 
D2 & & $\!-2^{4}109^{3}/5^{6}$ & $\led{\fr13,\fr23,\fr19,\fr19}$ & 14 & [239] & $G = S_{14}$ \\ 
D3 & & & $\led{\fr15,\fr15,\fr15,\fr45}$ & 12 & [245] & $6 \CPR{P_{46}}{255}\circ 2$ \\ 
D4 & & & $\led{\fr15,\fr25,\frc1{10},\frc1{10}}$ & 18 & [10] & $6 \CPR{P_{46}}{\stackrel{2\,5\,10}{}}\circ 3$ \\ 
D5 & & & $\led{\fr17,\fr17,\fr27,\fr57}$ & 30 & [237] & $G = S_{30}$ \\ 
D6 & & & $\led{\frc1{10},\frc1{10},\frc1{10},\frc1{10}}$ & 24 & [10] & 
$4_H\{\AJ{G5}\!\stackrel{\!\times}{},\AJ{P43}\}\circ 6$  \\ 
D7 & $-2$ & $\!-2^{5}19^{3}/3^{6}$ & $\led{\fr13,\fr13,\fr13,\fr13}$ & 8 & [334] & $2_H\circ4$ \\
\hline
\end{tabular}  

\end{minipage} \hfil
\begin{minipage}{80pt}
\input{Ddessins.tex}

\begin{table}[b]
\begin{minipage}{366pt}
\begin{tabular}{@{}llllccl@{}} 
\hline Id & disc\,$\QQ(t)$ & $j$-invariant & Branching frac. & $d$ & $\![k\ell m]\!$ & Monodromy/comp. \\ 
\hline
E1 & $-2^{9}3^{7}$ & $-2^6 3^3 23^3$ & $\led{\fr12,\fr13,\fr13,\fr13}$ & 6 & [334] & $G = A_6$ \\ 
E2 & & & 
 & 12 & [238] & $\AJ{E1}\circ2$ \\ 
E3 & & $-3^3 17^3/2$ & $\led{\fr23,\fr19,\fr19,\fr19}$ & 18 & [239] & $G = S_{18}$ \\ 
E4 & $-2^{9}3^{8}$ & $-3^3 5^3 383^3/2^{7}$ & $\led{\fr13,\fr13,\fr13,\fr14}$ 
 & 9 & [334] & \scriptsize $G=\!((C_3\!\times\!C_3)\!:\!Q_8)\!:\!C_3\!$ \\ 
E5 & & & 
 & 18 & [238] & $\AJ{E4}\circ2$ \\ 
E6 & & & $\led{\fr13,\fr13,\fr13,\fr47}$ & 18 & [237] & $G = S_{18}$ \\ 
E7 & $-2^{9}5^{4}$ & $-5^2 241^3/2^{3}$ & $\led{\fr12,\frc1{10},\frc1{10},\frc1{10}}$ 
 & 18 & [10] & $G = S_{18}$ \\
\hline
\end{tabular}  

\end{minipage} \hfil
\begin{minipage}{60pt}
\input{Edessins.tex}

\begin{table}[b]
\begin{minipage}{326pt}
\begin{tabular}{lcllccl} 
\hline Id & \mbox{}\hspace{-2pt}$\sqrt{\ }$ & disc\,$\QQ(t)$ & Branching frac. & $d$ & $\![k\ell m]\!$ 
& Monodromy/comp. \\  
\hline
F1 & 2 & as $j$ & $\led{\fr13,\fr13,\fr13,\fr13}$ & 8 & [334] & $G=$PSL(3,2) \\ 
F2 & & & 
 & 16 & [238] & $2_H\circ \AJ{F5}$, $\AJ{F1}\circ 2$  \\ 
F3 & & & $\led{\fr13,\fr27,\fr37,\fr47}$ & 16 & [237] & $G = A_{16}$ \\ 
F4 & & $2^{11}$ & $\led{\fr12,\fr12,\fr18,\fr18}$ & 18 & [238] & $G=$PSL(2,17) \\ 
F5 & & $2^{11}3^{2}$ & $\led{\fr12,\fr12,\fr13,\fr13}$ & 8 & [238] & $G=$PSL(3,2)$:\!C_2$ \\ 
F6 & & $2^{6}17$ & $\led{\fr18,\fr18,\fr18,\fr18}$ & 36 & [238] & $2_H\circ \AJ{F4}$ \\ 
F7 & & $-2^{6}3^{2}7$ & $\led{\fr12,\fr13,\fr13,\fr14}$ & 7 & [334] & $G=$PSL(3,2) \\
\hline
\end{tabular} 
\end{minipage} \hfil
\begin{minipage}{100pt}
\begin{picture}(110,154)(-11,0) \thicklines
\put(0,191){F22} \put(36,202){\line(0,-1){30}} \qbezier(36,172)(20,187)(36,202) \put(36,172){\line(1,1){15}} \put(36,202){\line(1,-1){15}} \put(51,187){\line(1,0){11}} \put(62,187){\circle*3} {\thinlines\put(62,187){\line(1,-1){10}} \put(62,187){\line(1,1){10}} \put(73,198){\circle3} \put(73,176){\circle3} \put(9,160){\circle*3} \put(9,160){\line(1,0){10}} \put(20,160){\circle3} \put(62,160){\circle*3} \put(62,160){\line(1,0){10}} \put(73,160){\circle3} }\qbezier(9,160)(21,179)(33,160) \qbezier(9,160)(21,141)(33,160) \put(33,160){\line(1,0){11}} \qbezier(44,160)(53,173)(62,160) \qbezier(44,160)(53,147)(62,160) 
\put(0,120){F23} \put(30,128){\line(0,-1){34}} \qbezier(30,94)(12,111)(30,128) \qbezier(30,128)(53,125)(57,111) \qbezier(30,94)(53,97)(57,111) \put(57,111){\line(1,0){11}} \qbezier(68,111)(76,126)(84,111) \qbezier(68,111)(76,96)(84,111) \put(84,111){\circle*3} \put(30,111){\line(1,0){9}} \qbezier(39,111)(49,100)(49,111) \qbezier(39,111)(49,122)(49,111) \put(71,82){\circle*3} \qbezier(63,82)(53,72)(53,82) \qbezier(63,82)(53,92)(53,82) \put(63,82){\line(1,0){20}} \qbezier(83,82)(60,112)(43,82) \qbezier(83,82)(60,52)(43,82) \put(43,82){\line(-1,0){15}} \put(28,82){\line(-1,-1){12}} \put(28,82){\line(-1,1){12}} \put(16,94){\line(0,-1){24}} \qbezier(16,70)(0,82)(16,94) 
\put(0,30){F24} \put(49,3){\line(0,1){31}} \put(49,3){\line(2,1){20}} \put(49,3){\line(-2,1){20}} \put(69,13){\line(0,1){31}} \put(29,13){\line(0,1){31}} \put(29,44){\line(1,0){40}} \put(69,44){\line(-2,-1){20}} \put(29,44){\line(2,-1){20}} \put(69,13){\line(-3,5){6}} \put(29,13){\line(3,5){6}} \qbezier(35,23)(49,23)(42,30) \qbezier(35,23)(35,37)(42,30) \qbezier(63,23)(49,23)(56,30) \qbezier(63,23)(63,37)(56,30) \put(83,10){\line(0,-1){34}} \qbezier(83,-24)(101,-7)(83,10) \qbezier(83,10)(61,7)(57,-7) \qbezier(83,-24)(61,-21)(57,-7) \put(57,-7){\line(-1,0){16}} \qbezier(41,-7)(28,19)(11,-7) \qbezier(41,-7)(28,-33)(11,-7) \put(11,-7){\line(1,0){9}} \qbezier(20,-7)(30,-18)(30,-7) \qbezier(20,-7)(30,4)(30,-7) \put(83,-7){\line(-1,0){8}} \qbezier(75,-7)(65,-18)(65,-7) \qbezier(75,-7)(65,4)(65,-7) 
\put(0,-50){F25} \put(15,-60){\circle*3} \put(15,-60){\line(1,-1){14}} \put(15,-60){\line(1,1){14}} \put(29,-46){\line(0,-1){28}} \put(29,-46){\line(1,0){12}} \put(29,-74){\line(1,0){12}} \put(41,-74){\circle*3} \put(41,-46){\circle*3} {\thinlines\put(3,-60){\circle3} \put(4,-60){\line(1,0){11}} \put(48,-60){\circle*3} \put(48,-60){\line(1,0){10}} \put(59,-60){\circle3} }\qbezier(48,-60)(60,-41)(72,-60) \qbezier(48,-60)(60,-79)(72,-60) \put(72,-60){\line(1,0){11}} \put(83,-60){\line(1,-1){10}} \put(83,-60){\line(1,1){10}} \put(93,-50){\circle*3} \put(93,-70){\circle*3} 
\end{picture}\end{minipage}\end{table}

\begin{table}
\begin{minipage}{300pt}
\renewcommand{\arraystretch}{0.99} 
\begin{tabular}{@{}lcllccl@{}} 
\hline Id & $\sqrt{\ }$ & disc\,$\QQ(t)$  & Branching frac. & $d$ & $\![k\ell m]\!$ & Monodromy/comp. \\ 
\hline
G1 & $-1$ & as $j$ & $\led{\fr12,\fr12,\fr14,\fr18}$ & 15 & [238] & $5\CPR{P_{60}}{248}\circ 3$ \\ 
G2 & & & $\led{\fr13,\fr15,\frc1{10},\frc3{10}}$ & 16 & [10] & $G = A_{16}$ \\ 
G3 & & $2^{8}$ & $\led{\fr23,\fr14,\fr18,\fr18}$ & 20 & [238] & $G = S_{20}$ \\ 
G4 & & $2^{6}5$ & $\led{\fr14,\fr14,\fr18,\fr18}$ & 10 & [248] & $2_H\circ 5$ \\ 
G5 & & & 
 & 30 & [238] & $2_H\circ \AJ{G1:\,5\,:G4}\circ 3$  \\ 
G6 & & & $\led{\fr12,\fr18,\fr18,\fr18}$ & 27 & [238] & $G = S_{27}$ \\ 
G7 & & $2^{8}5$ & $\led{\fr12,\fr12,\frc1{10},\frc1{10}}$ & 12 & [10] & $2_H\circ 6$ \\
\hline
\end{tabular} 

\end{minipage} \hfil
\begin{minipage}{80pt}
\input{Gdessins.tex}

\begin{table}
\begin{minipage}{300pt}
\renewcommand{\arraystretch}{0.98}  
\begin{tabular}{@{}l@{\hspace{10pt}}lll@{\hspace{11pt}}c@{\hspace{11pt}}c@{\hspace{11pt}}l@{}} 
\hline Id &   disc\,$\QQ(j)$  & \mbox{}\hspace{-3pt}disc\,$\QQ(t)$  & Branching frac. & $d$ & $\![k\ell m]\!$ & Monodromy/comp. \\ 
\hline
H1 & $\;3^4$ & as $j$ & $\led{\fr13,\fr13,\fr19,\fr19}$ & 20 & [239] & $G=$PSL(2,19) \\ 
H2 & $-3^5$ & as $j$ & $\led{\fr13,\fr23,\fr19,\fr29}$ & 12 & [239] & $G = S_{12}$ \\ 
H3 & & $-2^{2}3^{10}5$ & $\led{\fr13,\fr13,\fr23,\fr19}$ & 10 & [239] & $G = S_{10}$ \\ 
H4 & $-2^23^3$ & as $j$ & $\led{\fr12,\fr13,\fr19,\fr29}$ & 15 & [239] & $G = S_{15}$ \\ 
H5 & & $-2^{8}3^{6}5$ & $\led{\fr12,\fr14,\fr14,\fr16}$ & 10 & [246] & $G =  (A_6\!:\!C_2)\!:\!C_2$ \\ 
H6 & $-2^33^3$ & as $j$ & $\led{\fr13,\fr19,\fr29,\fr49}$ & 16 & [239] & $G = A_{16}$ \\ 
H7 & $-2^23^4$ & $-2^{8}3^{9}$ & $\led{\fr12,\fr12,\fr19,\fr29}$ & 12 & [239] & $G = S_{12}$ \\
\hline
\end{tabular} 

\end{minipage} \hfil
\begin{minipage}{80pt}
\begin{picture}(310,520)(-26,20) \thicklines
\put(1,595){H43} \put(12,542){\line(1,0){24}} \qbezier(36,542)(30,554)(42,560) \qbezier(36,542)(48,548)(42,560) \qbezier(12,542)(18,554)(6,560) \qbezier(12,542)(0,548)(6,560) \put(6,560){\line(0,1){18}} \put(42,560){\line(0,1){18}} \put(6,578){\line(1,-1){18}} \put(42,578){\line(-1,-1){18}} \put(6,578){\line(1,1){18}} \put(42,578){\line(-1,1){18}} \put(24,596){\line(0,-1){12}} \put(24,584){\circle*3} {\thinlines\put(24,560){\circle*3} \put(24,560){\line(0,1){11}} \put(24,572){\circle3} \put(71,559){\circle3} \put(72,559){\line(1,0){11}} \put(83,559){\circle*3} }\put(48,569){\line(1,0){40}} \put(48,569){\line(1,-2){11}} \put(48,569){\line(1,2){11}} \qbezier(59,591)(67,601)(77,591) \qbezier(59,591)(67,581)(77,591) \qbezier(59,547)(67,537)(77,547) \qbezier(59,547)(67,557)(77,547) \put(77,547){\line(1,2){11}} \put(77,591){\line(1,-2){11}} \put(68,569){\line(0,1){10}} \put(68,579){\circle*3} 
\put(2,523){H44} \qbezier(39,530)(61,528)(65,514) \qbezier(39,498)(61,500)(65,514) \put(65,514){\line(1,0){10}} \qbezier(75,514)(87,526)(87,514) \qbezier(75,514)(87,502)(87,514) \put(39,530){\line(0,-1){32}} \put(39,514){\line(1,0){8}} \qbezier(47,514)(58,525)(58,514) \qbezier(47,514)(58,503)(58,514) \put(39,498){\line(-5,4){20}} \put(39,530){\line(-5,-4){20}} \put(19,514){\circle*3} {\thinlines\put(19,514){\line(1,0){11}} \put(31,514){\circle3} }\qbezier(75,491)(64,479)(64,491) \qbezier(75,491)(64,503)(64,491) \put(75,491){\line(1,0){9}} \qbezier(84,491)(68,465)(54,491) \qbezier(84,491)(68,517)(54,491) \put(54,491){\line(-1,0){18}} \put(36,491){\line(-1,1){18}} \put(36,491){\line(-1,-1){18}} \put(18,473){\circle*3} {\thinlines\put(18,473){\line(0,1){11}} \put(18,485){\circle3} }\qbezier(18,473)(-5,491)(18,509) \put(18,509){\line(0,-1){8}} \qbezier(18,501)(8,491)(18,491) \qbezier(18,501)(28,491)(18,491) 
\put(2,456){H45} \put(28,450){\circle*3} \put(28,450){\line(1,-1){18}} \put(28,450){\line(1,1){18}} \put(46,468){\circle*3} \put(46,432){\circle*3} \put(46,432){\line(1,1){18}} \put(46,468){\line(1,-1){18}} \put(64,450){\line(1,0){11}} \qbezier(75,450)(88,463)(88,450) \qbezier(75,450)(88,437)(88,450) {\thinlines\put(28,450){\line(1,0){10}} \put(39,450){\circle3} \put(46,432){\line(0,1){10}} \put(46,468){\line(0,-1){10}} \put(46,457){\circle3} \put(46,443){\circle3} }\put(55,425){\circle*3} \put(55,425){\line(-1,0){12}} \put(43,425){\line(-1,1){18}} \put(43,425){\line(-1,-1){18}} \put(25,407){\circle*3} \qbezier(25,407)(2,425)(25,443) \put(25,443){\line(0,-1){8}} \qbezier(25,435)(15,425)(25,425) \qbezier(25,435)(35,425)(25,425) {\thinlines\put(25,407){\line(0,1){11}} \put(25,419){\circle3} \put(55,425){\line(5,4){10}} \put(55,425){\line(5,-4){10}} \put(66,416){\circle3} \put(66,434){\circle3} }
\put(2,391){H47} \qbezier(33,375)(20,388)(20,375) \qbezier(33,375)(20,362)(20,375) \put(33,375){\line(1,0){22}} \put(55,375){\circle*3} {\thinlines\put(55,375){\line(5,-3){10}} \put(55,375){\line(5,3){10}} \put(66,382){\circle3} \put(66,368){\circle3} \put(44,375){\circle*3} }\qbezier(45,393)(32,380)(32,393) \qbezier(45,393)(32,406)(32,393) \put(45,393){\line(1,0){22}} \put(67,393){\circle*3} {\thinlines\put(67,393){\line(1,0){11}} \put(79,393){\circle3} \put(54,393){\circle*3} \put(54,393){\line(0,1){10}} \put(54,404){\circle3} }
\put(2,351){H48} \put(24,337){\circle*3} \put(64,337){\line(-5,-2){20}} \put(24,337){\line(5,-2){20}} \put(44,329){\line(0,1){8}} \qbezier(44,337)(55,348)(44,348) \qbezier(44,337)(33,348)(44,348) \put(64,337){\line(-4,5){20}} \put(24,337){\line(4,5){20}} \put(44,362){\circle*3} \put(62,337){\line(1,0){12}} \qbezier(74,337)(87,350)(87,337) \qbezier(74,337)(87,324)(87,337) \qbezier(29,317)(16,330)(16,317) \qbezier(29,317)(16,304)(16,317) \put(29,317){\line(1,0){20}} \qbezier(49,317)(65,346)(87,317) \qbezier(49,317)(65,288)(87,317) \put(87,317){\line(-1,0){18}} \qbezier(69,317)(58,328)(58,317) \qbezier(69,317)(58,306)(58,317) \put(39,317){\circle*3} \put(76,317){\circle*3} 
\put(2,285){H49} \put(82,279){\circle*3} \put(82,279){\line(-1,0){10}} \put(72,279){\line(-1,1){13}} \put(72,279){\line(-1,-1){13}} \put(59,266){\line(0,1){26}} \put(59,292){\line(-1,0){20}} \put(59,266){\line(-1,0){20}} \put(39,266){\line(5,4){10}} \put(49,274){\circle*3} \put(39,266){\line(0,1){26}} \put(39,292){\circle*3} {\thinlines\put(39,292){\line(5,-4){10}} \put(50,283){\circle3} }\put(17,258){\circle*3} \qbezier(17,258)(28,278)(41,258) \qbezier(17,258)(28,238)(41,258) \put(41,258){\line(1,0){18}} \put(59,258){\line(1,-1){16}} \put(59,258){\line(1,1){16}} \put(75,242){\line(0,1){11}} \put(75,274){\line(0,-1){11}} \put(75,263){\circle*3} \put(75,253){\circle*3} \qbezier(75,274)(94,258)(75,242) {\thinlines\put(17,258){\line(1,0){11}} \put(29,258){\circle3} }
\put(2,229){H50} \put(52,234){\circle*3} \put(52,234){\line(0,-1){22}} \put(52,212){\line(1,0){18}} \put(52,234){\line(1,0){18}} \put(70,234){\line(0,-1){22}} \qbezier(70,212)(85,223)(70,234) \put(38,212){\circle*3} \put(38,212){\line(1,0){14}} {\thinlines\put(38,234){\circle3} \put(39,234){\line(1,0){13}} }\put(27,200){\circle*3} \put(27,200){\line(-1,0){11}} \qbezier(16,200)(27,218)(38,200) \qbezier(16,200)(27,182)(38,200) \put(38,200){\line(1,0){11}} \qbezier(49,200)(58,214)(67,200) \qbezier(49,200)(58,186)(67,200) \put(67,200){\circle*3} {\thinlines\put(67,200){\line(1,0){10}} \put(78,200){\circle3} }
\put(2,163){H51} \put(32,145){\line(1,0){40}} \put(72,145){\line(-1,-2){11}} \put(32,145){\line(1,-2){11}} \qbezier(43,123)(51,113)(61,123) \qbezier(43,123)(51,133)(61,123) \qbezier(61,123)(53,113)(43,123) \qbezier(61,123)(53,133)(43,123) \put(72,145){\line(0,1){35}} \put(32,145){\line(0,1){35}} \put(32,180){\line(1,0){40}} \put(72,180){\line(-5,-3){20}} \put(32,180){\line(5,-3){20}} \put(52,168){\line(0,-1){12}} \put(52,156){\circle*3} \put(55,135){\circle*3} \put(55,135){\line(1,0){12}} \put(16,111){\line(1,-1){10}} \put(16,111){\line(1,1){10}} \put(26,121){\line(0,-1){20}} \qbezier(26,101)(38,111)(26,121) \put(16,111){\line(-1,0){9}} \qbezier(7,111)(25,145)(43,111) \qbezier(7,111)(25,77)(43,111) \put(43,111){\line(1,0){16}} \put(59,111){\line(1,-1){16}} \put(59,111){\line(1,1){16}} \put(75,95){\line(0,1){11}} \put(75,127){\line(0,-1){11}} \put(75,116){\circle*3} \put(75,106){\circle*3} \qbezier(75,127)(94,111)(75,95) 
\put(2,79){H52} \put(9,63){\line(1,0){8}} \qbezier(17,63)(28,51)(28,63) \qbezier(17,63)(28,75)(28,63) \put(9,63){\line(1,1){17}} \put(9,63){\line(1,-1){17}} \put(26,46){\line(1,1){17}} \put(26,46){\line(1,0){34}} \put(60,46){\line(0,1){34}} \put(26,80){\line(1,-1){34}} \put(26,80){\line(1,0){48}} \put(74,80){\circle*3} \put(23,29){\circle*3} \put(23,29){\line(-1,0){13}} \qbezier(10,29)(23,50)(36,29) \qbezier(10,29)(23,8)(36,29) \put(36,29){\line(1,0){13}} \qbezier(49,29)(55,43)(78,46) \qbezier(49,29)(55,15)(78,12) \put(78,46){\line(0,-1){34}} \qbezier(78,12)(94,29)(78,46) \qbezier(70,29)(59,40)(59,29) \qbezier(70,29)(59,18)(59,29) \put(70,29){\line(1,0){8}} 
\put(2,1){H53} \put(7,-16){\circle*3} \put(7,-16){\line(1,-1){16}} \put(7,-16){\line(1,1){16}} \put(23,0){\circle*3} \put(23,0){\line(0,-1){32}} \put(23,-32){\line(1,0){9}} \qbezier(32,-32)(44,-20)(44,-32) \qbezier(32,-32)(44,-44)(44,-32) {\thinlines\put(23,0){\line(1,0){15}} \put(39,0){\circle3} \put(74,-13){\circle3} \put(75,-13){\line(1,0){11}} }\qbezier(43,-13)(30,0)(30,-13) \qbezier(43,-13)(30,-26)(30,-13) \put(43,-13){\line(1,0){20}} \qbezier(63,-13)(74,7)(85,-13) \qbezier(63,-13)(74,-33)(85,-13) \put(85,-13){\circle*3} \put(53,-13){\circle*3} 
\end{picture}\end{minipage}\end{table}

\begin{table}
\begin{minipage}{300pt}
\begin{tabular}{@{}lclllccl@{}} 
\hline Id & $\!d_j\!$ &  disc\,$\QQ(j)$  & \mbox{}\hspace{-3pt}disc\,$\QQ(t)$  
 & Branching frac. & $d$ & $\![k\ell m]\!$ & Mndr./cmp. \\ 
\hline
I1 & 4 & $2^{6}3^{4}$ 
& as $j$ & $\led{\fr12,\fr13,\fr14,\fr16}$ & 9 & [246] & $G = A_{9}$ \\ 
I2 & & $2^{2}7^{3}$ 
& as $j$ & $\led{\fr13,\fr23,\fr17,\fr27}$ & 24 & [237] & $G = S_{24}$ \\ 
I3 & & 
& $2^{12}3^{2}7^{6}$ & $\led{\fr12,\fr12,\fr13,\fr13}$ & 14 & [237] & $2_H\circ \AJ{G35}$ \\ 
I4 & & $2^37^3$ 
& $2^{13}7^{6}$ & $\led{\fr12,\fr12,\fr17,\fr27}$ & 24 & [237] & $G = S_{24}$ \\ 
I5 & & & $2^{10}7^{6}23$ & $\led{\fr17,\fr17,\fr27,\fr27}$ & 48 & [237] & $2_H\circ \AJ{I4}$  \\ 
I6 & & 
& \xsix{24} $2^{37}5^{6}7^{22}$ & $\led{\fr23,\fr17,\fr17,\fr17}$ & 38 & [237] & $G = S_{38}$ \\ 
I7 & & $2^67^2$ 
& $2^{12}3\cdt5\cdt7^{5}$ & $\led{\fr12,\fr14,\fr17,\fr17}$ & 9 & [247] & $G = A_9$ \\
\hline
\end{tabular} 

\end{minipage} \hfil
\begin{minipage}{80pt}
\input{Idessins.tex}

\begin{table}
\begin{tabular}{@{}lclllccl@{}} 
\hline Id & $\!d_j\!$ &  disc\,$\QQ(j)$  & \mbox{}\hspace{-3pt}disc\,$\QQ(t)$  
& Branching frac. & $d$ & $\![k\ell m]\!$ & Mndr./cmp. \\ 
\hline
J1 & 6 & $-2^{14}3^3$ 
& $2^{37}3^{7}$ & $\led{\fr12,\fr12,\fr18,\fr18}$ & 18 & [238] & $2_H\circ \AJ{G17}$  \\  
J2 & & $-2^{11}3^65^2$ 
& as $j$ & $\led{\fr12,\fr13,\fr23,\fr18}$ & 9 & [238] & $G = S_{9}$ \\ 
J3 & & $-3^85^27$ 
& $2^{3}3^{16}5^{5}7^{3}$ & $\led{\fr13,\fr13,\fr19,\fr19}$ & 20 & [239] & $2_H\circ \AJ{J4}$  \\ 
J4 & & & $2^{10}3^{17}5^{5}7^{2}$ & $\led{\fr12,\fr12,\fr13,\fr19}$ & 10 & [239] & $G = A_{10}$ \\ 
J5 & & $-3^{3}5^{5}13$ 
& as $j$ & $\led{\fr12,\fr13,\fr15,\frc1{10}}$ & 13 & [10] & $G = A_{13}$ \\ 
J6 & & $-2^4 3^3 5\cdt7^4$ 
& $2^{12}3^{7}5^{2}7^{10}$ & $\led{\fr13,\fr23,\fr17,\fr17}$ & 30 & [237] & $G = S_{30}$ \\ 
J7 & & $-2^47^423$ 
& $2^{12}7^{8}23^{3}$ & $\led{\fr13,\fr17,\fr17,\fr27}$ & 46 & [237] & $G = S_{46}$ \\
\hline
\end{tabular} 

\end{table}

\input{Jdessins.tex}


%

\section{Appendix: Composite Belyi functions}
\label{sec:compose}


Composition of a Belyi function $\varphi(x)$ into smaller degree rational functions
can be decided from the function field lattice between $\CC(x)$ and $\CC(\varphi)$,
as described in \cite[\S ~1.7.2]{LandoZvonkin}. 
The subfield lattices are listed in our online table \cite[\sf Decomposition\_or\_GaloisGroup]{HeunURL}.

On the other hand, composite minus-4-hyperbolic Belyi functions induce 
composite hypergeometric-to-Heun transformations. 
Thereby special cases of the parametric transformations P1--P61 of \cite[\S 2.2]{HeunForm} 
and the Heun-to-Heun transformations $2_H$, $4_H$ of \cite[\S 4.3]{HeunForm}
often occur as composition parts. 
The quadratic transformation $2_H$ acts on the exponent differences
as $\heunde{1/2,1/2,\alpha,\beta} \leftarrow \heunde{\alpha,\alpha,\beta,\beta}$
and changes the $j$-invariant to a 2-isogenous $j$-invariant. 
The transformation $4_H=2_H\circ2_H$ transforms 
$\heunde{1/2,1/2,1/2,\alpha} \leftarrow \heunde{\alpha,\alpha,\alpha,\alpha}$
and does not change the $j$-invariant.  
The composite transformations could be figured out by a careful consideration of possible compositions of hypergeometric-to-hypergeometric, indecomposable hypergeometric-to-Heun 
(parametric or some newly implied), and Heun-to-Heun transformations. 
That would constitute yet another check\footnote{For example,
any transformation to Heun's equation with 2 (or 3) exponent differences equal to $1/2$
can be composed with $2_H$ (or $4_H$, respectively). 
Further, any Belyi function of the $[k\ell m]$-type $[344]$ or $[266]$ gives rise to a type-$[246]$ composition (with the degree doubled), while all $[334]$, $[248]$-type functions 
give type-$[238]$ compositions, with the degree 2 or 3 times larger. In the same way, 
the $[335]$, $[255]$ Belyi functions give type $[2\,3\,10]$, $[245]$ (respectively) 
compositions. Quadratic transformation P1 of \cite{HeunForm} 
can be composed to C1 and all compositions in Table A of \S \ref{sec:ajtables}, 
as its $j$-invariant 1728 is 2-isogenous to itself and the $j$-value of C1. }
of our list of Belyi functions.
%
The most complicated decomposition lattice is for A18:
\begin{equation} \label{pc:a18comp}
\begin{picture}(380,100)(-6,0)
\put(15,45){$[4444]$}   
  \put(43,54){\line(5,6){28}} \put(44,48){\line(1,0){23}} \put(43,40){\line(4,-5){24}} 
\put(75,90){$[488]^{\times}$:P1} \put(120,92){\line(1,0){24}} 
\put(72,45){P12:$[2244]$}  \put(121,48){\line(1,0){24}}
  \put(120,54){\line(3,4){24}}  \put(120,40){\line(4,-5){24}}  
\put(71,0){C1:$[2244]^{\times}$}  \put(119,2){\line(1,0){22}} 
\put(148,90){$[288]$:P2}   \put(187,88){\line(6,-5){36}} 
\put(148,45){$[444]$:P9}   \put(186,54){\line(4,3){32}}  \put(187,48){\line(1,0){33}}
\put(146,0){P4:$[2224]$}   \put(189,8){\line(1,1){33}} 
\put(222,80){$[248]^{\times}$:A16}
\put(225,45){$[248]$:P10} 
\put(223,10){$[334]$:A17}  \put(267,19){\line(4,3){28}}  
\put(297,45){$[238]$}  
\thicklines \put(269,48){\line(1,0){24}}
\put(186,41){\line(3,-2){33}}   \put(272,78){\line(6,-5){24}} 
\end{picture}
\end{equation}
In the square brackets, we see the $[k\ell m]$ triples of intermediate hypergeometric equations,
or similar indication of intermediate Heun equations. 
The transformation from $[238]$ to the Heun equations 
is indicated before their square brackets.
Similarly, the $[k\ell m]$ triples are followed by the indication of
transformations from them to the final $[4444]$. The $\times$-power indicates two copies of
that intermediate function field.
The diagram includes P10 and P12, the most complicated
parametric compositions \cite[\S C]{HeunForm}. 
The components $[238]\!-\![248]$ and $[334]\!-\![444]$ are cubic transformations, 
while the other lines represent quadratic ones (possibly $2_H$).

In Tables of \S \ref{sec:ajtables}, we indicate the components either by an A-J label 
from our list (if applicable), or by the degree otherwise. In the latter case, we give intermediate
hypergeometric equations in the $[k\ell m]$ notation. Intermediate Heun equations are clear,
hence no extras to $2_H$. Deeper branching is indicated by $\{\}$. 
The A-J, P labels inside them either mean a transformation from a starting $[k\ell m]$
to an intermediate Heun equation (after $4_H$) or to the target Heun equation (otherwise).
The $\times\times$-power indicates three copies of an intermediate function field.
A label inside nested $\{\}$ refers to a composition string avoiding the merging point of the outer $\{\}$.
These hints should be enough to recover the composition lattices. 

\section{Appendix: Coxeter decompositions}
\label{sec:coxeter}

If a minus-4-hyperbolic Belyi function in a canonical form (of Definition \ref{df:jtfield})
is defined over $\RR$, 
the Schwarz maps\footnote{We already considered Schwarz maps in the paragraph 
after Remark \ref{rm:noexist}. If a hypergeometric equation has real local exponent differences
$\alpha,\beta,\gamma$ in the interval $[0,1]$, the image of the upper half plane $\subset\CC$
under its Schwarz map is a curvilinear triangle with the angles $\pi\alpha,\pi\beta,\pi\gamma$. 
A nice illustration can be found in \cite[pg. 38]{beukers}.
Analytic continuation of Schwarz maps follows the Schwarz reflection principle.
Hodgkinson \cite{hodgkins1} first observed that pull-back transformations 
of hypergeometric equations induce tessellations of Schwarz triangles 
into smaller congruent Schwarz triangles.
Similarly, if a Heun equation has real local exponent differences
$\alpha,\beta,\gamma,\delta$ in the interval $[0,1]$, 
the image of its Schwarz map is a curvilinear quadrangle
with the angles $\pi\alpha,\pi\beta,\pi\gamma,\pi\delta$. }
of the related hypergeometric and Heun equations fit together nicely.
Particularly, the quadrangle of Heun's equation is then tessellated into
congruent (in the hyperbolic metric) triangles of the hypergeometric equation. 
The degree formula in Lemma \ref{th:klm}\refpart{ii} can be interpreted as
the area ratio between the hyperbolic quadrangle and the triangles, if we multiply both the
numerator and the denominator by $\pi$. 
Subdivisions of hyperbolic quadrangles (or triangles) into congruent hyperbolic triangles
are called {\em Coxeter decompositions} in \cite{Felixon}. The list of Coxeter decompositions
can be compared with our list of Belyi maps with the $r$-field $\subset\RR$, 
providing a mutual check of completeness.  

The Belyi functions of Tables D, E, G  (of Appendix \ref{sec:ajtables}) 
give no Coxeter decompositions, as their $r$-fields certainly have no real embeddings. 
The obstructed Belyi functions of \S \ref{sec:conic} give no Coxeter decompositions either
(except F7 of \S \ref{sec:f7}).  Here is the count of Coxeter decompositions
induced by our Belyi functions: 
\begin{itemize} \itemseparate
\item Table A gives 10 Coxeter decompositions. The last column shows that 
the other 14 Belyi functions have imaginary quadratic $r$-fields.
\item Tables B, C give $23+34$ decompositions. 
The cases\footnote{Details of the $r$-extensions can be found in 
\cite[\sf j\_t\_and\_r\_Field\_MinPoly]{HeunURL}. The list of cases with additional extensions
for the $r$-field correlates well with the list of  Belyi coverings with interesting monodromy groups 
(such as PSL in tables of \S \ref{sec:ajtables}) and the list of multiple Galois orbits with the same branching pattern (as one can inspect empty entries in the first columns in tables of \S \ref{sec:ajnumbers}).}
with an imaginary quadratic extension $\QQ(t)\supset \QQ(j)$ are
B2, B6, B9, B10, B12, B18, B19, B21, B22, B27, B28 and
C2, C3, C6, C11, C22, C24, C30, C31.
\item Each entry of the F-table with discrim $\QQ(t)<0$ gives one 
Coxeter decomposition; 10 in total.
\item The entries F3, F23 with $\QQ(t)=\QQ(j)$ give pairs of Coxeter decompositions.
F20 gives another pair with the $t$-field $\QQ(\sqrt7,\sqrt3)$,
but F25 gives none with the $t$-field $\QQ(\sqrt{4\sqrt{22}-22})$.
\item Each entry of the H-table with discrim $\QQ(j)<0$ and either $\QQ(t)=\QQ(j)$ 
or discrim $\QQ(t)>0$ gives a Coxeter decomposition; 11+11 in total.
\item Similarly, the odd degree I, J-orbits with $\QQ(t)=\QQ(j)$ or discrim $\QQ(t)>0$
give single Coxeter decompositions; $6+3$ among I22--I33 and $2+3$ in the J-table.
\item 
 I10, I15, J11 have pairs of real dessins and $\QQ(t)=\QQ(j)$.
They give pairs of Coxeter decompositions.
\end{itemize}
In total, we have 125 
decompositions, just as listed in 
\cite[Figures 10 (5)--(11), 12, 13, 15--18]{Felixon}.
There is a caveat, however. The decompositions 24 and 36 in \cite[Figure 18]{Felixon} coincide,
while one triangulated quadrangle with the angles $\pi/3,2\pi/3,\pi/7,3\pi/7$ is missing.
We identify the repeated decomposition as C4, and the missing one as I26. 
All Coxeter decompositions from our Belyi 
functions can be discerned in Figure \ref{fg:coxeter}. The similar pictures for Coxeter decompositions
from parametric hypergeometric-to-Heun transformations are given in \cite[Figure 2]{HeunClass}.

\begin{vchfigure}
\begin{picture}(460,390)(30,219)
\put(0,190){\includegraphics[height=250pt]{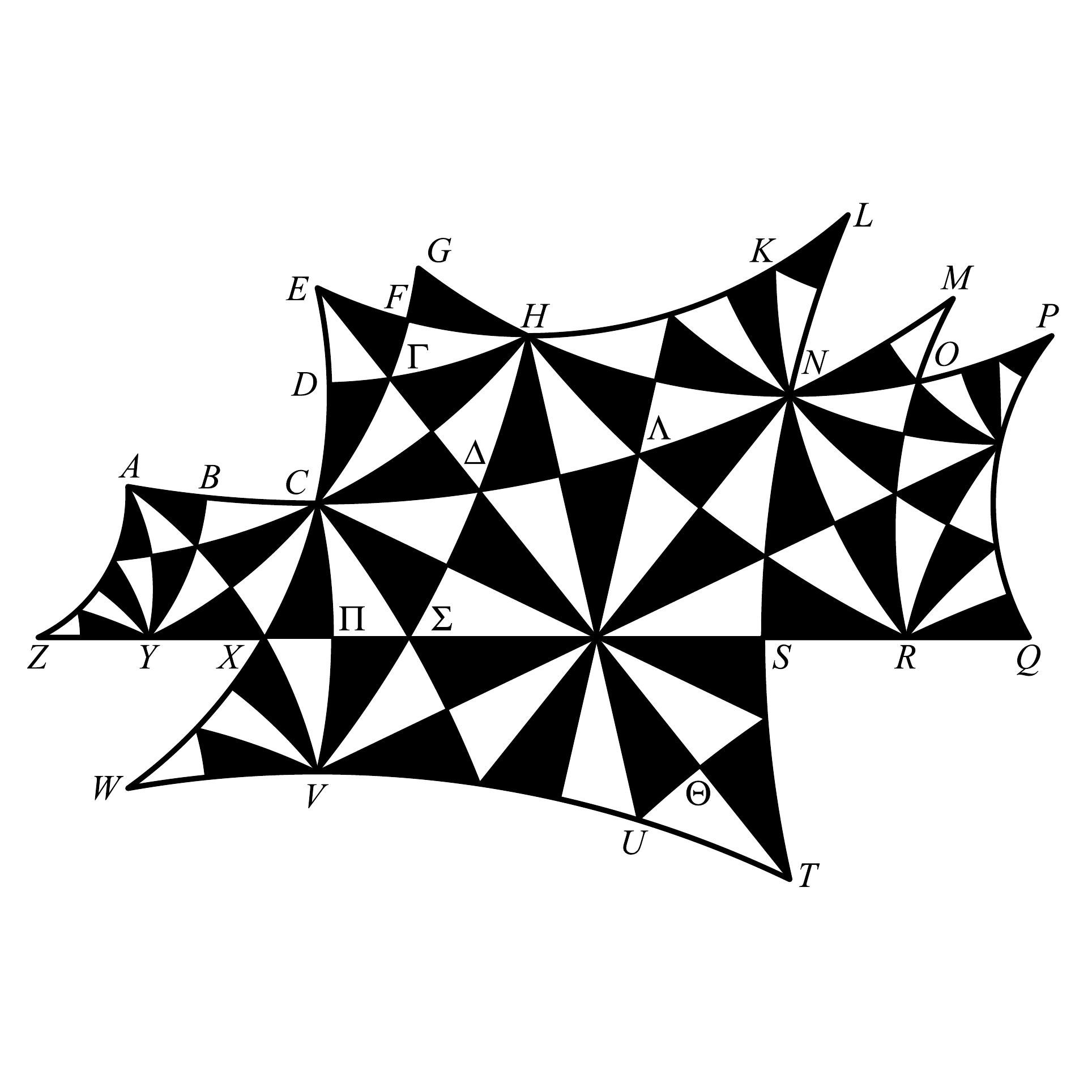}}
\put(285,357){\includegraphics[height=162pt]{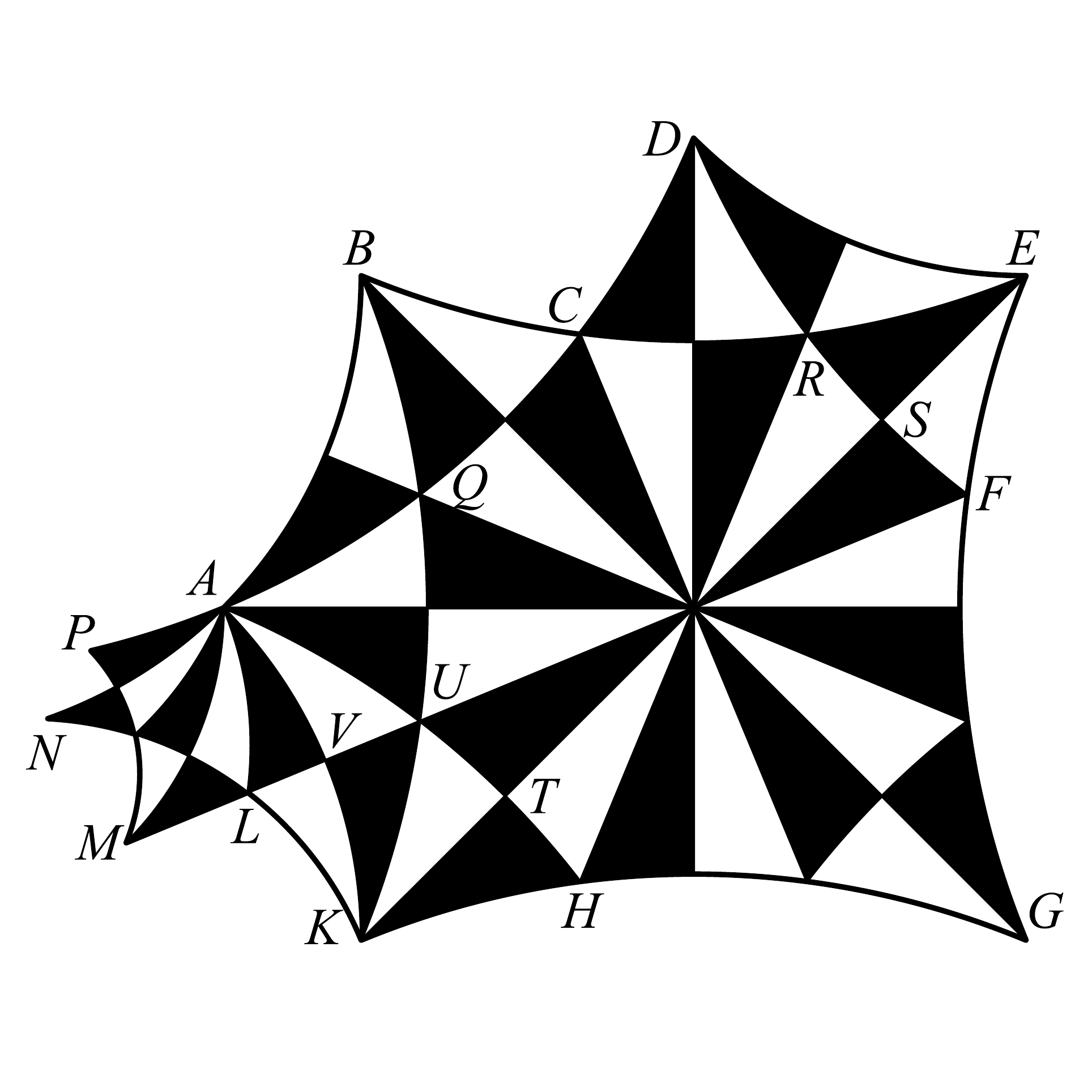}}
\put(131,395){\includegraphics[height=154pt]{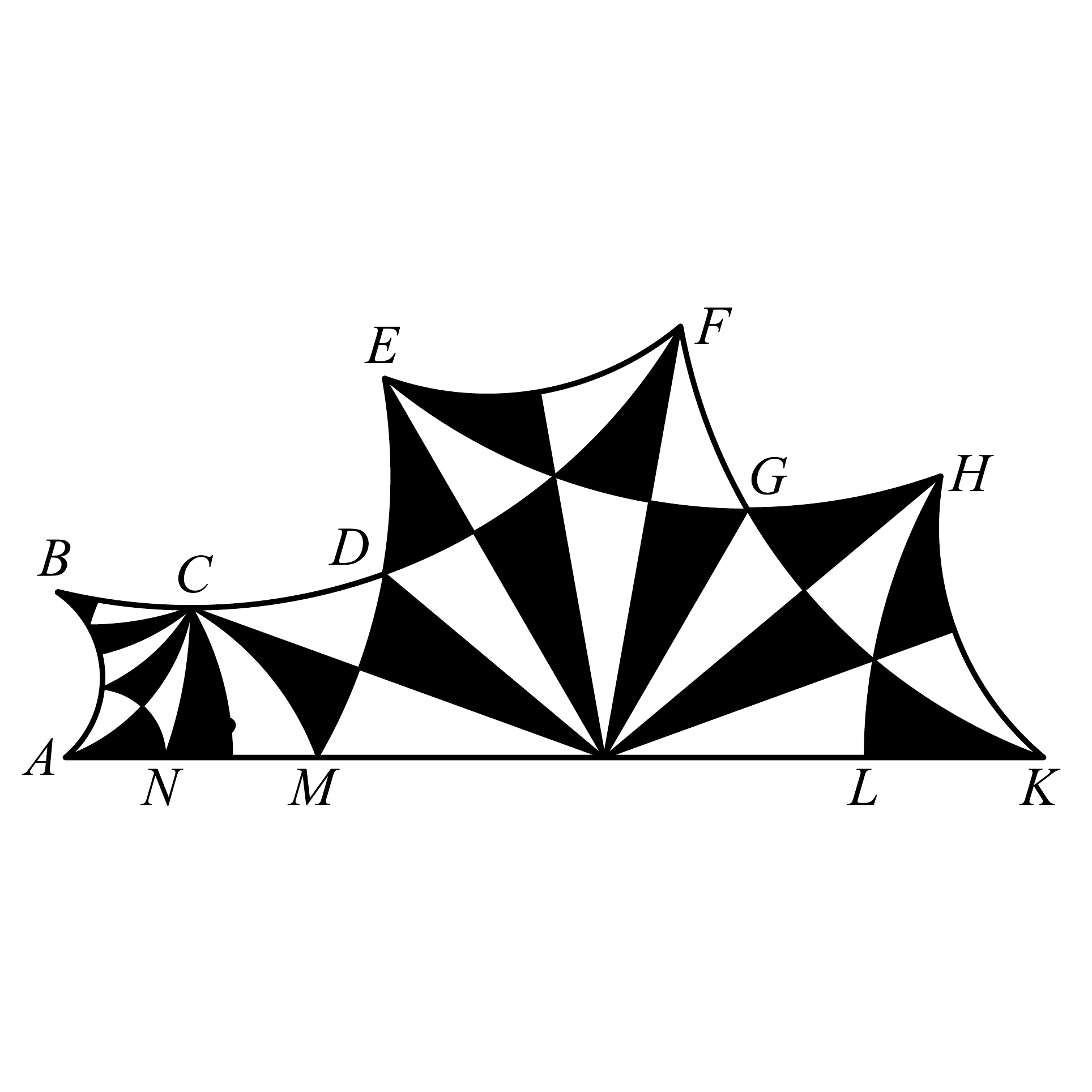}}
\put(0,510){\includegraphics[height=112pt]{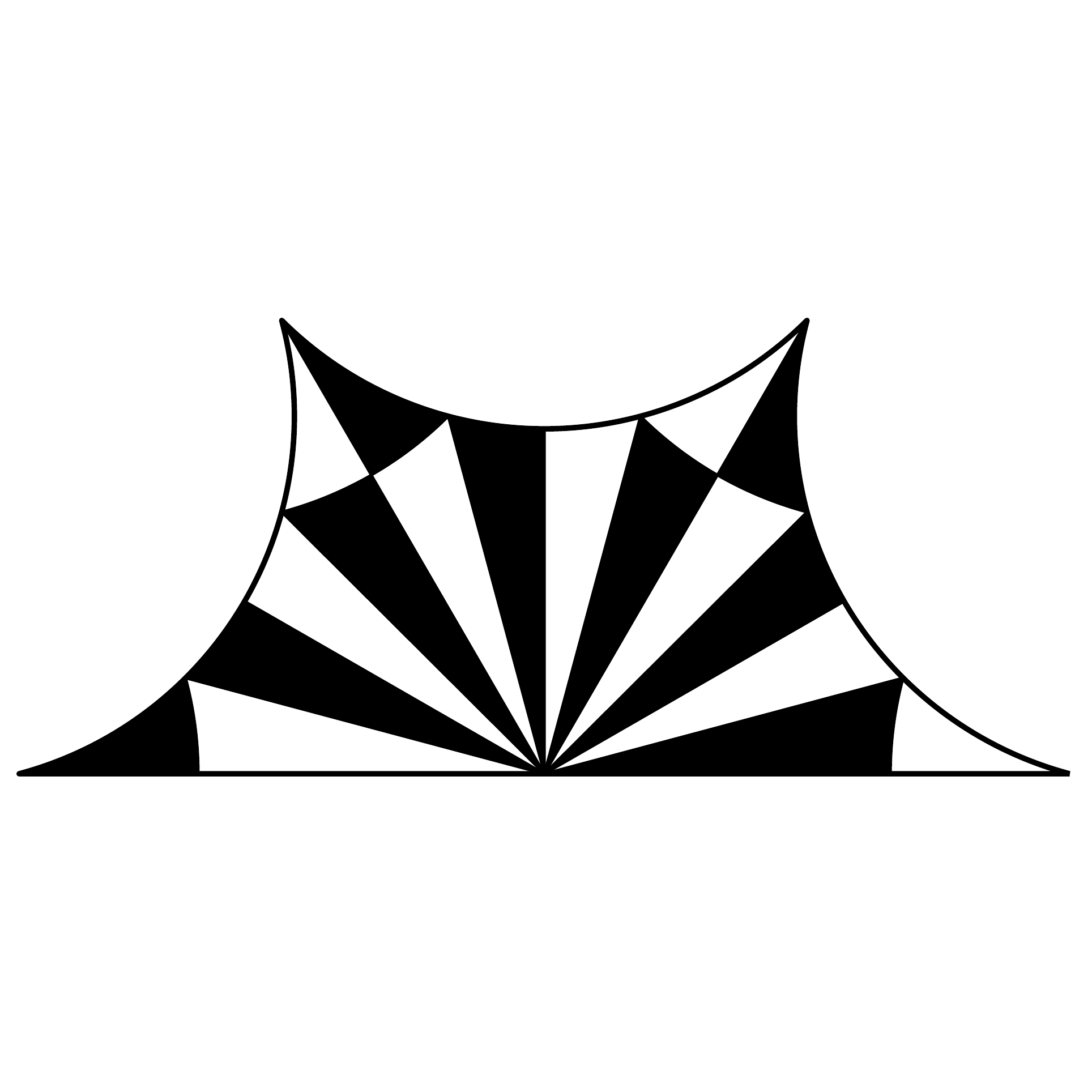}}
\put(137,519){\includegraphics[height=100pt]{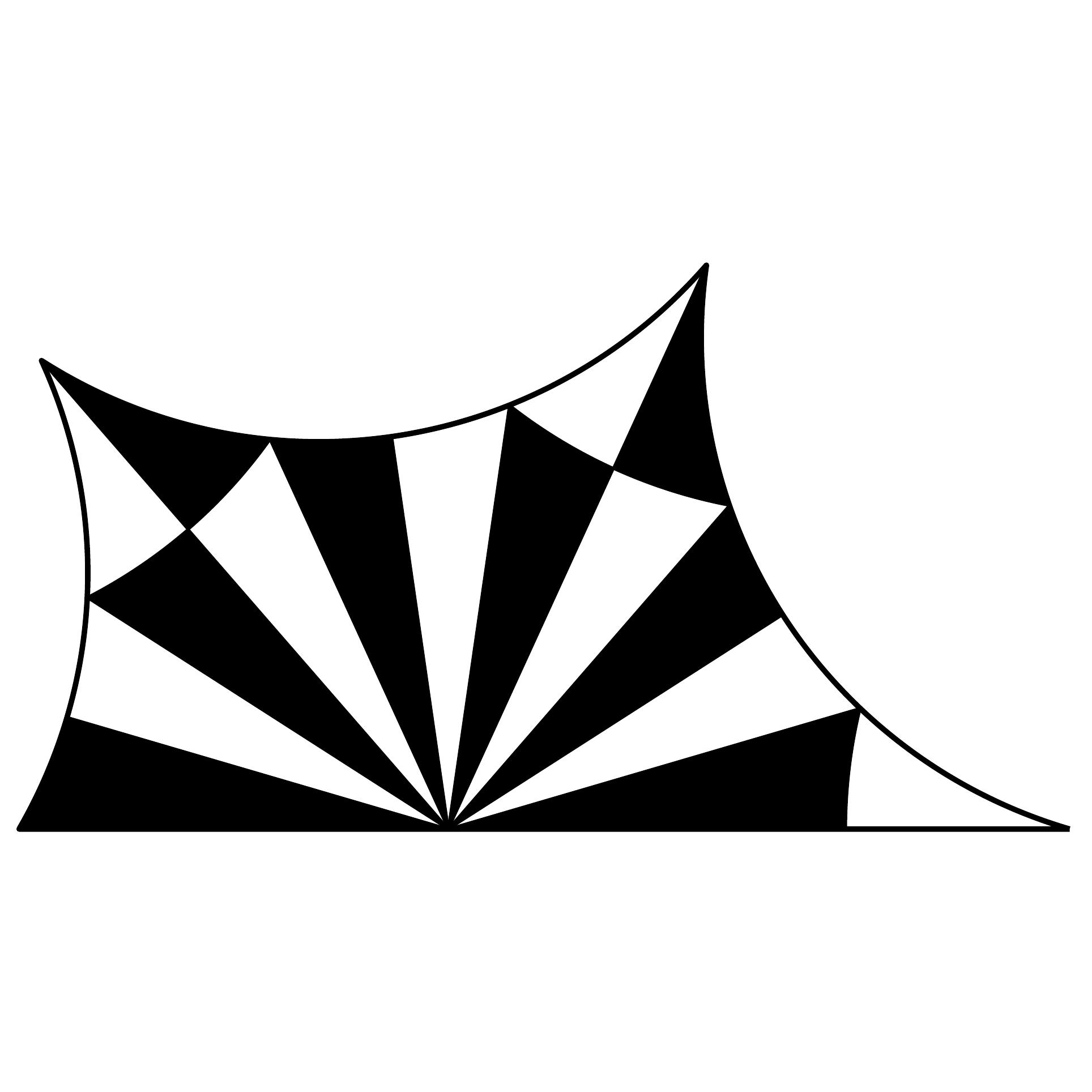}}
\put(261,497){\includegraphics[height=106pt]{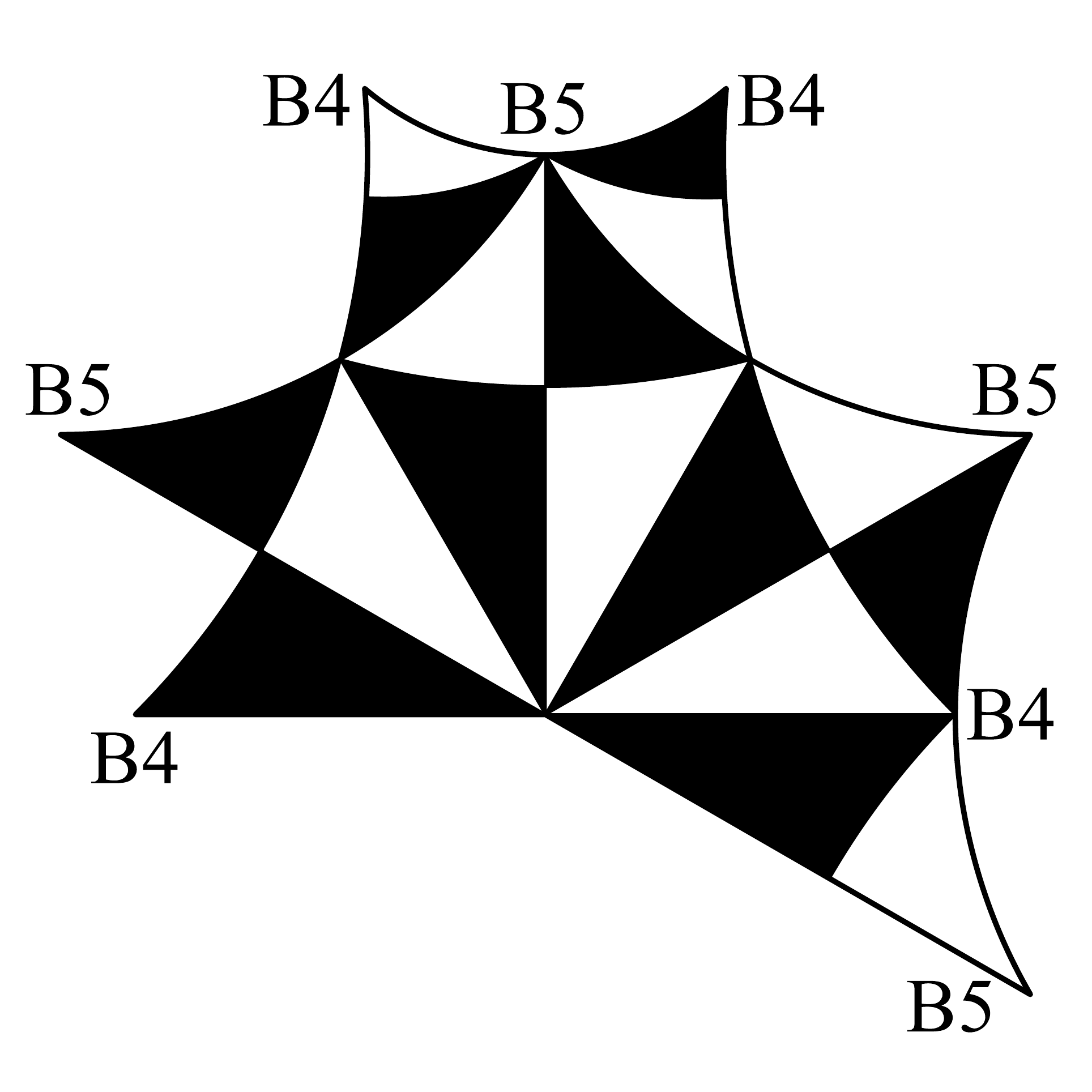}}
\put(-2,401){\includegraphics[height=120pt]{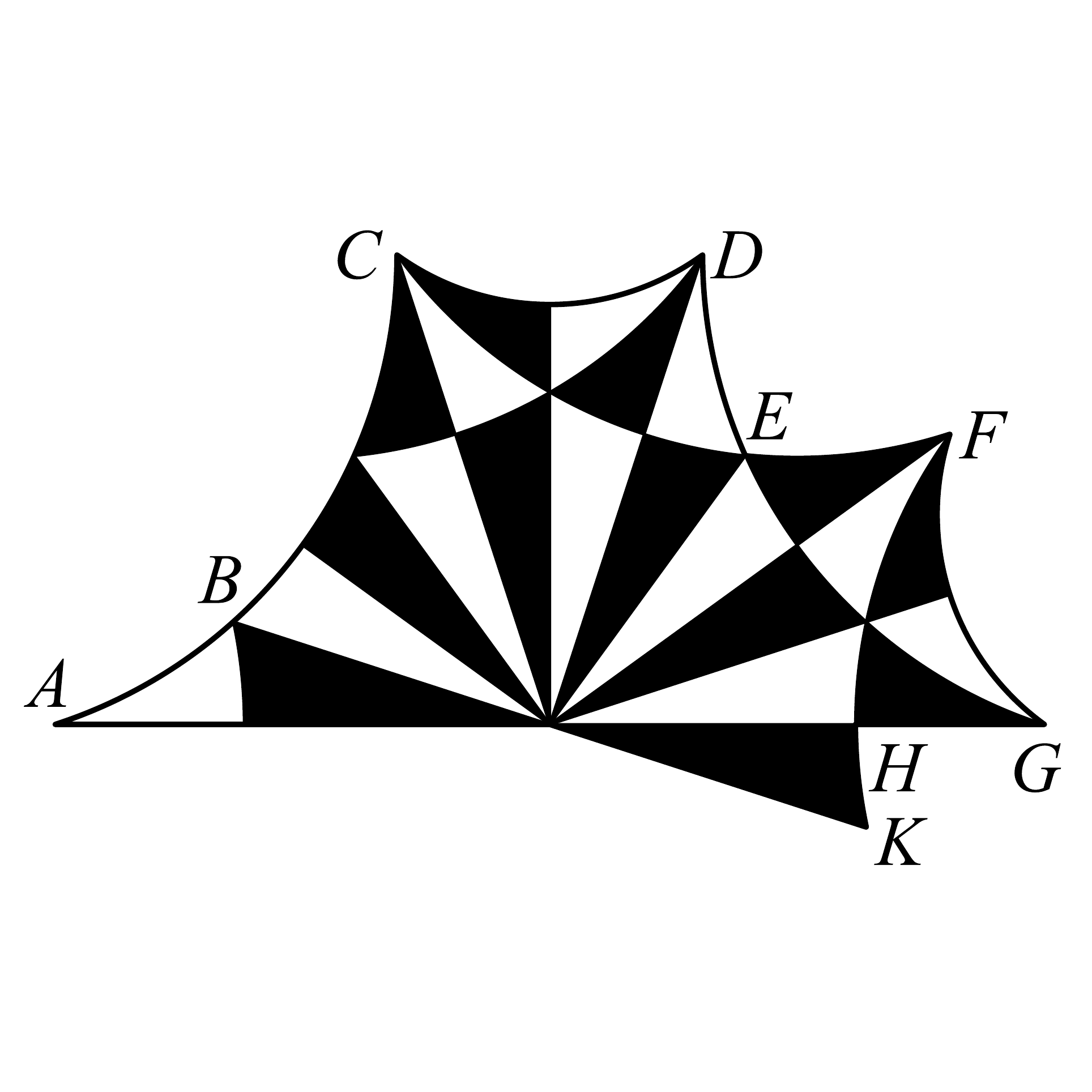}}
\put(275,211){\includegraphics[height=165pt]{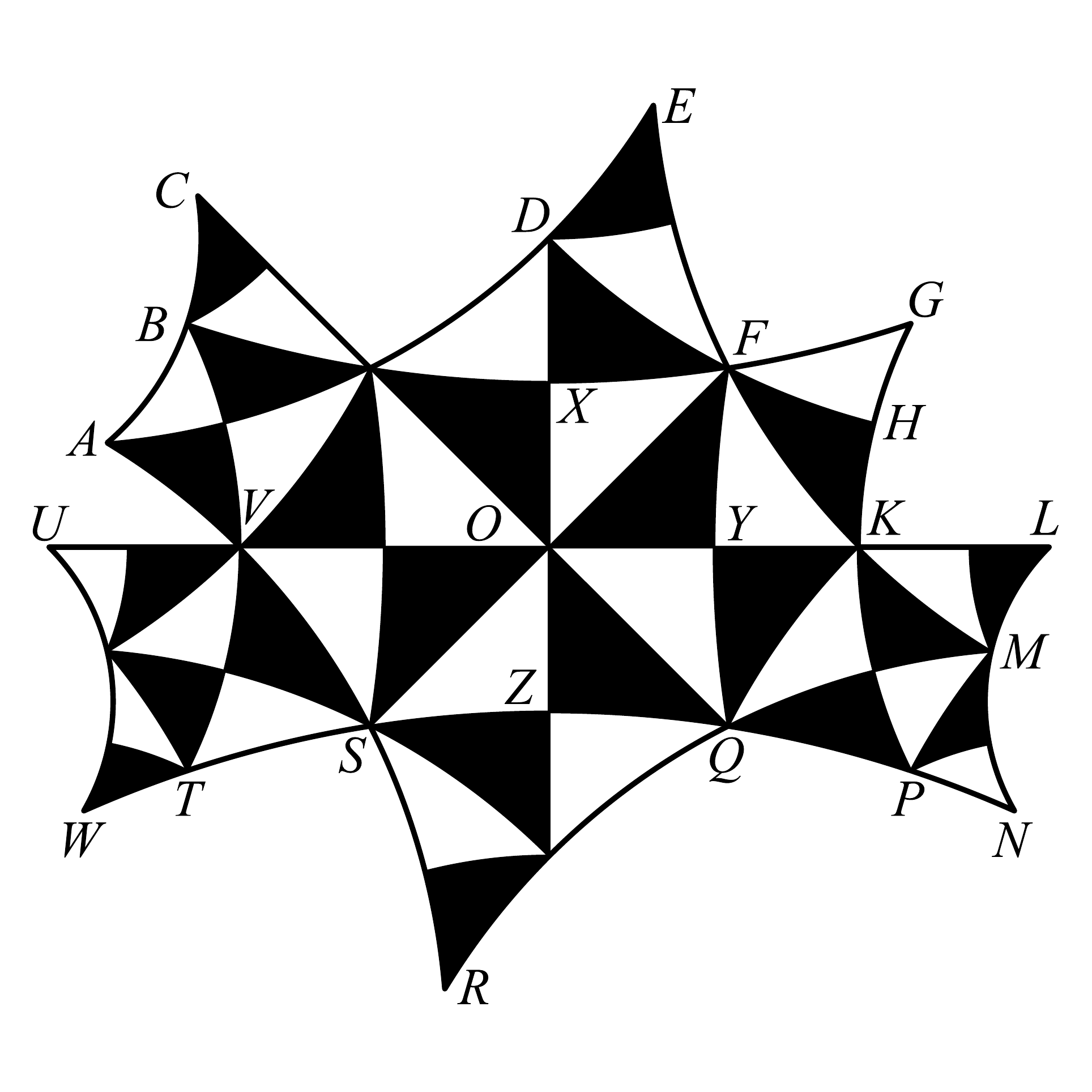}}
\put(2,590){\refpart{a}} \put(123,590){\refpart{b}} \put(252,590){\refpart{c}}
\put(2,495){\refpart{d}} \put(139,495){\refpart{e}} 
\put(301,485){\refpart{f}} \put(10,375){\refpart{g}}  \put(273,351){\refpart{h}} 
\end{picture}
\caption{The Coxeter decompositions of Felikson}
\label{fg:coxeter}
\end{vchfigure}

Belyi functions (with the $r$-field in $\RR$) and Coxeter decompositions
are identified\footnote{Dessins d'enfant and Coxeter decompositions are different geometric representations 
of a Belyi covering. The difference is twofold: the decompositions represent only a half of the Riemann sphere, and their vertices are the points not just above $z\in\{0,1\}$ but above $z=\infty$ as well.
To get a corresponding (real) dessin, two parallel copies of a Coxeter decomposition have 
to be glued along the edges to a topological sphere, and the vertices above $z=\infty$ 
with the incident edges, triangles have to be removed. }
by multiplying the branching fractions by $\pi$ and looking for quadrangles in \cite{Felixon}
with the same angles.
Pictures \refpart{a}, \refpart{b} in Figure \ref{fg:coxeter} 
show  the Coxeter decompositions 7, 6 
in \cite[Figure 15]{Felixon}. 
They represent the Belyi functions B11 and C13, respectively.
Picture \refpart{c} contains two hyperbolic quadrangles 
subdivided into twelve $(\pi/2,\pi/4,\pi/6)$-triangles.  
They represent the Belyi functions B4 and B5, and coincide with the
triangulations 4, 3 in \cite[Figure 12]{Felixon}, respectively.
Picture \refpart{d} contains the first five triangulations in \cite[Figure 15]{Felixon},
into $(\pi/2,\pi/3,\pi/10)$-triangles.
Here are the labels of Belyi maps and the quadrangles, 
in the same sequence as in \cite{Felixon}:\\[2pt]
\begin{tabular}{r@{: }lr@{: }lr@{: }lr@{: }lr@{: }l}
\swq{BCFK}{C38}, & \swq{ACEG}{F14}, & \swq{ACFH}{C15}, 
 & \swq{ACDG}{C25}, & \swq{ACFG}{B24}.
\end{tabular} \\[2pt]
Picture \refpart{e} contains the 10 $(\pi/2,\pi/3,\pi/9)$-triangulations in
\cite[Figure 16]{Felixon}:\\[2pt] 
\begin{tabular}{r@{: }lr@{: }lr@{: }lr@{: }lr@{: }lr@{: }l}
\swq{DFKM}{H2}, & \swq{EGKM}{B7}, & \swq{EHLM}{C41}, 
 & \swq{EHKM}{C7}, & \swq{CFKM}{C9}, \\ 
\swq{CFKN}{H4}, & \swq{CFKP}{H6}, & \swq{EFKM}{A6},  & \swq{ACFK}{C23}, & \swq{ABFK}{H8}.
\end{tabular} \\[2pt]
There is initial ambiguity for assigning B7 and H2 because of the same branching fractions.
But B7 is a composition $3\,[339]\circ4$ as shown in the B-table,
and its Coxeter decomposition splits\footnote{Coxeter decompositions
do not always split according to (all) compositions of their Belyi functions, because smaller degree components do not necessarily have Coxeter decompositions. 
For example, consider $\AJ{A18}=\AJ{A16}\circ3$, 
$\AJ{A19}=2_H\circ\AJ{A1}$, $\AJ{B4}=2_H\circ\AJ{D25}$, $\AJ{B14}=2_H\circ\AJ{D9}$, 
$\AJ{J19}=2_H\circ\AJ{H45}$, $\AJ{J26}=2_H\circ\AJ{J25}$, etc.}
into 3 triangles with the angles $\pi/3,\pi/3,\pi/9$ (each formed by 4 smaller triangles). 
Picture \refpart{f} contains the 19 $(\pi/2,\pi/3,\pi/8)$-triangulations in \cite[Figure 17]{Felixon}:\\[2pt] 
\begin{tabular}{r@{: }lr@{: }lr@{: }lr@{: }lr@{: }lr@{: }lr@{: }l}
\swq{ABEK}{B1}, & \swq{BEGK}{A18}, & \swq{BEKN}{B14}, 
 & \swq{BGKN}{A5}, & \swq{ADST}{C1}, \\
\swq{FGKS}{I10}, & \swq{ACET}{I10}, & \swq{CEKQ}{C32}, 
& \swq{DFUQ}{F19}, & \swq{ADFU}{C21}, \\
\swq{ADFV}{H29}, & \swq{ABRH}{C28},  & \swq{ACEK}{B29}, 
& \swq{ADSK}{B15}, & \swq{ADFL}{F8}, \\
\swq{ADFM}{B30}, & \swq{ABET}{C8}, & \swq{PDFM}{H22}, 
& \swq{ADEK}{C5}.
\end{tabular} \\[2pt]
Picture \refpart{g} contains the 58 $(\pi/2,\pi/3,\pi/7)$-triangulations in
\cite[Figure 18]{Felixon}. Here is the respective sequence of Belyi maps and the quadrangles,
with the repeated decomposition 36 replaced by the one for I26: \\[2pt]
\begin{tabular}{r@{: }lr@{: }lr@{: }lr@{: }lr@{: }lr@{: }lr@{: }l}
\swq{CHR$\Sigma$}{C34}, & \swq{HMR$\Sigma$}{F3}, & \swq{MTV$\Delta$}{A15},
& \swq{HKR$\Sigma$}{A11}, & \swq{HOR$\Sigma$}{H33},  & \swq{A$\Delta\!$TV}{A14}, \\
\swq{ELTW}{A21}, & \swq{R$\Sigma\Delta\Lambda\,$}{C29}, & \swq{MS$\Sigma\Delta$}{I15}, 
& \swq{CMS$\,\Sigma$}{C42}, & \swq{CMS$\,\Pi$}{F22}, & \swq{UV$\Delta\Theta$}{I15}, \\
\swq{CHR$\,\Pi$}{B33},  & \swq{C$\Lambda$R$\Sigma$}{C40}, & \swq{C$\Lambda$R$\,\Pi$}{H53},
& \swq{MR$\Sigma\Delta$}{B20}, & \swq{CHRX}{B17}, & \swq{C$\Lambda$RX\,}{B25}, \\
\swq{CMSX}{H50}, & \swq{HMS$\Sigma$}{C33}, & \swq{CMR$\Sigma$}{C26}, 
& \swq{CMR$\Pi$}{C35}, & \swq{C$\Lambda$RY}{C27},  & \swq{CHRY}{C4}, \\
\swq{A$\Lambda$RX}{H48}, & \swq{CMRX}{F3}, & \swq{CMSY}{H40}, 
& \swq{AMSX}{I33},  & \swq{BMSY}{H37}, & \swq{HRY$\,\Gamma$}{B32}, \\ 
\swq{CMRY}{B31}, & \swq{EHRX}{B34}, & \swq{AMRX}{F23}, 
& \swq{BMRY}{C39}, & \swq{AMRY}{F20}, & \swq{NR$\Sigma\Delta$}{I26}, \\ 
\swq{CNR$\Sigma$}{F23}, & \swq{CNR$\Pi$}{C36}, & \swq{CNRX}{H52}, 
& \swq{CNRY}{F24}, & \swq{CMTV}{F20}, & \swq{CMTW}{B13}, \\
\swq{EKRX}{I23}, & \swq{DHRX}{J11},  & \swq{HLS$\,\Sigma$}{J11}, 
& \swq{EHRX}{H44}, & \swq{BNRY}{I25}, & \swq{FKRY}{J24}, \\
\swq{FLSY}{J19},  & \swq{GHRY}{H36}, & \swq{ANRX}{C12}, 
& \swq{GMSY}{J26}, & \swq{GMRY}{B16},  & \swq{ANRY}{C18}, \\  
\swq{E$\Theta\!$UW}{J17}, & \swq{GPQY}{J23}, & \swq{ANRZ}{I28}, & \swq{GORY}{F21}.
\end{tabular} \\[2pt]
The ambiguity between A14 and B32 (due to the same branching fractions)
is resolved by  the reflection symmetry of $\AJ{A14}=2\circ10$. 
Picture \refpart{h} contains the 20 $(\pi/2,\pi/4,\pi/5)$-triangulations in
\cite[Figure 13]{Felixon}\\[2pt] 
\begin{tabular}{r@{: }lr@{: }lr@{: }lr@{: }lr@{: }lr@{: }l}
\swq{FKPS}{B26}, & \swq{VYQT}{H39}, & \swq{VKQT}{C19}, 
& \swq{ACOS}{C16}, & \swq{ACQS}{B23}, \\
\swq{ACQR}{H16}, & \swq{FHPS}{I24}, & \swq{DHPZ}{H17},  
& \swq{VKPT}{F13},  & \swq{BXZT}{C14}, \\
\swq{BFST}{B8}, & \swq{BFQT}{H41}, & \swq{OLNZ}{I29}, 
& \swq{OLNS}{F15}, & \swq{VLNS}{C10}, \\
\swq{WUYQ}{I31}, & \swq{WUKP}{H18}, & \swq{BGPT}{C17}, 
& \swq{AEMR}{A20}, & \swq{WULN}{I30}.
\end{tabular} \\[2pt]
The is ambiguity between C14 and H17 is resolved by the composition $\AJ{C17}=2_H\circ\AJ{C14}$. 
The non-parametric decompositions (5)--(11) of \cite[Figure 10]{Felixon}
and the decompositions (1), (2) of \cite[Figure 12]{Felixon}
represent the Galois orbits  F18, B3, C20, C37, F7, H21, A17, H15, A19, respectively.
They can be obtained from our listed quadrangles of (respectively)
F19, B4, C21, C38, F8, H22, A18, H16, A20
by pairing their triangles to larger triangles with the requisite angles $(\pi/3,\pi/3,\pi/4)$, 
$(\pi/3,\pi/4,\pi/4)$, $(\pi/3,\pi/3,\pi/5)$ or $(\pi/2,\pi/5,\pi/5)$.

\section{Appendix: Arithmetic observations}
\label{sec:arithmetic}

As observed in \cite[\S 2.3]{HeunForm}, the $t$-parameters of Heun equations
reducible to hypergeometric equations by a pul-back transformation are arithmetically interesting.
The whole orbit (\ref{S3action}) of $t$-values can be encoded by an arithmetic identity 
$A+B=C$ with algebraic integers $A,B,C$ (as ``co-prime" as possible),
as the set $\{A/C,B/C,C/A,C/B,-A/B,-B/A\}$. Here are these identities for a few $t$-orbits in $\QQ$:
\begin{align*}
& \AJ{B25}: 1+2\!\cdot\!11^2 =3^5,
&& \AJ{B29}: 2^2+11^2=5^3,
&& \AJ{B30}: 1+3^35^2=2^213^2, \\
& \AJ{B31}: 1+2^5\,3\!\cdot\!5^2=7^4,
&& \AJ{B33}: 11^3+2^27^4=3^{7}\,5, 
&& \AJ{B34}: 7^4+3^35^3=2^{4}19^2.
\end{align*}
The terms in these identities involve only small primes, usually in some power. 
Correspondingly, the $t$-values factorize nicely in $\QQ$.
These identities are interesting in the context of the ABC conjecture \cite{Wikipedia}
and $S$-unit equations \cite{Wikipedia}. 
The ``factorization" pattern holds for the $t$-values in algebraic extensions of $\QQ$ as well, 
though arithmetic quality is then measured more technically \cite{nitaj} by the prime places 
and arithmetic height in $\PP^2(\overline{\QQ})$. The underlying reason is that the Belyi coverings
(of pull-back transformations) tend to degenerate only modulo a few small primes \cite{beckmann}.
Hence the $t$-orbit (\ref{S3action}) degenerates only modulo those bad primes.

Amidst the encountered examples, we find the following well-known identities $A+B=C$
in quadratic fields:
\begin{align*} \textstyle
\AJ{C18}:  \left(\frac{\sqrt5-1}2\right)^{12} \! +2^43^2\,\sqrt5=\left(\frac{\sqrt5+1}2\right)^{12} \! ,  \qquad
\AJ{D37/D39}:  \left(\frac{1+\sqrt{-7}}2\right)^{13} \! +\sqrt{-7}=\left(\frac{1-\sqrt{-7}}2\right)^{13} \! .
\end{align*}
They are among top 12 known examples of remarkable ABC identities \cite{nitaj}
in algebraic number fields. Their 
ABC-quality is $\approx1.697794$, 
$1.707222$, 
respectively, while Nitaj's table \cite{nitaj} includes examples with the quality $>1.5$.  
The Belyi function D42 gives a new example in $\QQ(\sqrt{-14})$ 
with the quality $\log(3^{13}5^3)/\log(56\cdot2\cdot7\cdot3^2\cdot5^2)\approx 1.581910$.
However, the class number of $\QQ(\sqrt{-14})$ is equal to 4,
hence an explicit arithmetic identity is less impressive, without 13th powers:
\begin{align}
\left(5-2\sqrt{-14}\right) \left(11+\sqrt{-14}\right)^3+\left(\sqrt{-14}\right)^3=
 \left(5+2\sqrt{-14}\right) \left(11-\sqrt{-14}\right)^3.
\end{align}
Less symmetric quadratic identities arrise from the F, G-cases with $\QQ(t)=\QQ(j)$.
For example, G30 gives
\begin{equation} \textstyle
\left(\frac{1+\sqrt{-7}}2\right)^{10}+\left(\frac{1-\sqrt{-7}}2\right)^{5}+\left(2+\sqrt{-7}\right)^3=0.
\end{equation}
The Belyi coverings E10/E11give the following $A+B=C$ example in a number field of degree 6.
Let $\zeta$ denote a root of $z^6+4z^4-3z^2+2$. Then
\begin{align*}
\zeta^{23}+
\left(\frac{\zeta+\zeta^2}2-\frac{5\zeta^3+\zeta^5}4\right)^{23}
\left(\frac{1-\zeta}2-\frac{3\zeta^2-3\zeta^3+\zeta^4-\zeta^5}4\right)^{-6} 
  \hspace{50pt} \nonumber\\
=\left(\frac{-\zeta+\zeta^2}2+\frac{5\zeta^3+\zeta^5}4\right)^{23}
\left(\frac{1+\zeta}2-\frac{3\zeta^2+3\zeta^3+\zeta^4+\zeta^5}4\right)^{-6}.
\end{align*}
The numbers under the 23rd power have the norm 2,
while the numbers in the (-6)th power are units.

\small

\bibliographystyle{plain}
\bibliography{AGSF}

\end{document}